\documentclass[12pt,a4paper]{amsart}
\usepackage[english]{babel}
\usepackage{amsmath,amssymb,amsthm,latexsym,MnSymbol,shuffle,comment}
\usepackage[utf8]{inputenc}
\usepackage[usenames]{color}
\usepackage{tikz-cd}
\usepackage[a4paper, portrait]{geometry}
\usepackage{todonotes}
\usepackage[colorlinks=true]{hyperref}

\newcommand{\Z}{\mathbf{Z}}

\newcommand{\Q}{\mathbf{Q}}

\newcommand{\R}{\mathbf{R}}
\newcommand{\C}{\mathbf{C}}
\newcommand{\h}{\mathcal{H}}

\newcommand{\p}{\mathbf{P}}
\newcommand{\un}{\mathbf{1}}
\newcommand{\imaxis}{\left]0,i\infty\right[}
\newcommand{\bolda}{{\boldsymbol a}}
\newcommand{\boldb}{{\boldsymbol b}}
\newcommand{\boldc}{{\boldsymbol c}}
\newcommand{\boldd}{{\boldsymbol d}}
\newcommand{\boldu}{{\boldsymbol u}}
\newcommand{\boldv}{{\boldsymbol v}}
\newcommand{\boldx}{{\boldsymbol x}}
\newcommand{\boldy}{{\boldsymbol y}}
\newcommand{\boldz}{{\boldsymbol z}}
\newcommand{\Sha}{\shuffle}

\DeclareMathOperator{\dlog}{dlog}
\DeclareMathOperator{\darg}{darg}
\DeclareMathOperator{\Eis}{Eis}

\DeclareMathOperator{\im}{Im}

\DeclareMathOperator{\ord}{ord}

\DeclareMathOperator{\re}{Re}

\DeclareMathOperator{\SL}{SL}

\newtheorem{thm}{Theorem}
\newtheorem*{thm*}{Theorem}
\newtheorem{lem}[thm]{Lemma}
\newtheorem{pro}[thm]{Proposition}
\newtheorem{cor}[thm]{Corollary}
\newtheorem*{cor*}{Corollary}

\theoremstyle{definition}
\newtheorem{definition}[thm]{Definition}

\theoremstyle{remark}
\newtheorem{remark}[thm]{\bf Remark}
\newtheorem{example}[thm]{\bf Example}
\newtheorem{question}[thm]{\bf Question}

\begin{document}

\author[F. Brunault and W. Zudilin]{Fran\c cois Brunault and Wadim Zudilin}


\date{\today}

\address{\'ENS Lyon, Unit\'e de math\'ematiques pures et appliqu\'ees, 46 all\'ee d'Italie, 69007 Lyon, France}

\email{francois.brunault@ens-lyon.fr}
\urladdr{https://perso.ens-lyon.fr/francois.brunault}

\address{IMAPP, Radboud University Nijmegen, PO Box 9010,
6500\,GL Nijmegen, The Netherlands}

\email{w.zudilin@math.ru.nl}
\urladdr{https://www.math.ru.nl/~wzudilin/}

\title[Modular regulators and multiple Eisenstein values]{Modular regulators and multiple Eisenstein values}

\dedicatory{In memoriam: Professor Yuri Ivanovich Manin}

\subjclass[2020]{Primary 19F27; Secondary 11F67, 11G16, 11G55}
\keywords{Regulators; modular units; elliptic curves; modular curves; $L$-functions}

\thanks{The first author was supported by the research project ``Motivic homotopy, quadratic invariants and diagonal classes'' (ANR-21-CE40-0015) operated by the French National Research Agency (ANR)}

\begin{abstract}
We introduce a new methodology for length reduction of multiple modular values as developed by Brown; it involves an interpolation of multiple Eisenstein values and differentiation with respect to their continuous elliptic parameters.
We apply our method to computing explicitly the Goncharov regulator integral associated to $K_4$ classes on modular curves in terms of $L$-values of modular forms.
We use this expression to connect it with the Beilinson regulator integral.
Our general approach reveals new interesting arithmetic phenomena and prospects for the general $K$-groups of modular curves. 
\end{abstract}

\maketitle

\section{Introduction}

The principal goal of this paper is to give an explicit relation between the integrals of two regulators defined on the $K$-group $H_{\mathcal M}^2(Y(N),\Q(n))\cong K_{2n-2}^{(n)}(Y(N))$ in the motivic cohomology of the modular curve $Y(N)$ of full level~$N$, in the case $n=3$.

In the recent work \cite{Bru20}, Brunault constructs explicit motivic cohomology classes $\xi(\bolda,\boldb)$ in $H_{\mathcal M}^2(Y(N),\Q(3))$ enumerated by $\bolda,\boldb\in(\Z/N\Z)^2$;
the classes are the images of degree 2 cocycles $\tilde\xi(\bolda,\boldb)$ in the Goncharov polylogarithmic complex $\Gamma(Y(N),3)$ under De Jeu's map \cite{Jeu95,Jeu96,Jeu00}. The construction uses the so-called Siegel units $g_\boldx \in \mathcal{O}(Y(N))^\times \otimes \Q$, $\boldx \in (\Z/N\Z)^2$, and certain relations analogous to modular symbols involving Milnor symbols $\{g_\boldx,g_\boldy\}$ in $K_2(Y(N)) \otimes \Q$.

The nontriviality of $\xi(\bolda,\boldb)$ for small values of $N$ is shown in \cite{Bru20} via computing numerically their images under the Goncharov regulator map $r_3(2)$ defined in \cite{Gon02}.
It is harder to compute the integral of $r_3(2)(\tilde\xi(\bolda,\boldb))$ theoretically; the existing literature lacks any such explicit calculations for the weight $3$ polylogarithmic complex of curves.
At the same time\,---\,and this serves as a natural motivation for these calculations\,---\,such integrals are related to (longstanding conjectural evaluations of) the Mahler measure of three-variable polynomials \cite[Chapter~6]{BZ20}. As an example, Lal\'in \cite{Lal15} has made an explicit connection between the Mahler measure of $(1+x)(1+y)+z$ and the Goncharov regulator for the elliptic curve $(1+x)(1+y)(1+\frac{1}{x})(1+\frac{1}{y}) = 1$.

Another motivation for computing the Goncharov regulator integrals comes from a conjecture of the first author \cite[Conjecture 1.3]{Bru20} predicting the proportionality of the Goncharov type elements $\xi(\bolda,\boldb)$ and the Beilinson elements \cite{Bei86} in the motivic cohomology of the modular curve $Y(N)$. This conjecture is based on numerical computations of the associated regulator integrals. This suggests, more generally, the existence of a relation between the two integrals, possibly at the level of cocycles, not just cohomology classes. It is this task that we perform in the present paper.

In order to compare the Goncharov and Beilinson regulator integrals
\[
\mathcal{G}(\bolda,\boldb) = \int_0^\infty r_3(2)(\tilde\xi(\bolda,\boldb))
\quad\text{and}\quad
\mathcal{B}(\bolda,\boldb) = \int_0^\infty \Eis^{0,0,1}_{\mathcal{D}}(\bolda,\boldb),
\]
where $\bolda,\boldb\in(\Z/N\Z)^2$,
we first express $\mathcal{G}(\bolda,\boldb)$ in terms of multiple (in fact, triple) modular values (MMV)\,---\,more specifically, multiple Eisenstein values (MEV).
This step requires defining the latter objects and the corresponding regularisation of integrals along the imaginary axis $\imaxis$, and setting up numerous properties and rules for MMVs.
This part follows closely Brown's expositions \cite{Bro17,Bro19} which we complement with our needs in Sections~\ref{reg ints} and \ref{mmvs}; Section~\ref{baby} serves a toy model for expressing the regulator integral on $K_2^{(2)}(Y(N))$ as a double modular value (a fact that seems to escape the literature).
The MMV expression for the regulator integral $\mathcal{G}(\bolda,\boldb)$ is computed in Section~\ref{G-tmv} for generic $\bolda,\boldb\in(\Z/N\Z)^2$; the result can be interpreted in terms of Eisenstein series with continuous parameters, when each $\bolda\in(\Z/N\Z)^2$ is rescaled to $\bolda/N\in(\frac1N\Z/\Z)^2$ and the latter interpolates to a function of $\bolda$ on $(\R/\Z)^2$. 
This line famously settled by A.~Weil in \cite{Wei76} allows us to differentiate with respect to the (real) \emph{elliptic} parameters $\bolda,\boldb$; more specifically, we choose to differentiate with respect to~$a_2$.
The differentiation of the Goncharov regulator integral in Section~\ref{G-reg} is preceded, in Section~\ref{BG-rels}, by derivation of auxiliary Borisov--Gunnells relations for pairwise products of Eisenstein series, and followed by reduction, in Section~\ref{RZ}, of the resulting expression of $\frac{\partial}{\partial a_2}\mathcal{G}(\bolda,\boldb)$ using the Rogers--Zudilin method.
Note that our proof of the Borisov--Gunnells relations requires the level $N$ structure to be used, so that we make several switches between interpolated and non-interpolated Eisenstein series.
Finally, in Section \ref{L-value} we deduce an $L$-value expression for $\mathcal{G}(\bolda,\boldb)$ by integrating its $a_2$-derivative; this brings us to the comparison of $\mathcal{G}(\bolda,\boldb)$ with $\mathcal{B}(\bolda,\boldb)$ in Section~\ref{B-reg}.

Our main results can be stated precisely as follows. We need the following Eisenstein series. Given a level $N \geq 1$, a weight $k \geq 1$ and an elliptic parameter $\boldx = (x_1,x_2)$ in $(\Z/N\Z)^2$, we define as in \cite[Section 10.4]{BZ20}
\begin{equation} \label{def GkN}
G^{(k);N}_\boldx(\tau) = a_0(G^{(k);N}_\boldx) + \sum_{\substack{m,n \geq 1 \\ (m,n) \equiv \boldx \bmod{N}}} m^{k-1} q^{mn/N} +(-1)^k \sum_{\substack{m,n \geq 1 \\ (m,n) \equiv -\boldx \bmod{N}}} m^{k-1} q^{mn/N},
\end{equation}
where the constant term is given by
\begin{equation*}
a_0(G^{(1);N}_\boldx) = \begin{cases} -B_1(\{\frac{x_2}{N}\}) & \textrm{if } x_1=0 \textrm{ and } x_2 \neq 0,\\
-B_1(\{\frac{x_1}{N}\}) & \textrm{if } x_1 \neq 0 \textrm{ and } x_2=0,\\
0 & \textrm{otherwise,}
\end{cases}
\end{equation*}
and for $k \geq 2$,
\begin{equation*}
a_0(G^{(k);N}_\boldx) = \begin{cases} - N^{k-1} B_k(\{\frac{x_1}{N}\})/k & \textrm{if } x_2=0,\\
0 & \textrm{if } x_2 \neq 0.
\end{cases}
\end{equation*}
Here $B_k(t)$ is the $k$-th Bernoulli polynomial (in particular $B_1(t) = t-\frac12$), and $\{ \,\cdot\, \}$ stands for the fractional part. The function $G^{(k);N}_\boldx$ is an Eisenstein series of weight $k$ and level $\Gamma(N)$, except for the case $k=2$ and $x_1=0$. Given a modular form $f = \sum_{n \geq 0} a_n q^{n/N}$ on $\Gamma(N)$, we write $L(f,s) = \sum_{n \geq 1} a_n (n/N)^{-s}$ for (the analytic continuation of) the $L$-function of $f$.

\begin{thm} \label{main thm 1}
    For any $\bolda=(a_1,a_2)$, $\boldb=(b_1,b_2)$ in $(\Z/N\Z)^2$ such that the coordinates of $\bolda$, $\boldb$ and $\bolda+\boldb$ are non-zero, we have
    \begin{align*}
    \mathcal{G}(\bolda,\boldb) & = \frac{3\pi^2}{N} L'\big(G^{(1);N}_{a_1,b_2} G^{(1);N}_{b_1,-a_2} + G^{(1);N}_{a_1,-b_2} G^{(1);N}_{b_1,a_2}, -1\big) \\
    \nonumber & \quad - \frac{\zeta(3)}{4} \big(B_2(\{\tfrac{a_1}{N}\}) + B_2(\{\tfrac{b_1}{N}\}) + 4 B_1(\{\tfrac{a_1}{N}\}) B_1(\{\tfrac{b_1}{N}\}) \\
    \nonumber & \qquad \qquad \quad - B_2(\{\tfrac{a_2}{N}\}) - B_2(\{\tfrac{b_2}{N}\}) - 4 B_1(\{\tfrac{a_2}{N}\}) B_1(\{\tfrac{b_2}{N}\}) \big).
    \end{align*}
\end{thm}

In his PhD thesis, Weijia Wang has made explicit Beilinson's theorem, by computing $\mathcal{B}(\bolda, \boldb)$ using the Rogers--Zudilin method \cite[Th\'eor\`eme 0.1.3]{Wan20}. The resulting $L$-value turns out to match the one in Theorem \ref{main thm 1}. We deduce our second main result, which is an explicit connection between $\mathcal{G}(\bolda, \boldb)$ and $\mathcal{B}(\bolda,\boldb)$.

\begin{thm} \label{main thm 2}
For any $\bolda=(a_1,a_2)$, $\boldb=(b_1,b_2)$ in $(\Z/N\Z)^2$ such that the coordinates of $\bolda$, $\boldb$ and $\bolda+\boldb$ are non-zero, we have
\begin{align*}
\mathcal{G}(\bolda,\boldb) = \frac{N^2}{6}\mathcal{B}(\bolda,\boldb)
& - \frac{\zeta(3)}{4} \big(B_2(\{\tfrac{a_1}{N}\}) + B_2(\{\tfrac{b_1}{N}\}) + 4 B_1(\{\tfrac{a_1}{N}\}) B_1(\{\tfrac{b_1}{N}\}) \\
    \nonumber & \qquad \qquad - B_2(\{\tfrac{a_2}{N}\}) - B_2(\{\tfrac{b_2}{N}\}) - 4 B_1(\{\tfrac{a_2}{N}\}) B_1(\{\tfrac{b_2}{N}\}) \big).
\end{align*}
\end{thm}

This gives some evidence for \cite[Conjecture 1.3]{Bru20} asserting the proportionality of the motivic cohomology classes $\xi(\bolda,\boldb)$ and $\Eis^{0,0,1}(\bolda,\boldb)$. The discrepancy appearing with the rational multiple of $\zeta(3)$ may come from the particular choices of representatives of the Deligne--Beilinson cohomology classes, since $\imaxis$ is not a closed path in $Y(N)(\C)$.

Though Theorem \ref{main thm 2} compares two periods\,---\,in the sense of Kontsevich and Zagier \cite{KZ01}\,---\,with each other, our proof of the equality exploits integration of transcendental functions. This is one of many examples when it is hard to expect that an equality can be deduced from a manipulation with periods using only the rules of Calculus.

Our strategy and its execution reveal several interesting arithmetic phenomena and prospects for the general $K$-groups $K_{2n-2}^{(n)}(Y(N))$ with $n\ge2$.
First of all, we find the theory of multiple modular values developed by Brown \cite{Bro17,Bro19}, and specifically the MEVs, intrinsic to dealing with both the $L$-values $L(E,n)$ of modular elliptic curves $E$ and regulators of Beilinson and Goncharov types.
One may hope that if $E$ has conductor $N$, then $L(E,n)$ can be always written as a
$\mathbf Q$-linear combination of length $n$ MEVs with Eisenstein series of weight 2 and level $N$. This should be explained by a relation between the Goncharov regulator $r_n(2)$ and iterated integrals of length $n$.

In contrast, the Beilinson regulator produces MEVs of length 2, with weights of Eisenstein series depending on~$n$, and this corresponds to a representation of $L(E,n)$ as a $\mathbf Q$-linear combination of length 2 MEVs.
The difference in production from the two regulators suggests the existence of intermediate regulators in the case $n \geq 4$, to cover the entire spectrum of possibilities of MEVs.
At the moment we can only speculate in this direction.
Notice that representativeness of $L(E,n)$ by different length MEVs seems to be part of some general structure; this indicates existence of possible `length drops' for MMVs themselves.
Our calculation of $\frac{\partial}{\partial a_2}\mathcal{G}(\bolda,\boldb)$ in Sections~\ref{BG-rels}--\ref{RZ} gives an example of such a length drop by~1. Are there identities of MMVs in which the length drops by 2 or more? Does a general theory for length reduction exist? Answering such questions will help to understand the cases with $n \geq 4$.

Most of our results in Sections~\ref{BG-rels}--\ref{RZ} are limited to the situations required for dealing with the Goncharov regulator $r_n(2)$ when $n=3$ but can be potentially generalised.
Our Theorem~\ref{thm RZ Iuv} below is already more general than needed in this paper but can be generalised further; the Borisov--Gunnells relations exist in arbitrary weight.
Differentiation of such relations with respect to elliptic parameters was already used by Borisov and Gunnells in \cite[Section 3]{BG03}, though with no connection to computing regulators or MMVs.

Our final remark is that writing $r_3(2)$ in terms of MMVs provides one with an efficient way for computing the Goncharov regulator, which is faster when compared with the method used in~\cite{Bru20}.

\medskip
This project greatly benefited from discussions at the International Groupe de Travail on differential equations in Paris. The first author thanks the participants of the group, especially Spencer Bloch, Vasily Golyshev, Rob de Jeu and Matt Kerr, for illuminating perspectives.
We are particularly thankful to Mista Boisan, whose discovery of three-term relations for double Eisenstein values in the thesis \cite{Boi26} has led to a significant simplification of the statement of Theorem~\ref{thm r32 formula}, as well as of subsequent computations in Section~\ref{G-reg}.
We are also grateful to our colleagues whose feedback on several aspects of this work have been instrumental, to Francis Brown, Kamal Khuri-Makdisi, Matilde Lal\'\i n, Riccardo Pengo and Weijia Wang.

\medskip
Iterated integrals of modular forms appear intrinsically in the study of modular regulators and we feel appropriate to dedicate our work to Yuri Manin, who pioneered this topic in \cite{Man05,Man06}. We would benefit from discussing our results with him. But he passed away unexpectedly, full of many ideas that our mathematics world could have grown further on.

\section{Regularised iterated integrals} \label{reg ints}

\subsection{Admissible functions}

We define the class of functions and differential forms that we wish to integrate. Let $\h = \{\tau \in \C : \im(\tau)>0\}$ be the upper half-plane, and $\imaxis = \{iy : \, y > 0\}$ the imaginary axis.

\begin{definition}[Admissibility at infinity] \label{Defi 1}
A $C^\infty$ function $f \colon \imaxis \to \C$ is called \emph{admissible at $\infty$} if it can be written $f(\tau) = f^\infty(\tau) + f^0(\tau)$, where $f^\infty(\tau) \in \C[\tau]$ is a polynomial, and $f^0(\tau)$ has exponential decay as $\im(\tau) \to +\infty$: there exists $0<c<1$ such that $f^0(\tau) = O_{\tau \to \infty}(c^{\im(\tau)})$. In this case, the \emph{regularised value of $f$ at infinity}, denoted by $f(\infty)$, is defined as the constant term of the polynomial $f^\infty$.
\end{definition}

Note that the decomposition $f = f^\infty + f^0$ is unique, hence $f(\infty)$ is well defined.

\begin{definition}
A $C^\infty$ differential form $\omega = f(\tau) \, d\tau$ on $\imaxis$ is called \emph{admissible at $\infty$} if $f$ is admissible at $\infty$. We then write $\omega = \omega^\infty + \omega^0$ with $\omega^\infty = f^\infty(\tau) \, d\tau$ and $\omega^0 = f^0(\tau) \, d\tau$.
\end{definition}

As an example, if $f$ is a modular form of weight $k \geq 1$ on some finite index subgroup of $\SL_2(\Z)$, then $\omega = f(\tau) \tau^m \, d\tau$ is admissible at $\infty$ for any integer $m \geq 0$. Note that if a form $\omega$ is admissible at $\infty$, then so are $\re(\omega) = \frac12 (\omega + \bar{\omega})$ and $\im(\omega) = \frac{1}{2i} (\omega - \bar{\omega})$. 


\begin{lem} \label{lem f omega}
If a function $f$ and a form $\omega$ on $\imaxis$ are admissible at $\infty$, then so is $f\omega$.
\end{lem}

\subsection{Regularisation at infinity}

We now come to regularisation of iterated integrals. We follow Brown's definition \cite[Section 4.1]{Bro17} and show how it can be expressed via successive one-variable regularisations.

Let us first consider the case of a single integral from $\tau$ to $\infty$. Let $\omega$ be a differential form on $\imaxis$ which is admissible at $\infty$. Brown's definition translates to
\begin{equation} \label{def reg single}
\int_\tau^\infty \omega := \lim_{p \to \infty} \int_\tau^p \omega + \int_p^0 \omega^\infty \qquad (p = iy, \; y \to +\infty).
\end{equation}

We can actually get rid of the limit in \eqref{def reg single}.

\begin{lem} \label{lem 1}
Let $\omega$ be a differential form on $\imaxis$ which is admissible at $\infty$. The limit in \eqref{def reg single} exists, and we have
\begin{equation} \label{lem 1 eq}
\int_\tau^\infty \omega = \int_\tau^\infty \omega^0 + \int_\tau^0 \omega^\infty.
\end{equation}
Moreover, the error term in the convergence of \eqref{def reg single} is $O_{p \to \infty}(c^{\im(p)})$ with $0 < c < 1$, the constant $c$ being uniform with respect to $\tau$ on domains of the form $\{\im(\tau) \geq y_0 > 0\}$.
\end{lem}

\begin{proof} Indeed,
\begin{equation*}
\int_\tau^p \omega + \int_p^0 \omega^\infty = \int_\tau^p \omega^0 + \int_\tau^p \omega^\infty + \int_p^0 \omega^\infty = \int_\tau^p \omega^0 + \int_\tau^0 \omega^\infty. \qedhere
\end{equation*}
\end{proof}

We refer to the right-hand side of \eqref{lem 1 eq} as the practical regularised integral. Note that the regularised integral recovers the classical integral in the case $\omega$ is integrable on $\left[\tau, i\infty\right[$ (which happens if and only if $\omega^\infty = 0$). Lemma \ref{lem 1} has the following consequence.

\begin{lem} \label{lem diff int tau infty}
Let $\omega$ be a differential form on $\imaxis$ which is admissible at $\infty$. Then the function $F(\tau) = -\int_\tau^\infty \omega$ is admissible at $\infty$. Moreover, $F$ is the unique primitive of $\omega$ whose regularised value at $\infty$ is zero.
\end{lem}

In particular, if a form $\omega$ is admissible at $\infty$, then any primitive of $\omega$ is again admissible at $\infty$. On the other hand, the differential of an admissible function $f$ need not be admissible, because there is no control on the derivative of $f^0$.


\begin{lem} \label{lem int df infty}
Let $f \colon \imaxis \to \C$ be a function such that $df$ is admissible at $\infty$. Then $f$ is admissible at $\infty$ and $\int_\tau^\infty df = f(\infty)-f(\tau)$, where $f(\infty)$ is the regularised value at $\infty$ as in Definition~\textup{\ref{Defi 1}}.
\end{lem}

\begin{proof}
This follows from Lemma \ref{lem diff int tau infty} applied to $\omega = df$.
\end{proof}

One should be careful that in general $\int_\tau^\infty \omega$ does not converge to zero as $\tau \to \infty$: this can be seen from \eqref{lem 1 eq}. For example, if $f(\tau) = \sum_{n \geq 0} a_n q^n$ is a modular form, then
\begin{equation*}
\int_\tau^\infty f(\tau_1) \, d\tau_1 = -\frac{1}{2\pi i} \sum_{n \geq 1} \frac{a_n}{n} q^n - a_0 \tau.
\end{equation*}

One outcome of Lemma \ref{lem int df infty} is the following formula for integration by parts: if the forms $df$ and $dg$ are admissible, then $f$ and $g$ are admissible as well, and
\begin{equation}
\int_\tau^\infty\frac{df}{d\tau}(\tau_1)g(\tau_1)\,d\tau_1
=f(\infty)g(\infty)-f(\tau)g(\tau)
-\int_\tau^\infty f(\tau_1)\frac{dg}{d\tau}(\tau_1)\,d\tau_1.
\label{IPoo}
\end{equation}
Once again, here $f(\infty)$ and $g(\infty)$ are the regularised values at $\infty$ as in Definition~\ref{Defi 1}.

Now let us consider the case of iterated integrals. Brown's definition \cite[Section 4.1]{Bro17} uses a tangential base point at $\infty$. This intrinsic definition has the advantage of giving naturally the shuffle relations for the regularised iterated integrals. Unraveling Brown's definition gives:

\begin{definition}
Let $\omega_1,\ldots,\omega_n$ be differential forms on $\imaxis$ which are admissible at $\infty$. Define
\begin{equation} \label{def reg multiple}
\int_\tau^\infty \omega_1 \ldots \omega_n := \lim_{p \to \infty} \sum_{k=0}^n \int_\tau^p \omega_1 \ldots \omega_k \times \int_p^0 \omega_{k+1}^\infty \ldots \omega_n^\infty.
\end{equation}
\end{definition}

We will justify below the convergence in \eqref{def reg multiple}. For certain computations, we will need to express the regularised iterated integral as a succession of one-variable regularised integrals. We introduce the following `na\"\i ve' regularisation:
\begin{equation} \label{notation practical reg int}
\int_\tau^{\infty,*} \omega_1 \ldots \omega_n := \int_\tau^{\infty} \omega_1(\tau_1) \int_{\tau_1}^{\infty} \omega_2(\tau_2) \cdots \int_{\tau_{n-1}}^{\infty} \omega_n(\tau_n),
\end{equation}
where the right-hand integrals are understood as \eqref{lem 1 eq}.

\begin{lem} \label{lem nested adm}
The na\"\i ve regularised integral $\int_\tau^{\infty,*} \omega_1 \ldots \omega_n$ is well-defined and is admissible at $\infty$ as a function of $\tau$. Its regularised value at $\infty$ is zero.
\end{lem}

\begin{proof}
This follows from inductive application of Lemmas \ref{lem f omega} and \ref{lem diff int tau infty}.
\end{proof}

\begin{pro} \label{main pro}
Let $\omega_1,\ldots,\omega_n$ be differential forms which are admissible at $\infty$. Then
\begin{equation*}
\int_\tau^\infty \omega_1 \ldots \omega_n = \int_\tau^{\infty,*} \omega_1 \ldots \omega_n.
\end{equation*}
\end{pro}


To prove this, we need the following lemma.

\begin{lem} \label{lem 2}
The polynomial part of the na\"\i ve regularised integral is given by
\begin{equation*}
\Bigl( \int_\tau^{\infty,*} \omega_1 \ldots \omega_n \Bigr)^\infty = \int_\tau^0 \omega_1^\infty \ldots \omega_n^\infty,
\end{equation*}
where the right-hand side is the usual (absolutely convergent) iterated integral.
\end{lem}

\begin{proof}
We proceed by induction on $n$. The case $n=1$ follows from Lemma \ref{lem 1}. For $n \geq 2$, we have
\begin{equation*}
\int_\tau^{\infty,*} \omega_1 \ldots \omega_n = \int_\tau^{\infty,*} \omega_1(\tau_1) \int_{\tau_1}^{\infty,*} \omega_2 \ldots \omega_n.
\end{equation*}
By the induction hypothesis applied to $\omega_2 \ldots \omega_n$, we have
\begin{equation*}
\Bigl(\omega_1(\tau_1) \int_{\tau_1}^{\infty,*} \omega_2 \ldots \omega_n\Bigr)^\infty = \omega_1^\infty(\tau_1) \Bigl(\int_{\tau_1}^{\infty,*} \omega_2 \ldots \omega_n\Bigr)^\infty = \omega_1^\infty(\tau_1) \int_{\tau_1}^0 \omega_2^\infty \ldots \omega_n^\infty.
\end{equation*}
Therefore, using Lemma \ref{lem 1},
\begin{equation} \label{lem 2 eq 3}
\int_\tau^{\infty,*} \omega_1 \ldots \omega_n = \int_\tau^\infty \Bigl(\omega_1(\tau_1) \int_{\tau_1}^{\infty,*} \omega_2 \ldots \omega_n\Bigr)^0 + \int_\tau^0 \omega_1^\infty(\tau_1) \int_{\tau_1}^0 \omega_2^\infty \ldots \omega_n^\infty.
\end{equation}
The first term in \eqref{lem 2 eq 3} decays exponentially as $\tau \to \infty$, and the second term is a polynomial in $\tau$, which finishes the proof.
\end{proof}

Proposition \ref{main pro} is now a consequence of the following finer result, which controls the convergence as $p \to \infty$.

\begin{pro} \label{pro 1}
We have
\begin{equation} \label{pro 1 eq}
\sum_{k=0}^n \int_\tau^p \omega_1 \ldots \omega_k \times \int_p^0 \omega_{k+1}^\infty \ldots \omega_n^\infty = \int_\tau^{\infty,*} \omega_1 \ldots \omega_n + O_{p \to \infty}(c^{\im(p)}) \qquad (0 < c < 1),
\end{equation}
the constant $c$ being uniform with respect to $\tau$ on domains of the form $\{\im(\tau) \geq y_0 > 0\}$.
\end{pro}

\begin{proof}
We proceed by induction on $n$. The case $n=1$ follows from Lemma \ref{lem 1}. Let $n \geq 2$. Using the induction hypothesis to $\omega_2 \ldots \omega_n$, the left-hand side of \eqref{pro 1 eq} can be written as
\begin{align}
\nonumber & \int_\tau^p \omega_1(\tau_1) \Bigl( \sum_{k=1}^n \int_{\tau_1}^p \omega_2 \ldots \omega_k \times \int_p^0 \omega_{k+1}^\infty \ldots \omega_n^\infty \Bigr) + \int_p^0 \omega_1^\infty \ldots \omega_n^\infty \\
\nonumber &\quad = \int_\tau^p \omega_1(\tau_1) \Bigl( \int_{\tau_1}^{\infty,*} \omega_2 \ldots \omega_n + O_{p \to \infty}(c^{\im(p)}) \Bigr) + \int_p^0 \omega_1^\infty \ldots \omega_n^\infty \\
\nonumber &\quad = \int_\tau^p \omega_1(\tau_1) \int_{\tau_1}^{\infty,*} \omega_2 \ldots \omega_n + \Bigl( \int_\tau^p \omega_1(\tau_1) \Bigr) O_{p \to \infty}(c^{\im(p)}) + \int_p^0 \omega_1^\infty \ldots \omega_n^\infty \\
\label{pro 1 eq 2} &\quad = \int_\tau^p \omega_1(\tau_1) \int_{\tau_1}^{\infty,*} \omega_2 \ldots \omega_n + \int_p^0 \omega_1^\infty \ldots \omega_n^\infty + O_{p \to \infty}(c_2^{\im(p)}).
\end{align}
Consider the differential form
\begin{equation*}
\alpha(\tau_1) = \omega_1(\tau_1) \int_{\tau_1}^{\infty,*} \omega_2 \ldots \omega_n.
\end{equation*}
Applying Lemma \ref{lem 2} to $\omega_2 \ldots \omega_n$, the polynomial part of $\alpha$ is
\begin{equation*}
\alpha^\infty(\tau_1) = \omega_1^\infty(\tau_1) \int_{\tau_1}^0 \omega_2^\infty \ldots \omega_n^\infty.
\end{equation*}
Therefore,
\begin{align*}
\eqref{pro 1 eq 2} & = \int_\tau^p \alpha(\tau_1) + \int_p^0 \omega_1^\infty(\tau_1) \int_{\tau_1}^0 \omega_2^\infty \ldots \omega_n^\infty + O_{p \to \infty}(c_2^{\im(p)}) \\
& = \int_\tau^p \alpha(\tau_1) + \int_p^0 \alpha^\infty(\tau_1) + O_{p \to \infty}(c_2^{\im(p)}) \\
& = \int_\tau^{\infty,*} \alpha + O_{p \to \infty}(c_3^{\im(p)}). \qedhere
\end{align*}
\end{proof}

Proposition \ref{main pro} and Lemma \ref{lem diff int tau infty} have the following consequence.

\begin{lem} \label{lem diff int n}
For any differential forms $\omega_1,\ldots,\omega_n$ which are admissible at $\infty$, we have
\begin{equation*}
d \left(\int_\tau^\infty \omega_1 \ldots \omega_n \right) = - \omega_1(\tau) \int_\tau^\infty \omega_2 \ldots \omega_n.
\end{equation*}
\end{lem}

\subsection{Regularisation at zero} \label{sec reg 0}
The matrix $\sigma = (\begin{smallmatrix} 0 & -1 \\ 1 & 0 \end{smallmatrix})$ acts on $\h$ by $\tau \mapsto -1/\tau$. For a differential form $\omega$ on $\imaxis$, we write $\omega^\sigma = \sigma^* \omega$.

\begin{definition}[Admissibility at 0] \label{def reg 0}
A $C^\infty$ function $f \colon \imaxis \to \C$ is called \emph{admissible at $0$} if the function $g(\tau) = f(-1/\tau)$ is admissible at $\infty$. In this case, the \emph{regularised value of $f$ at $0$} is defined as $f(0) = g(\infty)$.

A $C^\infty$ differential form $\omega$ on $\imaxis$ is called \emph{admissible at $0$} if $\omega^\sigma$ is admissible at $\infty$.
\end{definition}

\begin{definition}[Admissibility] \label{def reg 0 oo}
A function or differential form on $\imaxis$ is called admissible if it is admissible at both $0$ and $\infty$.
\end{definition}

\begin{example} \label{exadmissible}
\begin{itemize}
\item The only polynomials in $\tau$ which are admissible are the constants.
\item If $f$ is a modular form of weight $k \geq 2$ on a finite index subgroup of $\SL_2(\Z)$, then $\omega = f(\tau) \,d\tau$ is admissible. In fact $f(\tau) \tau^{m-1} d\tau$ is admissible for any $m \in \{1,\ldots,k-1\}$. If $f$ is a cusp form, then $f(\tau) \tau^{m-1} d\tau$ is admissible for any $m \in \Z$.
\end{itemize}
\end{example}

Lemmas  \ref{lem f omega} and \ref{lem diff int tau infty} also hold for admissibility at $0$, and thus for admissibility:

\begin{lem} \label{lem admissible primitive}
Let $\omega$ be an admissible differential form on $\imaxis$, and $f(\tau)$ any primitive of $\omega$. Then $f$ is admissible.
\end{lem}

We now want to define the regularised iterated integral from $0$ to $\tau$ of differential forms $\omega_1,\ldots,\omega_n$ which are admissible at $0$. We begin with the case $n=1$. Formal considerations lead to the following definition.

\begin{definition}
Let $\omega$ be a differential form on $\imaxis$ which is admissible at $0$. We set
\begin{equation*}
\int_0^\tau \omega := - \int_{-1/\tau}^\infty \omega^\sigma,
\end{equation*}
which is well-defined since $\omega^\sigma$ is admissible at $\infty$.
\end{definition}

\begin{lem} \label{lem diff int 0 tau}
Let $\omega$ be a differential form on $\imaxis$ which is admissible at $0$. Then $\int_0^\tau \omega$ is the unique primitive of $\omega$ whose regularised value at $0$ is zero.
\end{lem}

\begin{proof}
By Lemma \ref{lem diff int tau infty} applied to $\omega^\sigma$, we know that $d \left(\int_\tau^\infty \omega^\sigma\right) = - \omega^\sigma$. Pulling back by $\sigma \colon \tau \mapsto -1/\tau$ gives the desired identity. The statement about the regularised value at $0$ follows from the definition and Lemma \ref{lem diff int tau infty}.
\end{proof}

Now we proceed to the iterated case. Let $\omega_1,\ldots,\omega_n$ be differential forms on $\imaxis$ which are admissible at $0$. We want to set
\begin{equation*}
\int_0^\tau \omega_1 \ldots \omega_n = \int_\infty^{-1/\tau} \omega_1^\sigma \ldots \omega_n^\sigma.
\end{equation*}
The right-hand side can be given a meaning using the reversal of paths formula
\begin{equation*}
\int_a^b \omega_1 \ldots \omega_n = (-1)^n \int_b^a \omega_n \ldots \omega_1.
\end{equation*}
This leads to:

\begin{definition}
For any forms $\omega_1,\ldots,\omega_n$ on $\imaxis$ which are admissible at $0$, we define
\begin{equation*}
\int_0^\tau \omega_1 \ldots \omega_n := (-1)^n \int_{-1/\tau}^\infty \omega_n^\sigma \ldots \omega_1^\sigma.
\end{equation*}
\end{definition}

We have the following analogues of Lemmas \ref{lem nested adm} and \ref{lem diff int n}.

\begin{lem} \label{lem adm 0}
The integral $\int_0^\tau \omega_1 \ldots \omega_n$ is admissible at $0$ as a function of $\tau$, and its regularised value at $0$ is zero.
\end{lem}

\begin{proof}
This follows from Lemma \ref{lem nested adm} applied to $\omega_n^\sigma \ldots \omega_1^\sigma$.
\end{proof}

\begin{lem} \label{lem diff int n 0}
We have
\begin{equation*}
d \left(\int_0^\tau \omega_1 \ldots \omega_n \right) = \omega_n(\tau) \int_0^\tau \omega_1 \ldots \omega_{n-1}.
\end{equation*}
\end{lem}

\begin{proof}
By Lemma \ref{lem diff int n}, we have
\begin{equation*}
d \left(\int_\tau^\infty \omega_n^\sigma \ldots \omega_1^\sigma \right) = - \omega_n^\sigma(\tau) \int_\tau^\infty \omega_{n-1}^\sigma \ldots \omega_1^\sigma.
\end{equation*}
Applying $\sigma^*$ to this identity gives
\begin{equation*}
d \left(\int_{-1/\tau}^\infty \omega_n^\sigma \ldots \omega_1^\sigma \right) = - \omega_n(\tau) \int_{-1/\tau}^\infty \omega_{n-1}^\sigma \ldots \omega_1^\sigma = (-1)^n \omega_n(\tau) \int_0^\tau \omega_1 \ldots \omega_{n-1}. \qedhere
\end{equation*}
\end{proof}

\subsection{Regularisation from zero to infinity} \label{reg 0 oo}

Note that if $\omega$ is admissible, then the integral $\int_0^\infty \omega := \int_0^\tau \omega + \int_\tau^\infty \omega$ is well-defined and independent of $\tau$ by Lemmas \ref{lem diff int tau infty} and \ref{lem diff int 0 tau}. Moreover, if $\omega$ is integrable, then this definition coincides with the usual (convergent) integral of $\omega$ on $\imaxis$. 

In the iterated case, the composition of paths formula forces the following definition. 

\begin{definition} \label{def int reg 0 infty}
Let $\omega_1,\ldots,\omega_n$ be admissible differential forms on $\imaxis$. We define
\begin{equation} \label{eq int reg 0 infty}
\int_0^\infty \omega_1 \ldots \omega_n = \sum_{k=0}^n \int_0^\tau \omega_1 \ldots \omega_k \times \int_\tau^\infty \omega_{k+1} \ldots \omega_n.
\end{equation}
\end{definition}

\begin{lem}
The definition \eqref{eq int reg 0 infty} does not depend on $\tau$.
\end{lem}

\begin{proof}
Using Lemmas \ref{lem diff int n} and \ref{lem diff int n 0}, the differential of the right-hand side of \eqref{eq int reg 0 infty} is
\begin{align*}
& \sum_{k=0}^{n-1} \int_0^\tau \omega_1 \ldots \omega_k \times \big(- \omega_{k+1}(\tau)\big) \int_\tau^\infty \omega_{k+2} \ldots \ldots \omega_n \\
& \quad + \sum_{k=1}^n \omega_k(\tau) \int_0^\tau \omega_1 \ldots \omega_{k-1} \times \int_\tau^\infty \omega_{k+1} \ldots \ldots \omega_n
\end{align*}
which vanishes by changing $k \to k+1$ in the second sum.
\end{proof}

The last lemma naturally brings us to a statement which will be important for expressing the Goncharov regulator integral in terms of iterated integrals.

\begin{pro} \label{pro int 0 infty}
Let $\omega_1,\ldots,\omega_n$ be admissible differential forms on $\imaxis$. Then
\begin{equation*}
\int_\tau^\infty \omega_1 \ldots \omega_n
\end{equation*}
is admissible at $0$ as a function of $\tau$, and its regularised value at $0$ is $\int_0^\infty \omega_1 \ldots \omega_n$. Moreover,
\begin{equation} \label{eq int 0 infty}
\int_0^\infty \omega_1 \ldots \omega_n = \int_0^\infty \omega_1(\tau_1) \int_{\tau_1}^\infty \omega_2(\tau_2) \cdots \int_{\tau_{n-1}}^\infty \omega_n(\tau_n),
\end{equation}
where the right-hand side of \eqref{eq int 0 infty} is understood as successive one-variable regularisations.
\end{pro}

\begin{proof}
For the first part of the proposition, we proceed by induction on $n$. The case $n=1$ follows from $\int_0^\infty \omega_1 = \int_0^\tau \omega_1 + \int_\tau^\infty \omega_1$ and Lemma \ref{lem diff int 0 tau}. For $n \geq 2$, we can write
\begin{equation*}
\int_\tau^\infty \omega_1 \ldots \omega_n = \int_0^\infty \omega_1 \ldots \omega_n - \sum_{k=1}^n \int_0^\tau \omega_1 \ldots \omega_k \times \int_\tau^\infty \omega_{k+1} \ldots \omega_n.
\end{equation*}
By the induction hypothesis and Lemma \ref{lem adm 0}, the right-hand side is admissible at $0$. Moreover, the regularised value at $0$ of the product
\begin{equation*}
\int_0^\tau \omega_1 \ldots \omega_k \times \int_\tau^\infty \omega_{k+1} \ldots \omega_n
\end{equation*}
is the product of the regularised values, hence it is zero by Lemma \ref{lem adm 0}.

Finally, \eqref{eq int 0 infty} follows formally by using the case $n=1$ with the form $\omega_1(\tau_1) \int_{\tau_1}^\infty \omega_2 \ldots \omega_n$.
\end{proof}

\subsection{Shuffle relations of iterated integrals}
An important feature of all the regularisations we have discussed, $\int_0^\infty$ as well as $\int_0^\tau$ and $\int_\tau^\infty$, is that they satisfy the shuffle relations.
Let $V$ be the $\C$-vector space of admissible differential $1$-forms on $\imaxis$. Consider the functional $I_0^\infty \colon V \to \C$ sending $\omega$ to the regularised integral $\int_0^\infty \omega$. Then the regularised iterated integrals of Section \ref{reg 0 oo} provide a natural extension of $I_0^\infty$ to the tensor algebra $T(V) = \bigoplus_{n \geq 0} V^{\otimes n}$,
\begin{equation*}
I_0^\infty \colon T(V) \to \C, \quad \omega_1 \otimes \ldots \otimes \omega_n \mapsto \int_0^\infty \omega_1 \ldots \omega_n \qquad (\omega_i \in V).
\end{equation*}
The algebra $T(V)$ has a structure of Hopf algebra, called the shuffle algebra, with the multiplication $T(V) \otimes T(V) \to T(V)$ given by the shuffle product
\begin{equation*}
\omega_1 \ldots \omega_p \Sha \omega_{p+1} \ldots \omega_n
= \sum_{\sigma \in S_{p,n-p}} \omega_{\sigma^{-1}(1)} \ldots \omega_{\sigma^{-1}(n)},
\end{equation*}
where the sum is over the $(p,n-p)$-shuffles.

More generally, one may integrate over a path $\gamma$ which is either a finite interval in $\imaxis$, or a path in the `tangent space of $\h$ at $0$ or $\infty$' involving tangential base points $\vec{1}_0$ or $\vec{1}_\infty$, as defined in \cite[Section 4]{Bro17}. For such a path $\gamma$, there is an associated functional $I_\gamma : T(V) \to \C$. The important point is that, as $I_\gamma$ is essentially an ordinary iterated integral, it satisfies the shuffle relations; in other words, $I_\gamma$ is a morphism of algebras. Moreover, regularised integrals on $\imaxis$ are defined by formally concatenating the paths $\vec{1}_0 \to i/y \to iy \to \vec{1}_\infty$ (with $y \to \infty$). Formal considerations using the Hopf algebra structure on $T(V)$ lead to the following proposition.

\begin{pro} \label{pro 0 shuffle}
The functional $I_0^\infty \colon T(V) \to \C$ satisfies the shuffle relations\textup; in other words,
\begin{equation*}
\int_0^\infty \omega_1 \ldots \omega_p
\times \int_0^\infty \omega_{p+1} \ldots \omega_n
= \int_0^\infty \omega_1 \ldots \omega_p \Sha \omega_{p+1} \ldots \omega_n
\qquad (\omega_i \in V)
\end{equation*}
for any choice of $p\in\{1,\dots,n\}$.
\end{pro}

For more details, we refer the reader to \cite[Section~4]{Bro17}.

\subsection{The Newton--Leibniz formula and integration by parts} \label{general NL}
We now want to generalise Lemma \ref{lem int df infty} in the form of formula \eqref{IPoo} for integration by parts to iterated integrals $\int_\tau^\infty \omega_1 \ldots \omega_n $. `with respect to a particular form' $\omega_p(\tau)=f(\tau)\,d\tau$, assuming that the 1-forms $\omega_1,\dots,\omega_n$ are admissible.
As we already know from the lemma, $f$ is an admissible function;
we keep the notation $f(\infty)$ and $f(0)$ for its regularised values at $\infty$ and $0$ as in Definition \ref{Defi 1}.

If $p=1$ we get, using \eqref{IPoo},
\begin{align*}
&
\int_\tau^\infty\frac{df}{d\tau}(\tau_1)\,d\tau_1\,\omega_2(\tau_2)\ldots\omega_n(\tau_n)
\\ &\quad
=\int_\tau^\infty\frac{df}{d\tau}(\tau_1)\,d\tau_1\int_{\tau_1}^\infty\omega_2(\tau_2)\ldots\omega_n(\tau_n)
\\ &\quad
=-f(\tau)\int_\tau^\infty\omega_2(\tau_2)\ldots\omega_n(\tau_n)
+\int_\tau^\infty f(\tau_2)\omega_2(\tau_2)\ldots\omega_n(\tau_n).
\end{align*}
For $p>1$, we write
\begin{align*}
&
\int_\tau^\infty\omega_1(\tau_1)\ldots\frac{df}{d\tau}(\tau_p)\,d\tau_l\ldots\omega_n(\tau_n)
\\ &\quad
=\int_\tau^\infty\omega_1(\tau_1)\ldots\int_{\tau_{p-3}}^\infty\omega_{p-2}(\tau_{p-2})\int_{\tau_{p-2}}^\infty\omega_{p-1}(\tau_{p-1})\int_{\tau_{p-1}}^\infty\frac{df}{d\tau}(\tau_p)\,d\tau_p\,\omega_{p+1}(\tau_{p+1})\ldots\omega_n(\tau_n)
\\ \intertext{and use the above derivation to conclude that this is}
&\quad
=-\int_\tau^\infty\omega_1(\tau_1)\ldots\omega_{p-2}(\tau_{p-2})\,f(\tau_{p-1})\omega_{p-1}(\tau_{p-1})\,
\omega_{p+1}(\tau_{p+1})\ldots\omega_n(\tau_n)
\\ &\quad\qquad
+\int_\tau^\infty\omega_1(\tau_1)\ldots\omega_{p-1}(\tau_{p-1})\,
f(\tau_{p+1})\omega_{p+1}(\tau_{p+1})\,\omega_{p+2}(\tau_{p+2})\ldots\omega_n(\tau_n).
\end{align*}
Taking the regularised value as $\tau \to 0$ and using Proposition \ref{pro int 0 infty}, we get
\begin{align}
\label{NL-gen}
&
\int_0^\infty\omega_1(\tau_1)\ldots\frac{df}{d\tau}(\tau_p)\,d\tau_p\ldots\omega_n(\tau_n)
\\ &\quad
=\int_0^\infty\omega_1(\tau_1)\ldots\omega_{p-1}(\tau_{p-1})
\,f(\tau_{p+1})\omega_{p+1}(\tau_{p+1})\,
\omega_{p+2}(\tau_{p+2})\ldots\omega_n(\tau_n)
\nonumber\\ &\quad\qquad
-\int_0^\infty\omega_1(\tau_1)\ldots\omega_{p-2}(\tau_{p-2})\,f(\tau_{p-1})\omega_{p-1}(\tau_{p-1})\,\omega_{p+1}(\tau_{p+1})\ldots\omega_n(\tau_n),
\nonumber
\end{align}
where the first summand is interpreted as
\[
f(\infty)\int_0^\infty\omega_1(\tau_1)\ldots\omega_{n-1}(\tau_{n-1})
\]
when $p=n$, while the second summand is
\[
-f(0)\int_0^\infty\omega_2(\tau_2)\ldots\omega_n(\tau_n)
\]
when $p=1$.

In the particular case $p=n=1$, formula \eqref{NL-gen} extends Lemma~\ref{lem int df infty} to regularised integrals from $0$ to $\infty$:

\begin{lem} \label{lem int df 0 infty}
Let $f \colon \imaxis \to \C$ be a $C^\infty$ function such that $df$ is admissible. Then $f$ is admissible and $\int_0^\infty df = f(\infty)-f(0)$.
\end{lem}

\subsection{Iterated integrals with parameters} \label{subsec reg int param}
In this part we record our needs for differentiating the regularised (iterated) integral when a differential form depends smoothly on a \emph{real} parameter.

\begin{pro} \label{pro reg int param}
Let $(\omega_a)_a$ be a family of differential forms on $\imaxis$ admissible at $\infty$ depending on a single real parameter $a$. Write $\omega_a = f_a(\tau) \,d\tau$, and assume that\textup:
\begin{itemize}
\item[\textup{(i)}] the polynomial $f_a^\infty(\tau)$ has degree bounded independently of $a$, and its coefficients are differentiable functions of $a$\textup;
\item[\textup{(ii)}] $f_a^0(\tau)$ is differentiable as a function of $a$\textup;
\item[\textup{(iii)}] locally on $a$, there exists a constant $0<c<1$ such that $\frac{d}{da} f_a^0(\tau) = O_{\tau \to \infty}(c^{\im(\tau)})$, where the implied constant does not depend on $a$.
\end{itemize}
Then the function $a \mapsto \int_\tau^\infty \omega_a$ is differentiable, and we have
\begin{equation*}
\frac{d}{da} \Big(\int_\tau^\infty \omega_a \Big) = \int_\tau^\infty \frac{d}{da} \omega_a.
\end{equation*}
\end{pro}

\begin{proof}
Note that the assumptions imply that for every $a$, the form $\frac{d}{da} \omega_a = \frac{d}{da} f_a(\tau) \, d\tau$ is admissible, with $(\frac{d}{da} \omega_a)^\infty = \frac{d}{da} \omega_a^\infty$ and $(\frac{d}{da} \omega_a)^0 = \frac{d}{da} \omega_a^0$. We have:
\begin{equation*}
\int_\tau^\infty \omega_a = \int_\tau^\infty \omega_a^0 + \int_\tau^0 \omega_a^\infty.
\end{equation*}
This shows that $a \mapsto \int_\tau^\infty \omega_a$ is differentiable, and we can differentiate inside the integral:
\begin{equation*}
\frac{d}{da} \Big( \int_\tau^\infty \omega_a \Big) = \int_\tau^\infty \Big(\frac{d}{da} \omega_a\Big)^0 + \int_\tau^0 \Big(\frac{d}{da} \omega_a\Big)^\infty = \int_\tau^\infty \frac{d}{da} \omega_a. \qedhere
\end{equation*}
\end{proof}

Proposition \ref{pro reg int param} motivates calling a real-parameter family $(\omega_a)_a$ of admissible differential forms on $\imaxis$ differentially admissible at $\infty$ if they are subject to conditions (i)--(iii) above, written as $\omega_a = f_a(\tau) \,d\tau$.
Furthermore, we call a real-parameter family $(\omega_a)_a$ differentially admissible at $0$ if the family $(\omega_a^\sigma)_a$ is differentially admissible at $\infty$; see Definition \ref{def reg 0}.
With these definitions in mind, we apply Proposition \ref{pro reg int param} twice to deduce the following statement.

\begin{pro} \label{gen reg int param}
Let $(\omega_a)_a$ be a family of differentially admissible at $0$ and $\infty$ differential forms depending on a single real parameter $a$. Then the function $a \mapsto \int_0^\infty \omega_a$ is differentiable, and we have
\begin{equation*}
\frac{d}{da} \Big(\int_0^\infty \omega_a \Big) = \int_0^\infty \frac{d}{da} \omega_a.
\end{equation*}
\end{pro}

Observe that Propositions \ref{pro reg int param} and \ref{gen reg int param} cover the iterated integral situation as well, since $f_a(\tau)$ themselves may come as iterated integrals of admissible forms.
For this, we simply apply the propositions inductively using Proposition \ref{pro int 0 infty}.

\subsection{Mellin transforms} \label{sec Mellin}

A powerful analytic tool to compute regularised integrals is the theory of Mellin transforms. Since we consider admissible forms on $\imaxis$ with possible poles at $0$ and $\infty$, we will need generalised Mellin transforms as described in \cite[Section 3.4]{Car92}. We use notably this theory in Section \ref{RZ} to compute integrals of products of two Eisenstein series using the Rogers--Zudilin method.

We enlarge a bit our setting by considering functions $f \colon \imaxis \to \C$ of the form $f(\tau) = f^\infty(\tau) + f^0(\tau)$, where $f^\infty(\tau) \in \C[\tau,\tau^{-1}]$ is a Laurent polynomial, and $f^0$ is a $C^\infty$ function with exponential decay at $i\infty$. Moreover, we assume that $f \circ \sigma(\tau) = f(-1/\tau)$ is also of this form. For such a function $f$, the (generalised) Mellin transform is defined as
\begin{equation*}
\mathcal{M}(f,s) = \int_0^\infty f(iy) y^s \frac{dy}{y} \qquad (s \in \C).
\end{equation*}
In general, this integral may not converge at any $s \in \C$. However, splitting the integral as $\int_0^1 + \int_1^\infty$, and analytically continuing each term, it is possible to make sense of $\mathcal{M}(f,s)$ as a meromorphic function of $s \in \C$, with at most simple poles at finitely many integers. A pole of $\mathcal{M}(f,s)$ can occur at $n_0 \in \Z$ only if $-n_0$ arises as an exponent in the polynomial $f^\infty$, or $n_0$ arises as an exponent in $(f \circ \sigma)^\infty$. As a remark, $\mathcal{M}(f,s)$ is identically zero if $f$ is a Laurent polynomial. Therefore, we can always reduce to the situation where $f^\infty=0$.

For any $s_0 \in \C$, we denote by $\mathcal{M}^*(f,s_0)$ the constant term of the Laurent expansion of $\mathcal{M}(f,s)$ at $s=s_0$. 

From Lemma \ref{lem int df 0 infty}, we get the following computational tool.

\begin{pro}[{\cite[Theorem 3.2.8]{Wan20}}, {\cite[Proposition 1.2.15]{Boi26}}] \label{pro reg int Mellin}
Let $\omega = f(\tau)\, d\tau$ be an admissible differential form on $\imaxis$. Then $\mathcal{M}(f,s)$ is holomorphic at $s=1$, and we have $\int_0^\infty \omega = i \mathcal{M}(f,1)$.
\end{pro}

\section{Multiple modular values} \label{mmvs}

We use the notation $e(z) = e^{2\pi iz}$ for $z \in \C$, so that $q=e(\tau)$ for $\tau \in \h$. For any $\alpha \in \R$, write also $q^\alpha = e(\alpha \tau)$. Introduce the differential operators
\[
\delta = \delta_\tau := \frac{1}{2\pi i} \,\frac{d}{d\tau} = q\frac{d}{dq}
\quad\text{and}\quad
\delta_a = \frac{1}{2\pi i} \,\frac{d}{da}
\]
if $a$ is a real variable.

Recall the Hurwitz zeta function
\[
\zeta(y,s) = \sum_{\substack{n>0\\ n \equiv y \bmod1}} n^{-s} \qquad (y \in \R/\Z, \; \re(s)>1),
\]
and the periodic zeta function
\[
\hat{\zeta}(y,s) = \sum_{n \geq 1} e(ny) n^{-s} \qquad (y \in \R/\Z, \; \re(s)>1).
\]
Some properties of these functions can be found in \cite[Section 2]{Bru17} (see \cite{AIK14} for the proofs). We point out that the relation \cite[eq.~(11)]{Bru17} is incorrect for $n=1$. Indeed, for $x \in \R/\Z$, we have
\begin{equation} \label{eq hat zeta 0}
    \hat{\zeta}(x,0) = \begin{cases} \frac{e(x)}{1-e(x)} & \textrm{if } x \neq 0, \\
    -\frac12 & \textrm{if } x = 0.
    \end{cases}
\end{equation}
This can be shown by differentiating the relation
\begin{equation*}
    \hat{\zeta}(x,1) = \sum_{n=1}^{\infty} \frac{e(nx)}{n} = - \log(1-e(x)).
\end{equation*}

\subsection{Eisenstein series} \label{subsec Eisenstein}

It will be essential to us to view Eisenstein series not only as functions of the modular variable $\tau \in \h$, but also as functions of the elliptic variable $z \in \C/(\Z+\tau\Z)$. To this end, we recall the Eisenstein-Kronecker function, in the notations of Weil \cite[VII, \S 12]{Wei76}.

Let $L$ be a lattice in $\C$, and let $(\omega_1,\omega_2)$ be a basis of $L$ such that $\im(\omega_2/\omega_1)>0$. Then $A(L) := (2\pi i)^{-1}(\overline{\omega}_1 \omega_2 - \omega_1 \overline{\omega}_2)$ is a positive real number which does not depend on the choice of $(\omega_1,\omega_2)$. 

\begin{definition} \label{def EK}
    For an integer $a \geq 0$ and $x, x_0, s \in \C$, introduce the \emph{Kronecker double series}
    \begin{equation*}
        K_a(x,x_0,s;L) = \sum_{\substack{w \in L \\ w \neq -x}} \exp\left(A(L)^{-1}(w \overline{x}_0 - \overline{w} x_0)\right) \frac{(\overline{w}+\overline{x})^a}{|w+x|^{2s}},
    \end{equation*}
    where the sum is extended to all $\omega \in L$, except $\omega=-x$ if $x \in L$. In the case $L = \Z+\tau\Z$ with $\tau \in \h$, we write $K_a(x,x_0,s;\tau)$ or simply $K_a(x,x_0,s)$ when the context is clear.
\end{definition}

The series $K_a(x,x_0,s;L)$ converges for $\re(s)>1+\frac{a}2$. For $a \geq 1$, the function $s \mapsto K_a(x,x_0,s;L)$ extends to a holomorphic function on $\C$ \cite[VII, \S 13]{Wei76}. Moreover, the functions $x \mapsto K_a(x,0,s;L)$ and $x \mapsto K_a(0,x,s;L)$ are periodic with respect to $L$, which justifies the following definition.

\begin{definition} \label{def E Ehat}
    Let $k \geq 1$ be an integer. For $\boldx = (x_1,x_2) \in (\R/\Z)^2$, we define
    \begin{equation*}
        E^{(k)}_\boldx(\tau) = - \frac{(k-1)!}{(-2\pi i)^k} K_k(0,x_1\tau+x_2,k) \qquad \hat{E}^{(k)}_\boldx(\tau) = \frac{(k-1)!}{(-2\pi i)^k} K_k(x_1\tau+x_2,0,k).
    \end{equation*}
\end{definition}

Kato has given in \cite[Section 3]{Kat04} an algebraic interpretation of $\hat{E}^{(k)}_\boldx$ in the case $\boldx \in (\frac{1}{N}\Z/\Z)^2$.

We are particularly interested in the series $E^{(k)}_\boldx$. We will determine its Fourier expansion with respect to $\tau$, and then examine its behaviour with respect to the action of $\SL_2(\Z)$ on $\h$. Finally, we will give a differential property of $E^{(k)}_\boldx$ with respect to the elliptic variable.

\begin{lem} \label{lem Fourier Ekx}
Let $k \geq 1$ be an integer, and $\boldx = (x_1,x_2) \in (\R/\Z)^2$, with $\boldx \neq \boldsymbol0$ in the case $k=2$. We have
\begin{equation} \label{eq Fourier Ekx}
E^{(k)}_\boldx(\tau) = a_0(E^{(k)}_\boldx) - \sum_{\substack{m \geq 1 \\ n \in \R_{>0} \\ n \equiv x_1 \bmod{1}}} e(mx_2) n^{k-1} q^{mn} + (-1)^{k+1} \sum_{\substack{m \geq 1 \\ n \in \R_{>0} \\ n \equiv -x_1 \bmod{1}}} e(-mx_2) n^{k-1} q^{mn},
\end{equation}
with
\begin{align*}
a_0(E^{(1)}_\boldx) & = \begin{cases} 0 & \text{if } x_1 = x_2 = 0, \\
-\frac12 \frac{1+e(x_2)}{1-e(x_2)} & \text{if } x_1 = 0 \text{ and } x_2 \neq 0, \\
\{x_1\} - \frac12 & \text{if } x_1 \neq 0, \end{cases} \\
a_0(E^{(k)}_\boldx) & = \frac{B_k(\{x_1\})}{k} \qquad (k \geq 2),
\end{align*}
where $B_k(t)$ is the $k$-th Bernoulli polynomial and $\{ \,\cdot\, \}$ stands for the fractional part.
\end{lem}

\begin{proof}
    In the case $\boldx$ is an $N$-torsion point in $(\R/\Z)^2$, the Fourier expansions of $E^{(k)}_\boldx$ and $\hat{E}^{(k)}_\boldx$ can be found in \cite[Proposition 3.10]{Kat04}.
    The general case can be handled as in \cite[VII]{Sch74}, we only sketch the details in the case $k \geq 3$. We have
    \begin{equation*}
        \sum_{\substack{(m,n) \in \Z^2 \\ (m,n) \neq (0,0)}} \frac{e(mx_2-nx_1)}{(m\tau+n)^k} = \sum_{n \neq 0} \frac{e(-nx_1)}{n^k} + \sum_{m \geq 1} e(mx_2) S(x_1;m\tau) + (-1)^k \sum_{m \geq 1} e(-mx_2) S(-x_1;m\tau)
    \end{equation*}
    with $S(x;\tau) = \sum_{n \in \Z} e(-nx)(\tau+n)^{-k}$. By \cite[Theorem 4.11]{AIK14}, the first term is
    \[
    \sum_{n \neq 0} \frac{e(-nx_1)}{n^k} = \hat{\zeta}(-x_1,k) + (-1)^k \hat{\zeta}(x_1,k) = (-1)^{k+1} \frac{(2\pi i)^k}{k!} B_k(\{x_1\}).
    \]
    For any $x \in \R$, the function $e(-\tau x) S(x;\tau)$ is invariant under $\tau \mapsto \tau+1$, hence has a Fourier expansion
    \begin{equation*}
        e(-\tau x) S(x;\tau) = \sum_{r \in \Z} c_r(x) e(r\tau).
    \end{equation*}
    The Fourier coefficients $c_r(x)$ can be computed as in \cite[VII]{Sch74} using the Poisson summation formula and the residue theorem, leading to \eqref{eq Fourier Ekx}.
\end{proof}

\begin{lem} \label{lem modularity Ekx}
    Let $k \geq 1$ be an integer, and $\boldx \in (\R/\Z)^2$. For any $\gamma = (\begin{smallmatrix} a & b \\ c & d \end{smallmatrix}) \in \SL_2(\Z)$, we have
    \begin{equation*}
    E^{(k)}_\boldx(\gamma\tau) = (c\tau+d)^k E^{(k)}_{\boldx \gamma}(\tau),
    \end{equation*}
where $\boldx\gamma$ means the right multiplication by $\gamma$ on the row vector $\boldx$.
\end{lem}

\begin{proof}
    Putting $x_0 = x_1\tau+x_2$ and $\alpha = c\tau+d$, this follows from the identity $K_k(0,x_0,k;L) = \alpha^k K_k(0,\alpha x_0, k; \alpha L)$, valid for any lattice $L$ in $\C$.
\end{proof}

Taking $\gamma=-I_2$ in Lemma \ref{lem modularity Ekx}, we see that $E^{(k)}_{-\boldx} = (-1)^k E^{(k)}_\boldx$. Lemma \ref{lem modularity Ekx} also shows that if $\boldx$ is $N$-torsion in $(\R/\Z)^2$, then $E^{(k)}_\boldx$ is a modular form of weight $k$ on $\Gamma(N)$, except when $k=2$ and $\boldx=\boldsymbol0$ (in which case $E^{(2)}_{\boldsymbol0}$ is not holomorphic).

Using Lemmas \ref{lem Fourier Ekx} and \ref{lem modularity Ekx} with $\gamma = (\begin{smallmatrix}
    0 & -1 \\ 1 & 0
\end{smallmatrix})$, we obtain the following admissibility property.

\begin{lem} \label{lem admissibility Ekx}
    For any $k \geq 2$ and $\boldx \in (\R/\Z)^2$, with $\boldx \neq \boldsymbol0$ in the case $k=2$, the differential form $E^{(k)}_\boldx(\tau) \tau^{m-1} d\tau$ is admissible on $\imaxis$ for any integer $1 \leq m \leq k-1$.
\end{lem}

The Eisenstein series $E^{(2)}_\boldx$ are related to the so-called Siegel units as follows. For $\boldx = (x_1,x_2) \in (\R/\Z)^2$, $\boldx \neq \boldsymbol0$, consider the following function on $\h$:
\begin{equation} \label{def gx}
g_\boldx(\tau) = q^{B_2(\{x_1\})/2} \prod_{\substack{n \in \R_{\geq 0} \\ n \equiv x_1 \bmod{1}}} (1-q^n e(x_2)) \prod_{\substack{n \in \R_{>0} \\ n \equiv -x_1 \bmod{1}}} (1-q^n e(-x_2)).
\end{equation}
For $a,b \in \Z$, $(a,b) \not\equiv (0,0) \bmod{N}$, the function $g_{a/N,b/N}$ is none other than the classical Siegel unit $g_{\overline{a},\overline{b}}$ \cite[Section 1]{Kat04}. This function is a $(12N)$-th root of a unit on the modular curve $Y(N)$ over $\Q$, thus defining an element of $\mathcal{O}(Y(N))^\times \otimes \Q$.

We also define, for $\boldx \in (\R/\Z)^2$, $\boldx \neq \boldsymbol0$, the logarithm of $g_\boldx$ by taking the logarithm of the infinite product \eqref{def gx} and specifying the branch:
\begin{equation} \label{def log gx}
\log g_\boldx (\tau) = \pi i B_2(\{x_1\}) \tau + \log(1-e(x_2)) \cdot \un_{x_1 = 0} - \sum_{\substack{m \geq 1 \\ n \in \R_{>0} \\ n \equiv x_1 \bmod{1}}} \frac{e(mx_2)}{m} q^{mn} - \sum_{\substack{m \geq 1 \\ n \in \R_{>0} \\ n \equiv -x_1 \bmod{1}}} \frac{e(-mx_2)}{m} q^{mn},
\end{equation}
where
\begin{equation*}
\log(1-e(x_2)) = - \hat{\zeta}(x_2,1) = \log |1-e(x_2)| + \pi i \Bigl(\{x_2\}-\frac12\Bigr).
\end{equation*}


\begin{lem} \label{lem dlog E2}
    For any $\boldx \in (\R/\Z)^2$, $\boldx \neq \boldsymbol0$, we have $\dlog g_\boldx(\tau) = 2\pi i E^{(2)}_\boldx(\tau) \, d\tau$.
\end{lem}

\begin{proof}
    This follows from comparing the Fourier expansions \eqref{eq Fourier Ekx} and \eqref{def log gx}.
\end{proof}



The Kronecker double series $K_a(x,x_0,s;L)$ satisfies differential equations with respect to the elliptic parameters $x$ and $x_0$ \cite[Lemma 1.4]{BK10}. Similarly, the series $E^{(k)}_\boldx$ satisfies a differential relation with respect to both elliptic and modular parameters, which will be especially important.

\begin{lem} \label{diff property Ek}
For $k \geq 1$, the function $\boldx \mapsto E^{(k)}_\boldx(\tau)$ is smooth on the domain $(\R/\Z)^2 \setminus \{\boldsymbol0\}$. Moreover, we have
\begin{equation} \label{eq diff}
\delta_{x_2} E^{(k+1)}_\boldx(\tau) = \delta_\tau E^{(k)}_\boldx(\tau).
\end{equation}
\end{lem}

\begin{proof}
The Fourier expansion \eqref{eq Fourier Ekx} shows that $\boldx \mapsto E^{(k)}_\boldx(\tau)$ is smooth on the domain $\{x_1 \neq 0\}$. Using Lemma \ref{lem modularity Ekx} with $\gamma = \sigma$, the function is also smooth on $\{x_2 \neq 0\}$, whence the claim.

The identity \eqref{eq diff} follows either by inspecting the Fourier expansions of both sides (using Lemma \ref{lem Fourier Ekx}), or directly from Definition \ref{def E Ehat}.
\end{proof}

We now introduce an interpolated version of the Eisenstein series $G^{(k);N}_\boldx$ defined in \eqref{def GkN}.

\begin{definition} \label{def Gkx}
For an integer $k \geq 1$ and $\boldx = (x_1, x_2) \in (\R/\Z)^2$, define
\begin{equation*}
G^{(k)}_\boldx(\tau) = a_0(G^{(k)}_\boldx) + \Bigg(\sum_{\substack{m, n \in \R_{>0} \\ (m, n) \equiv \boldx \bmod{1}}} + (-1)^k \sum_{\substack{m, n \in \R_{>0} \\ (m, n) \equiv -\boldx \bmod{1}}} \Bigg) m^{k-1} q^{mn}
\end{equation*}
with
\begin{align*}
a_0(G^{(1)}_\boldx) & = \begin{cases}
-B_1(\{x_2\}) & \text{if } x_1 = 0 \text{ and } x_2 \neq 0, \\
-B_1(\{x_1\}) & \text{if } x_1 \neq 0 \text{ and } x_2 = 0, \\
0 & \text{otherwise,}
\end{cases} \\
(k \geq 2) \qquad a_0(G^{(k)}_\boldx) & = \begin{cases}
- \frac{B_k(\{x_1\})}{k} & \text{if } x_2 = 0, \\
0 & \text{if } x_2 \neq 0.
\end{cases}
\end{align*}
\end{definition}

The relation with $G^{(k);N}_\boldx$ is as follows. If $\boldx = (a/N,b/N)$ is an $N$-torsion point in $(\R/\Z)^2$, then
\begin{equation} \label{eq Gk GkN}
    G^{(k)}_\boldx(N\tau) = N^{1-k} G^{(k);N}_{\overline{a},\overline{b}}(\tau).
\end{equation}

\begin{lem} \label{lem diff Gk}
For $k \geq 1$, the function $\boldx \mapsto G^{(k)}_\boldx(\tau)$ is smooth on the domain $(\R/\Z \setminus \{0\})^2$, and we have
\begin{equation*}
\delta_{x_2} G^{(k)}_\boldx(\tau) = \tau G^{(k+1)}_\boldx(\tau).
\end{equation*}
\end{lem}

\begin{proof}
It suffices to consider the domain $0 < x_1, x_2 < 1$. There $G^{(k)}_\boldx$ can be written as
\begin{equation*}
G^{(k)}_\boldx(\tau) = \sum_{m,n \geq 0} (m+x_1)^{k-1} q^{(m+x_1)(n+x_2)} + (-1)^k \sum_{m,n \geq 1} (m-x_1)^{k-1} q^{(m-x_1)(n-x_2)}.
\end{equation*}
Therefore
\begin{equation*}
\delta_{x_2} G^{(k)}_\boldx(\tau) = \sum_{m,n \geq 0} (m+x_1)^k \tau q^{(m+x_1)(n+x_2)} + (-1)^k \sum_{m,n \geq 1} -(m-x_1)^k \tau q^{(m-x_1)(n-x_2)} = \tau G^{(k+1)}_\boldx(\tau). \qedhere
\end{equation*}
\end{proof}

To end this section, we give an explicit formula for the Mellin transform of the Eisenstein series of type $E^{(k)}$ and $G^{(k)}$.

\begin{lem} \label{Mellin Ek Gk}
For any integer $k \geq 1$ and $\boldx = (x_1,x_2) \in (\R/\Z)^2$, with $\boldx \neq \boldsymbol0$ in the case $k=2$, we have
\begin{align}
\label{eq Mellin Ek}
\mathcal{M}(E^{(k)}_\boldx, s) & = (2\pi)^{-s} \Gamma(s) \bigl( - \zeta(x_1, s-k+1) \hat{\zeta}(x_2, s) + (-1)^{k+1} \zeta(-x_1, s-k+1) \hat{\zeta}(-x_2, s) \bigr) \\
 \label{eq Mellin Gk}
 \mathcal{M}(G^{(k)}_\boldx, s) & = (2\pi)^{-s} \Gamma(s) \bigl( \zeta(x_1, s-k+1) \zeta(x_2, s) + (-1)^k \zeta(-x_1, s-k+1) \zeta(-x_2, s) \bigr).
\end{align}
\end{lem}

\begin{proof}
We give the proof for $G^{(k)}_\boldx$, the other case being similar. Writing $G^{(k)}_\boldx(\tau) = \sum_{n \in \R_{\geq 0}} c_n q^n$, we have for $\re(s)$ large enough:
\begin{align*}
\mathcal{M}(G^{(k)}_\boldx, s) & = (2\pi)^{-s} \Gamma(s) \sum_{n \in \R_{>0}} \frac{c_n}{n^s} \\
& = (2\pi)^{-s} \Gamma(s) \Bigg( \sum_{\substack{m_1,m_2 \in \R_{>0} \\ (m_1,m_2) \equiv (x_1,x_2) \bmod{1}}} + (-1)^k \sum_{\substack{m_1,m_2 \in \R_{>0} \\ (m_1,m_2) \equiv (-x_1,-x_2) \bmod{1}}} \Bigg) \frac{1}{m_1^{s-k+1} m_2^s}. \qedhere
\end{align*}
\end{proof}

From the description of the poles of the Mellin transform in Section \ref{sec Mellin}, one can show that the only possible poles of $\mathcal{M}(E^{(k)}_\boldx, s)$ and $\mathcal{M}(G^{(k)}_\boldx, s)$ are located at $s=0$ and $s=k$. For $E^{(k)}_\boldx$ this follows from using Lemma \ref{lem modularity Ekx} with $\gamma = \sigma$, while for $G^{(k)}_\boldx$ this follows from Lemma \ref{lem Gk Hk} below.

\subsection{Multiple modular values}

Recall that if $f$ is a modular form of weight $k \geq 2$ on some finite index subgroup $\Gamma$ of $\SL_2(\Z)$, the differential form $f(\tau) \tau^{m-1} d\tau$ is admissible on $\imaxis$ for any $1 \leq m \leq k-1$ (see Example~\ref{exadmissible}). For any modular forms $f_1,\ldots,f_n$ of respective weights $k_1,\ldots,k_n \geq 2$, and any integers $m_1, \ldots, m_n$ with $1 \leq m_i \leq k_i-1$, the regularised iterated integral
\begin{align} \label{eq def Lambda fi}
\Lambda(f_1, \ldots, f_n; m_1, \ldots, m_n) & = \int_0^\infty f_1(\tau) \tau^{m_1-1} d\tau \ldots f_n(\tau) \tau^{m_n-1} d\tau \\ 
\nonumber & = \int_0^\infty f_1(\tau_1) \tau_1^{m_1-1} d\tau_1 \int_{\tau_1}^\infty f_2(\tau_2) \tau_2^{m_2-1} d\tau_2 \cdots \int_{\tau_{n-1}}^\infty f_n(\tau_n) \tau_n^{m_n-1} d\tau_n
\end{align}
is called a totally holomorphic multiple modular value (MMV) \cite[Section 5]{Bro19}. In the case all $m_i$ are equal to $1$, we simply write $\Lambda(f_1, \ldots, f_n) = \Lambda(f_1, \ldots, f_n; 1, \ldots, 1)$.

In the case $\Gamma = \SL_2(\Z)$, the multiple modular values are periods of the relative completion of the fundamental group of $\mathcal{M}_{1,1}$ \cite{Bro17, Bro19}. In this article, we are particularly interested in the case $\Gamma$ is the principal congruence subgroup $\Gamma(N)$, and all $f_i$ are Eisenstein series of weight $\geq 2$ on $\Gamma(N)$. In this case \eqref{eq def Lambda fi} is called a multiple Eisenstein value (MEV).

\begin{example} \label{MEV diff}
When the Eisenstein series in question are $E^{(k_i)}_{\boldx_i}(\tau)$ with $k_i\ge2$, all $m_i=1$, and allowing continuous parameters $\boldx_i \in (\R/\Z)^2$, we can view the MEV as a function
\[
(\boldx_1,\dots,\boldx_n)\mapsto\Lambda(E^{(k_1)}_{\boldx_1}, \ldots, E^{(k_n)}_{\boldx_n})
\]
which has partial derivatives with respect all elliptic parameters $x_{p1},x_{p2}$ restricted to the interval $(0,1)$ (or to any shift of it by an integer) where $1\le p\le n$.
This follows from viewing the MEV as the iterated integral of a family of differential forms that depend on each such parameter, the forms differentially admissible at both $0$ and $\infty$ as defined in Section \ref{subsec reg int param}.
The diffentiation of the MEV with respect to the $x_2$-component of $\boldx_1,\dots,\boldx_n$ is particularly simple.
When an index $p$ in the range $1\le p\le n$ is fixed, we can apply Proposition \ref{gen reg int param} to the corresponding parameter $x_{p2}$ and then combine the result with Lemma \ref{diff property Ek} and formula~\eqref{NL-gen} to obtain
\begin{align} \label{Lambda diff}
\frac{\partial}{\partial x_{p2}}\Lambda(E^{(k_1)}_{\boldx_1}, \ldots, E^{(k_p)}_{\boldx_p}, \ldots, E^{(k_n)}_{\boldx_n})
&=\Lambda(E^{(k_1)}_{\boldx_1}, \ldots, E^{(k_{p-1})}_{\boldx_{p-1}}, E^{(k_p-1)}_{\boldx_p} E^{(k_{p+1})}_{\boldx_{p+1}}, \ldots, E^{(k_n)}_{\boldx_n})
\\ &\quad
-\Lambda(E^{(k_1)}_{\boldx_1}, \ldots, E^{(k_{p-1})}_{\boldx_{p-1}} E^{(k_p-1)}_{\boldx_p}, E^{(k_{p+1})}_{\boldx_{p+1}}, \ldots, E^{(k_n)}_{\boldx_n}),
\nonumber
\end{align}
with the first term interpreted as $a_0(E^{(k_n-1)}_{\boldx_n})\Lambda(E^{(k_1)}_{\boldx_1}, \ldots, E^{(k_{n-1})}_{\boldx_{n-1}})$ if $p=n$, while the second term is discarded if $p=1$. This formula means that differentiation of $\Lambda(E^{(k_1)}_{\boldx_1}, \ldots, E^{(k_n)}_{\boldx_n})$ with respect to the elliptic parameter $x_{p2}$ reduces the length of the MEV by~1.
\end{example}

\begin{definition} \label{def MEV}
For $\boldx_1,\ldots,\boldx_n \in (\R/\Z)^2$, we define
\begin{align*}
\Lambda(\boldx_1, \ldots, \boldx_n)
& = (2\pi i)^n \Lambda(E^{(2)}_{\boldx_1}, \ldots, E^{(2)}_{\boldx_n}) \\
& = (2\pi i)^n \int_0^\infty E^{(2)}_{\boldx_1}(\tau_1) d\tau_1 \int_{\tau_1}^\infty E^{(2)}_{\boldx_2}(\tau_2) d\tau_2 \cdots \int_{\tau_{n-1}}^\infty E^{(2)}_{\boldx_n}(\tau_n) d\tau_n.
\end{align*}
\end{definition}

We call $\Lambda(\boldx_1, \ldots, \boldx_n)$ a (totally holomorphic) \emph{multiple Eisenstein value} (MEV) of length $n$. In general, we expect $\Lambda(\boldx_1, \ldots, \boldx_n)$ to be a period only when the parameters $\boldx_i$ belong to $(\Q/\Z)^2$. In the sequel, we implicitly identify $(\Z/N\Z)^2$ with a subgroup of $(\R/\Z)^2$ by mapping a pair $(\overline{x}_1,\overline{x}_2)$ to the class of $(x_1/N, x_2/N)$. In this way $\Lambda(\boldx_1, \ldots, \boldx_n)$ makes sense for arguments $\boldx_i$ in $(\Z/N\Z)^2$.

Since $\dlog g_\boldx = 2\pi i E^{(2)}_\boldx(\tau) \, d\tau$, the multiple Eisenstein value can also be written
\begin{equation*}
\Lambda(\boldx_1, \ldots, \boldx_n) = \int_0^\infty \dlog g_{\boldx_1} \dlog g_{\boldx_2} \ldots \dlog g_{\boldx_n}.
\end{equation*}

Recall that $\sigma = (\begin{smallmatrix} 0 & -1 \\ 1 & 0 \end{smallmatrix})$ acts on $\h$. If $\boldx \in (\Z/N\Z)^2$, $\boldx \neq \boldsymbol0$, then $\sigma^* (\dlog g_\boldx) = \dlog g_{\boldx \sigma}$ by \cite[Lemma 1.7(1)]{Kat04}. By continuity, this identity holds for arbitrary $\boldx \neq \boldsymbol0$. Since $\sigma$ reverses the path $\imaxis$, the path reversal formula for iterated integrals gives
\begin{equation} \label{eq sigma action Lambda}
    \Lambda(\boldx_1 \sigma, \ldots, \boldx_n \sigma) = (-1)^n \Lambda(\boldx_n, \ldots, \boldx_1).
\end{equation}

The \emph{single} modular values are essentially the critical $L$-values of a modular form.
In the particular case of an Eisenstein series, these values are computed classically in terms of Bernoulli polynomials.

\begin{pro} \label{pro single Eis}
For any $\boldx = (x_1, x_2) \in (\R/\Z)^2 \setminus \{\boldsymbol0\}$, we have
\begin{equation*}
\Lambda(\boldx) = \begin{cases} 2\pi i \left(\{x_1\}-\frac12\right) \left(\{x_2\}-\frac12\right) & \text{if } x_1, x_2 \neq 0, \\
\log |1-e(x_2)| & \text{if } x_1 = 0, \, x_2 \neq 0, \\
- \log |1-e(x_1)| & \text{if } x_1 \neq 0, \, x_2 = 0.
\end{cases}
\end{equation*}
\end{pro}

Note that the function $\boldx \mapsto \Lambda(\boldx)$ has discontinuities at $\{x_1=0\} \cup \{x_2=0\}$.

\begin{proof}
    Assume first $x_2 \neq 0$. By Proposition \ref{pro reg int Mellin} and Lemma \ref{Mellin Ek Gk}, we have
    \begin{equation*}
        \Lambda(\boldx) = -2\pi \mathcal{M}(E^{(2)}_\boldx, 1) = \hat{\zeta}(x_2,1) \zeta(x_1,0) + \hat{\zeta}(-x_2,1) \zeta(-x_1,0).
    \end{equation*}
By \cite[Proposition 9.3]{AIK14}, we have
    \begin{equation*}
        \zeta(x_1,0) = \begin{cases} \frac12 - \{x_1\} & \textrm{if } x_1 \neq 0, \\
            - \frac12 & \textrm{if } x_1 = 0.
        \end{cases}
    \end{equation*}
Moreover
\begin{equation*}
        \hat{\zeta}(x_2, 1) = \sum_{n \geq 1} \frac{e(nx_2)}{n} = -\log(1-e(x_2)).
\end{equation*}
    The case $x_2=0$ follows by noting that $\Lambda((x_1,0)) = - \Lambda((0,x_1))$ thanks to \eqref{eq sigma action Lambda}.
\end{proof}

With the same method in mind, one can show that for $k \geq 2$ and $m \in \{1, \ldots, k-1\}$, we have
\begin{equation*}
    \Lambda(E^{(k)}_\boldx ; m) = (-1)^{m+1} \frac{B_{k-m}(x_1) B_m(x_2)}{(k-m)m} \qquad (0 < x_1, x_2 < 1).
\end{equation*}



We will also need the above iterated integrals with the Eisenstein series replaced by their real or imaginary parts. For $\boldx \in (\R/\Z)^2 \setminus \{\boldsymbol0\}$, write
\begin{equation*}
    \omega_\boldx^+ = \re(\dlog g_\boldx) = \dlog |g_\boldx|, \qquad \omega_\boldx^- = \im(\dlog g_\boldx) = \darg(g_\boldx).
\end{equation*}
Then for any $\boldx_1, \ldots, \boldx_n \in (\R/\Z)^2 \setminus \{\boldsymbol0\}$ and any sequence of signs $\varepsilon_1,\ldots,\varepsilon_n \in \{\pm\}$, consider the regularised iterated integral
\begin{equation*}
\Lambda^{\varepsilon_1 \ldots \varepsilon_n}(\boldx_1, \ldots, \boldx_n) = \int_0^\infty \omega_{\boldx_1}^{\varepsilon_1} \ldots \omega_{\boldx_n}^{\varepsilon_n}.
\end{equation*}
For example, taking the real and imaginary parts in Proposition \ref{pro single Eis}, we get
\begin{equation} \label{eq Lambda x +-}
    \Lambda^+(\boldx) = 0 \qquad \textrm{and} \qquad \Lambda^-(\boldx) = 2\pi \Bigl(x_1-\frac12\Bigr) \Bigl(x_2-\frac12\Bigr) \qquad (0<x_1,x_2<1).
\end{equation}

As discussed in Example \ref{MEV diff}, the function $(\boldx_1,\ldots,\boldx_n) \mapsto \Lambda(\boldx_1,\ldots,\boldx_n)$ is differentiable on the domain $(\R/\Z \setminus\{0\})^{2n}$ and its partial derivatives with respect to the $x_2$-components of indices $\boldx_1,\ldots,\boldx_n$ can be explicitly computed using equation \eqref{Lambda diff}; we make use of this differentiation in Section~\ref{G-reg}.

\section{A baby case: The $K_2$ regulator and double modular values} \label{baby}

Let $Y(N)$ be the modular curve over $\Q$ of level $N \geq 1$. The cup-products $\{g_\bolda,g_\boldb\}$ of two Siegel units $g_\bolda$ and $g_\boldb$ provide important elements in the $K$-group $K_2^{(2)}(Y(N))$. Let us consider their images under the Beilinson regulator map \cite[2.10]{Kat04}
\begin{equation*}
    K_2^{(2)}(Y(N)) \longrightarrow H^1(Y(N)(\C), \R \cdot i).
\end{equation*}
The regulator of $\{g_\bolda,g_\boldb\}$ is represented by the differential form $i \, \eta(g_\bolda,g_\boldb)$ on $Y(N)(\C)$, where
\begin{equation*}
    \eta(g_\bolda,g_\boldb) = \log |g_\bolda| \darg g_\boldb - \log |g_\boldb| \darg g_\bolda.
\end{equation*}
The regulator integral of $\eta(g_\bolda, g_\boldb)$ along the modular symbol $\{0,i\infty\}$ can be computed in terms of $L$-values at $s=0$ of modular forms of weight $2$ and level $\Gamma(N)$ \cite{Bru16}. Here we show that this regulator integral can be expressed in terms of double Eisenstein values.

\begin{pro} \label{pro K2 regulator}
Let $\bolda,\boldb \in (\Z/N\Z)^2 \setminus \{\boldsymbol0\}$. We have
\begin{equation*}
\int_0^{\infty} \eta(g_\bolda, g_\boldb) = \im \Lambda(\bolda,\boldb) - \Lambda^+(\bolda) \Lambda^-(\boldb) + R_\bolda \Lambda^-(\boldb) - R_\boldb \Lambda^-(\bolda),
\end{equation*}
where $R_\boldx$ is the regularised value of $\log |g_\boldx|$ at $\infty$, obtained from \eqref{def log gx} by taking the real part of the constant term. In the case the coordinates of $\bolda$ and $\boldb$ are non-zero, this simplifies to
\begin{equation*}
\int_0^{\infty} \eta(g_\bolda, g_\boldb) = \im \Lambda(\bolda,\boldb).
\end{equation*}
\end{pro}

\begin{proof}
Recall that $\dlog |g_\boldx|$ and $\darg g_\boldx$ are admissible by Lemmas \ref{lem admissibility Ekx} and \ref{lem dlog E2}, and note that $\log |g_\boldx(\tau)| = R_\boldx - \int_\tau^\infty \dlog |g_\boldx|$ by Lemma \ref{lem int df infty}. Then
\begin{equation*}
\eta(g_\bolda,g_\boldb)(\tau) = \Bigl( R_\bolda - \int_\tau^\infty \dlog |g_\bolda| \Bigr) \darg g_\boldb(\tau) - \Bigl( R_\boldb - \int_\tau^\infty \dlog |g_\boldb| \Bigr) \darg g_\bolda(\tau).
\end{equation*}
This expression shows that the form $\eta(g_\bolda,g_\boldb)$ is admissible at $\infty$. It is also admissible at $0$ since $\sigma^*(g_\boldx) = g_{\boldx \sigma}$ in $\mathcal{O}(Y(N))^\times \otimes \Q$ by \cite[Lemma 1.7(1)]{Kat04}. Integrating from $0$ to $\infty$ and using Proposition \ref{pro int 0 infty}, this gives
\begin{equation*}
\int_0^\infty \eta(g_\bolda,g_\boldb) = R_\bolda \Lambda^-(\boldb) - \Lambda^{-+}(\boldb,\bolda) - R_\boldb \Lambda^-(\bolda) + \Lambda^{-+}(\bolda,\boldb).
\end{equation*}
Using the shuffle relation $\Lambda^-(\boldb) \Lambda^+(\bolda) = \Lambda^{-+}(\boldb,\bolda) + \Lambda^{+-}(\bolda,\boldb)$, we arrive at
\begin{align*}
\int_0^\infty \eta(g_\bolda,g_\boldb) & = R_\bolda \Lambda^-(\boldb) - \Lambda^-(\boldb) \Lambda^+(\bolda) + \Lambda^{+-}(\bolda,\boldb) - R_\boldb \Lambda^-(\bolda) + \Lambda^{-+}(\bolda,\boldb) \\
\nonumber & = \im \Lambda(\bolda,\boldb) - \Lambda^+(\bolda) \Lambda^-(\boldb) + R_\bolda \Lambda^-(\boldb) - R_\boldb \Lambda^-(\bolda).
\qedhere
\end{align*}
\end{proof}

Proposition \ref{pro K2 regulator} could be refined by considering the \emph{integral} regulator of $\{g_\bolda, g_\boldb\}$, which is a class in $H^1(Y(N)(\C),\C/(2\pi i)^2\Q)$; for the definition of this regulator map see \cite[Exercise 7.10, p.~93]{BZ20}. The associated regulator integral should then involve the real part of $\Lambda(\bolda,\boldb)$.

\section{The Goncharov regulator in terms of triple modular values} \label{G-tmv}

In \cite{Bru20} the first author constructed classes $\xi(\bolda,\boldb)$ in $K_4^{(3)}(Y(N))$, for $\bolda,\boldb \in (\Z/N\Z)^2$. Our aim in this section is to express the regulator of $\xi(\bolda,\boldb)$ in terms of triple Eisenstein values. Since we integrate from $0$ to $\infty$, the regulator integral depends on a choice of representative $\tilde{\xi}(\bolda,\boldb)$ of $\xi(\bolda,\boldb)$ in the Goncharov complex $\Gamma(Y(N),3)$. We choose the one given in \cite[Section 5.1]{Bru20}. One consequence of our main formula (Theorem \ref{thm r32 formula}) is that the regulator integral interpolates as a function of $\bolda,\boldb \in (\R/\Z)^2$, at least in the domain where the coordinates of $\bolda$, $\boldb$ and $\bolda+\boldb$ are non-zero.

Let us recall the construction of $\tilde\xi(\bolda,\boldb)$. Let $\bolda,\boldb, \boldc \in (\Z/N\Z)^2$ be such that $\bolda+\boldb+\boldc=\boldsymbol0$. From now on, \emph{we assume that all the coordinates of $\bolda$, $\boldb$ and $\boldc$ are non-zero}. This considerably simplifies the expressions with multiple modular values below.

Following the $K_4$ construction, our starting point to express the Goncharov regulator is the $3$-term relation in $K_2(Y(N)) \otimes \Q$ \cite[Theorem 1.1]{Bru20}:
\begin{equation} \label{eq gagb 3term}
\{g_\bolda, g_\boldb\} + \{g_\boldb, g_\boldc\} + \{g_\boldc, g_\bolda\} = 0.
\end{equation}
This relation lifts in the exterior square of the group of modular units \cite[Theorem 1.2]{Bru20},
\begin{equation} \label{eq triangulation abc}
g_\bolda \wedge g_\boldb + g_\boldb \wedge g_\boldc + g_\boldc \wedge g_\bolda = \sum_i m_i \cdot u_i \wedge (1-u_i) \quad \text{in }\; \Lambda^2 \mathcal{O}(Y(N))^\times \otimes \Q,
\end{equation}
where $u_i$ and $1-u_i$ are certain explicit modular units, and $m_i \in \Q$. Then our cocycle is
\begin{equation*}
\tilde\xi(\bolda,\boldb) := \sum_i m_i \{u_i\}_2 \otimes \frac{g_\boldb}{g_\bolda} \in B_2(\Q(Y(N))) \otimes \mathcal{O}(Y(N))^\times \otimes \Q.
\end{equation*}
For the definition of the group $B_2(F)$ of a field $F$, see \cite[Section 2.2]{Gon02}.

Recall the expression of Goncharov's explicit regulator map $r_3(2)$. Let $D \colon \p^1(\C) \to \R$ be the Bloch--Wigner dilogarithm. For any two functions $f,g$ on a Riemann surface, define the $1$-form
\begin{equation} \label{eq def r32}
r_3(2)(\{f\}_2 \otimes g) = -D(f) \cdot \darg g - \frac13  \log |g| \cdot \alpha((1-f) \wedge f),
\end{equation}
where
\begin{equation*}
\alpha(f_1 \wedge f_2) = - \log |f_1| \dlog |f_2| + \log |f_2| \dlog |f_1|.
\end{equation*}

By linearity using \eqref{eq def r32}, the regulator $1$-form associated to $\tilde\xi(\bolda,\boldb)$ is
\begin{align*}
r_3(2)(\tilde\xi(\bolda,\boldb)) & = \sum_i m_i \Bigl(-D(u_i) \cdot \darg (g_\boldb/g_\bolda) - \frac13  \log |g_\boldb/g_\bolda| \cdot \alpha((1-u_i) \wedge u_i)\Bigr) \\
& = -\Bigl(\sum_i m_i D(u_i)\Bigr) \darg (g_\boldb/g_\bolda) + \frac13  \log |g_\boldb/g_\bolda| \cdot \alpha(g_\bolda \wedge g_\boldb + g_\boldb \wedge g_\boldc + g_\boldc \wedge g_\bolda).
\end{align*}
Let us introduce the following notation for the regulator integral:
\begin{equation*}
\mathcal{G}(\bolda,\boldb) = \int_0^\infty r_3(2)(\tilde\xi(\bolda,\boldb)).
\end{equation*}
By \cite[Corollary 6.4]{Bru20}, this integral is absolutely convergent. To express $\mathcal{G}(\bolda,\boldb)$ as a triple iterated integral, a key idea is to cast the Bloch--Wigner function $D(z)$ as a primitive:
\begin{equation*}
d(D(z)) = \eta(z \wedge (1-z)), \qquad \text{where}\quad \eta(f \wedge g) = \log |f| \darg(g) - \log |g| \darg(f).
\end{equation*}
Then using \eqref{eq triangulation abc} we can write
\begin{equation} \label{eq primitive 3term}
d\Bigl(\sum_i m_i D(u_i)\Bigr) = \sum_i m_i \, \eta(u_i \wedge (1-u_i)) = \eta(g_\bolda \wedge g_\boldb + g_\boldb \wedge g_\boldc + g_\boldc \wedge g_\bolda).
\end{equation}

As we saw in the proof of Proposition \ref{pro K2 regulator}, the right-hand side of \eqref{eq primitive 3term} is an admissible form on $\imaxis$. Moreover, if $u$ is a modular unit such that $1-u$ is also a modular unit, then $\eta(u \wedge (1-u))$ is admissible and Lemma \ref{lem admissible primitive} implies that $D(u)$ is admissible. Actually $D(u(\tau))$ converges as $\tau \to \infty$ since $D$ is continuous on $\p^1(\C)$. So the regularised value of $D(u)$ at $\infty$ is simply $D(u(\infty))$, and Lemma \ref{lem int df infty} tells us that
\begin{equation*}
D(u(\tau)) = D(u(\infty)) - \int_\tau^\infty \eta(u,1-u).
\end{equation*}
Note that the form $D(u) \darg(g_\boldb/g_\bolda)$ is then admissible. Therefore, using \eqref{eq primitive 3term} the regulator integral can be written
\begin{equation*}
\mathcal{G}(\bolda,\boldb) = A_1 + A_2 + A_3,
\end{equation*}
where
\begin{align}
\label{eq Dinf term} A_1 & = - \sum_i m_i D(u_i(\infty)) \int_0^\infty \darg (g_\boldb/g_\bolda), \\
\label{eq A2 term} A_2 & = \int_0^\infty \darg (g_\boldb/g_\bolda)(\tau) \int_{\tau}^\infty \eta(g_\bolda \wedge g_\boldb + g_\boldb \wedge g_\boldc + g_\boldc \wedge g_\bolda), \\
\label{eq A3 term} A_3 & = \frac13 \int_0^\infty \log |g_\boldb/g_\bolda| \cdot \alpha(g_\bolda \wedge g_\boldb + g_\boldb \wedge g_\boldc + g_\boldc \wedge g_\bolda).
\end{align}
Similar arguments show that the integrand of $A_3$ is admissible on $\imaxis$.

\subsection{The $A_1$ term} \label{A1 term}
The explicit form of the triangulation \eqref{eq triangulation abc} is given by \cite[Theorem 3.3]{Bru20}:
\begin{align*}
&   \sum_i m_i \{u_i\}_2 = \frac{1}{N^2} \sum_{\boldx \in (\Z/N\Z)^2} \{u(\boldsymbol0,\boldx,\bolda-\boldx,\boldb+\boldx)\}_2 \\ &\quad
-\frac{1}{4N^4} \sum_{\boldx,\boldy \in (\Z/N\Z)^2} \big( \{u(\boldsymbol0,\bolda,\boldc+2\boldx,\boldy)\}_2 +\{u(\boldsymbol0,\boldc,\boldb+2\boldx,\boldy)\}_2 + \{u(\boldsymbol0,\boldb,\bolda+2\boldx,\boldy)\}_2 \big),
\end{align*}
which simplifies to
 \begin{equation*}
    \sum_i m_i \{u_i\}_2 = \frac{1}{N^2} \sum_{\boldx \in (\Z/N\Z)^2} \{u(\boldsymbol0,\boldx,\bolda-\boldx,\boldb+\boldx)\}_2
\end{equation*}
in the case $N$ is odd. By convention, in the above sums we keep only those terms $u(\boldx,\boldy,\boldz,\boldsymbol t)$ for which $\boldx,\boldy,\boldz,\boldsymbol t$ are distinct in $(\Z/N\Z)^2/\pm 1$. The same convention takes place below.

\begin{lem} \label{lem1 A1 term}
Let $\bolda,\boldb,\boldc \in (\Z/N\Z)^2$ such that $\bolda+\boldb+\boldc=\boldsymbol0$. Assume that all the coordinates of $\bolda,\boldb,\boldc$ are non-zero. Then
\begin{equation*}
\sum_{\boldx \in (\Z/N\Z)^2} D\big(u(\boldsymbol0,\boldx,\bolda-\boldx,\boldb+\boldx)(\infty)\big) = 0.
\end{equation*}
\end{lem}

\begin{proof}
We write $\hat x$ for the representative of $x/N$, where $x\in\Z/N\Z$, on the interval $[0,1)$, so that $\hat x\in\frac1N\Z\cap[0,1)$. 
According to \cite[Proposition 2.4 and eq.~(7)]{Bru20} we have
\begin{equation} \label{eq uabx}
u(\boldsymbol0,\boldx,\bolda-\boldx,\boldb+\boldx)=\frac{\Delta_{\hat\bolda-\hat\boldx}^2}{\Delta_{\hat\bolda}\Delta_{\hat\bolda-2\hat\boldx}}\,\frac{\Delta_{\hat\boldb} \Delta_{\hat\boldb+2\hat\boldx}}{\Delta_{\hat\boldb+\hat\boldx}^2},
\end{equation}
where
\[
\Delta_{u,v}=(-e(-v))^{\lfloor u\rfloor}q^{B_2(\{u\})/2}\big(1-e(v)\bold1_{u\in\Z}+O(q^{1/N})\big)
\quad\text{as}\; q\to0.
\]
We now collect relevant information for determining when the unit \eqref{eq uabx} has order 0 at $\infty$ and what is the corresponding constant term in the latter case.

For $0\le\hat a_1<1$ and $0\le\hat x_1<1$ we have
\[
\ord_q\frac{\Delta_{\hat\bolda-\hat\boldx}^2}{\Delta_{\hat\bolda}\Delta_{\hat\bolda-2\hat\boldx}}=\begin{cases}
-\hat x_1^2 & \text{if}\; 0\le\hat x_1\le\frac12\hat a_1, \\
-(1-\hat x_1)^2+1-\hat a_1 & \text{if}\; \frac12\hat a_1<\hat x_1\le\hat a_1, \\
-\hat x_1^2+\hat a_1  & \text{if}\; \hat a_1<\hat x_1\le\frac12+\frac12\hat a_1, \\
-(1-\hat x_1)^2 & \text{if}\; \frac12+\frac12\hat a_1<\hat x_1<1.
\end{cases}
\]
If moreover $\hat a_1,\hat a_1-\hat x_1,\hat a_1-2\hat x_1 \notin \Z$, we find
\[
\frac{(-e(-\hat a_2+\hat x_2))^{2\lfloor\hat a_1-\hat x_1\rfloor}}{(-e(-\hat a_2))^{\lfloor\hat a_1\rfloor}(-e(-\hat a_2+2\hat x_2))^{\lfloor\hat a_1-2\hat x_1\rfloor}}
=\begin{cases}
1 & \text{if}\; 0\le\hat x_1\le\frac12\hat a_1, \\
-e(-\hat a_2+2\hat x_2) & \text{if}\; \frac12\hat a_1<\hat x_1\le\hat a_1, \\
-e(\hat a_2)  & \text{if}\; \hat a_1<\hat x_1\le\frac12+\frac12\hat a_1, \\
e(2\hat x_2) & \text{if}\; \frac12+\frac12\hat a_1<\hat x_1<1.
\end{cases}
\]
Similarly, for $0\le\hat b_1<1$ and $0\le\hat x_1<1$ we have
\[
\ord_q\frac{\Delta_{\hat\boldb} \Delta_{\hat\boldb+2\hat\boldx}}{\Delta_{\hat\boldb+\hat\boldx}^2}=\begin{cases}
\hat x_1^2 & \text{if}\; 0\le\hat x_1<\frac12-\frac12\hat b_1, \\
(1-\hat x_1)^2-\hat b_1 & \text{if}\; \frac12-\frac12\hat b_1\le\hat x_1<1-\hat b_1, \\
\hat x_1^2-(1-\hat b_1)  & \text{if}\; 1-\hat b_1\le\hat x_1<1-\frac12\hat b_1, \\
(1-\hat x_1)^2 & \text{if}\; 1-\frac12\hat b_1\le\hat x_1<1.
\end{cases}
\]
If moreover $\hat b_1,\hat b_1+\hat x_1,\hat b_1+2\hat x_1 \notin \Z$, we obtain
\[
\frac{(-e(-\hat b_2))^{\lfloor\hat b_1\rfloor}(-e(-\hat b_2-2\hat x_2))^{\lfloor\hat b_1+2\hat x_1\rfloor}}{(-e(-\hat b_2-\hat x_2))^{2\lfloor\hat b_1+\hat x_1\rfloor}}=\begin{cases}
1 & \text{if}\; 0\le\hat x_1<\frac12-\frac12\hat b_1, \\
-e(-\hat b_2-2\hat x_2) & \text{if}\; \frac12-\frac12\hat b_1\le\hat x_1<1-\hat b_1, \\
-e(\hat b_2)  & \text{if}\; 1-\hat b_1\le\hat x_1<1-\frac12\hat b_1, \\
e(-2\hat x_2) & \text{if}\; 1-\frac12\hat b_1\le\hat x_1<1.
\end{cases}
\]

Our sum of interest is
\begin{equation*}
\Sigma(\bolda,\boldb,\boldc) = \sum_{\boldx \in (\Z/N\Z)^2} D\big(u(\boldsymbol0,\boldx,\bolda-\boldx,\boldb+\boldx)(\infty)\big).
\end{equation*}
Notice the following symmetries of the sum: it is invariant under $(\bolda,\boldb,\boldc) \mapsto (-\bolda,-\boldb,-\boldc)$ (as $u(\bolda,\boldb,\boldc,\boldd)$ is defined for indices in $(\Z/N\Z)^2/{\pm 1}$) and it is cyclic invariant.
The latter follows from changing the summation for $u(\boldsymbol0,\boldx,\bolda-\boldx,\boldb+\boldx)=u(\boldsymbol0,-\boldx,-\bolda+\boldx,\boldb+\boldx)$ to the one over $\boldy=\boldb+\boldx$ and using the definition of $u(\bolda,\boldb,\boldc,\boldd)$ as the cross-ratio of Weierstrass $\wp$-functions:
\begin{equation*}
\Sigma(\bolda,\boldb,\boldc)
= \sum_{\boldy \in (\Z/N\Z)^2} D\big(u(\boldsymbol0,\boldb-\boldy,\boldc+\boldy,\boldy)(\infty)\big)
= \sum_{\boldy \in (\Z/N\Z)^2} D\big(u(\boldsymbol0,\boldy,\boldb-\boldy,\boldc+\boldy)(\infty)\big)
= \Sigma(\boldb,\boldc,\bolda).
\end{equation*}
For similar reasons $\Sigma(\bolda,\boldb,\boldc)$ is antisymmetric under transpositions:
\begin{align*}
\Sigma(\bolda,\boldb,\boldc)
&= \sum_{\boldx \in (\Z/N\Z)^2} D\big(1-u(\boldsymbol0,\bolda-\boldx,\boldx,\boldb+\boldx)(\infty)\big)
= -\sum_{\boldx \in (\Z/N\Z)^2} D\big(u(\boldsymbol0,\bolda-\boldx,\boldx,-\boldb-\boldx)(\infty)\big)
\\
&= -\sum_{\boldy \in (\Z/N\Z)^2} D\big(u(\boldsymbol0,\boldy,\bolda-\boldy,\boldc+\boldy)(\infty)\big)
= -\Sigma(\bolda,\boldc,\boldb).
\end{align*}
Recall that $\hat a_1,\hat b_1,\hat c_1$ are the representatives of $a_1/N,b_1/N,c_1/N$ in the interval $(0,1)$.
After possibly replacing $(\bolda,\boldb,\boldc)$ by $(-\bolda,-\boldb,-\boldc)$ we may assume that $\hat a_1+\hat b_1+\hat c_1=1$; furthermore, since our goal is to demonstrate that $\Sigma(\bolda,\boldb,\boldc)=0$, after possibly permuting $\bolda,\boldb,\boldc$ we may assume that $0<\hat a_1\le\hat b_1\le\hat c_1<1$. Then we get
\[
0<\tfrac12\hat a_1<\hat a_1\le\tfrac12-\tfrac12\hat b_1<\tfrac12+\tfrac12\hat a_1\le1-\hat b_1<1-\tfrac12\hat b_1<1,
\]
so that
\begin{align*}
\ord_q u(\boldsymbol0,\boldx,\bolda-\boldx,\boldb+\boldx)
&=\begin{cases}
0 & \text{if}\; 0\le\hat x_1\le\frac12\hat a_1, \\
2\hat x_1-\hat a_1 \ne 0 & \text{if}\; \frac12\hat a_1<\hat x_1\le\hat a_1, \\
\hat a_1 \ne 0 & \text{if}\; \hat a_1<\hat x_1<\frac12-\frac12\hat b_1, \\
\hat a_1-\hat b_1+1-2\hat x_1  & \text{if}\; \frac12-\frac12\hat b_1\le\hat x_1\le\frac12+\frac12\hat a_1, \\
-\hat b_1 \ne 0 & \text{if}\; \frac12+\frac12\hat a_1<\hat x_1<1-\hat b_1, \\
\hat b_1-2(1-\hat x_1) \ne 0 & \text{if}\; 1-\hat b_1\le\hat x_1<1-\frac12\hat b_1, \\
0 & \text{if}\; 1-\frac12\hat b_1\le\hat x_1<1.
\end{cases}
\end{align*}
This means that $\ord_q u(\boldsymbol0,\boldx,\bolda-\boldx,\boldb+\boldx)=0$ iff $\hat x_1\in[0,\frac12\hat a_1]\cup\{\frac12(\hat a_1-\hat b_1+1)\}\cup[1-\frac12\hat b_1,1)$.
Furthermore, the constant term of $u(\boldsymbol0,\boldx,\bolda-\boldx,\boldb+\boldx)$ is equal to~$1$ for $\hat x_1\in[0,\frac12\hat a_1)\cup(1-\frac12\hat b_1,1)$, and it is
\[
\begin{cases}
1/(1-e(\hat a_2-2\hat x_2)) & \text{if}\; \hat x_1=\frac12\hat a_1, \\
e(\hat a_2-\hat b_2-2\hat x_2) & \text{if}\; \hat x_1=\frac12(\hat a_1-\hat b_1+1), \\
1-e(\hat b_2+2\hat x_2) & \text{if}\; \hat x_1=1-\frac12\hat b_1.
\end{cases}
\]
No matter whether these values of $\hat x_1$ are in $\frac{1}{N}\Z$ or not, using the relations $D(1-x) = D(1/x) = -D(x)$ we see that the resulting sums over $\hat x_2 \in \frac{1}{N}\Z/\Z$ vanish. For example,
\begin{equation*}
\sum_{\hat x_2\in\frac1N\Z/\Z}D\big(1/(1-e(\hat a_2-2\hat x_2))\big) = \sum_{\hat x_2\in\frac1N\Z/\Z}D\big(e(\hat a_2-2\hat x_2)\big) = 0,
\end{equation*}
because the latter sum involves pairs of complex conjugate roots of unity, apart from possibly $\pm1$, and $D(\overline x)=-D(x)$.
Therefore, $\Sigma(\bolda,\boldb,\boldc)=0$.
\end{proof}

\begin{lem} \label{lem Dsum aux}
Let $N>1$ be an integer, and let $u \neq 1$ be an $N$-th root of unity. Then
\begin{equation*}
\sum_{v:v^N=1} D\Bigl(\frac{1-v}{1-u}\Bigr) = \frac{N}{2} D(u).
\end{equation*}
\end{lem}

\begin{proof}
We use the $5$-term relation for the Bloch--Wigner dilogarithm with the quintuple $(\infty,0,1,\allowbreak v,u)$:
\begin{equation*}
D(v)  + D\Bigl(\frac{1-u^{-1}}{1-v^{-1}}\Bigr) + D\Bigl(\frac{u-1}{u-v}\Bigr) + D\Bigl(\frac{u}{v}\Bigr) + D\Bigl(\frac{-u}{1-u}\Bigr) = 0.
\end{equation*}
Using the relations $D(1-x) = D(1/x) = D(\bar{x}) = -D(x)$ as well as $u^{-1}=\bar{u}$ and $v^{-1} = \bar{v}$, this can be written
\begin{equation*}
D(v)  + D\Bigl(\frac{1-v}{1-u}\Bigr) + D\Bigl(\frac{1-v}{1-u}\Bigr) + D\Bigl(\frac{u}{v}\Bigr) - D(u) = 0.
\end{equation*}
Summing over $v \ne 1,u$ and using the relation $\sum_{v:v^N = 1} D(v) = 0$, we deduce the required result.
\end{proof}

\begin{lem} \label{lem2 A1 term}
Let $\bolda,\boldc \in (\Z/N\Z)^2$ and the coordinates of $\bolda$ non-zero. Then the double sum
\begin{equation} \label{eq double sum}  
\sum_{\boldx,\boldy \in (\Z/N\Z)^2} D\big(u(\boldsymbol0,\bolda,\boldc+2\boldx,\boldy)(\infty)\big)
\end{equation}
vanishes.
\end{lem}

\begin{proof}
To compute the double sum \eqref{eq double sum} notice that
\[
u(\boldsymbol0,\bolda,\boldz,\boldy)=\frac{\mathcal E(\boldz,\bolda)}{\mathcal E(\boldy,\bolda)}
\]
where $\mathcal E(\boldz,\bolda)=\Delta_{\hat\boldz}^2/(\Delta_{\hat\boldz+\hat\bolda}\Delta_{\hat\boldz-\hat\bolda})$,
and the sum can be rearranged to run over $\boldz,\boldy$.
Notice that this rearrangement affects the summation on $\boldz=(z_1,z_2)$ in the case of even $N$, because it becomes $4$ times a sum over $\boldz\in(\Z/N\Z)^2$ subject to the congruence conditions $z_1\equiv c_1$, $z_2\equiv c_2\bmod2$.
Changing $\bolda$ into $-\bolda$ does not change the modular unit $u(\boldsymbol0,\bolda,\boldc+2\boldx,\boldy)$, hence we can assume that the representative $\hat a_1$ of $a_1/N$ satisfies $0 < \hat a_1 \le \frac12 \le 1-\hat a_1 < 1$.

With $0\le\hat z_1<1$ we obtain
\[
\ord_q\mathcal E(\boldz,\bolda)
=\hat a_1(1-\hat a_1)-\min\{\hat a_1,1-\hat a_1,\hat z_1,1-\hat z_1\}
=\hat a_1(1-\hat a_1)-\min\{\hat a_1,\hat z_1,1-\hat z_1\},
\]
while the leading coefficient of $\mathcal E(\boldz,\bolda)$ is equal to
\[
\begin{cases}
\text{(a)}\;\; -(1-e(\hat z_2))^2/e(\hat z_2-\hat a_2)=4e(\hat a_2)\sin^2(\pi\hat z_2) & \text{if}\; \hat z_1=0, \\    
\text{(b)}\;\; -1/e(\hat z_2-\hat a_2)=-e(\hat a_2)e(-\hat z_2) & \text{if}\; 0<\hat z_1<\hat a_1, \\    
\text{(c)}\;\; -1/e(-\hat z_2-\hat a_2)=-e(\hat a_2)e(\hat z_2) & \text{if}\; 0<1-\hat z_1<\hat a_1, \\    
\text{(d)}\;\; 1/(1-e(\hat z_2-\hat a_2)) & \text{if}\; \hat z_1=\hat a_1<1-\hat a_1=1-\hat z_1, \\    
\text{(e)}\;\; 1/(1-e(-\hat z_2-\hat a_2)) & \text{if}\; 1-\hat z_1=\hat a_1<1-\hat a_1=\hat z_1, \\    
\text{(f)}\;\; 1 & \text{if}\; \hat a_1<\min\{\hat z_1,1-\hat z_1\},
\\
\text{(g)}\;\; \frac12e(\hat a_2)/(\cos(2\pi\hat a_2) - \cos(2\pi\hat z_2)) & \text{if}\; \hat a_1=\hat z_1=\frac12
\end{cases}
\]
(the case $1-\hat a_1<\min\{\hat z_1,1-\hat z_1\}$ is excluded from the consideration because $\hat a_1\le\frac12$).
We now want to control when the terms $\mathcal E(\boldz,\bolda)/\mathcal E(\boldy,\bolda)$ in the sum \eqref{eq double sum} have constant terms, that is, when
\begin{equation} \label{eq min=min}
\min\{\hat a_1,\hat z_1,1-\hat z_1\}
=\min\{\hat a_1,\hat y_1,1-\hat y_1\}.
\end{equation}
For each of these situations, call them $(\text{r}_z)\times(\text{s}_y)$ with $\text{r},\text{s}\in\{\text{a},\dots,\text{g}\}$, we want to compute a related sum of the dilogarithms of the products of corresponding constant terms over $\hat z_2,\hat y_2\in\frac1N\Z/\Z$.
Because of condition \eqref{eq min=min}, the case $(\text{a}_z)$ occurs if only $(\text{a}_y)$ occurs, and vice versa;
the corresponding sum of $D(\sin^2(\pi\hat z_2)/\sin^2(\pi\hat y_2))$ over $\hat z_2,\hat y_2\in\frac1N\Z/\Z$ vanishes, because each term is zero (as $D(x)=0$ for $x\in\R$).
Similarly, the case $(\text{g}_z)$ exclusively pairs up with $(\text{g}_y)$, and the dilogarithm arguments are real-valued for this combination as well, leading to the zero value for the sum in question.
The case $(\text{f}_z)$ may only pair up with $(\text{f}_y)$, in which case the sum of $D(1)$ terms is void, or with $(\text{d}_y)$ or $(\text{e}_y)$.
If one of the latter situations occur, for instance $(\text{d}_y)$, we can write our sum as
\[
\sideset{}{'}\sum_{\hat z_2,\hat y_2\in\frac1N\Z/\Z}D(1-e(\hat y_2-\hat a_2))
= -\sideset{}{'}\sum_{\hat z_2,\hat y_2\in\frac1N\Z/\Z}D(e(\hat y_2-\hat a_2))
= -\sideset{}{'}\sum_{\hat z_2\in\frac1N\Z/\Z}\sum_{\hat t\in\frac1N\Z/\Z}D(e(\hat t)),
\]
where $\sum'_{\hat z_2}$ means that we sum over $\hat z_2$ under the constraint $z_2\equiv c_2\bmod2$ if $N$ is even.
The double sum then vanishes because the sum over $\hat t$ does.
Similarly, the case $(\text{f}_y)$ pairs up with $(\text{f}_z)$ (which we already discussed), or with $(\text{d}_z)$ or $(\text{e}_z)$, and we argue as above using the summation
\[
\sideset{}{'}\sum_{\hat z_2\in\frac1N\Z/\Z}D(e(\pm\hat z_2-\hat a_2))=0
\]
followed from
\begin{equation} \label{eq even and odd sum}
\sum_{\hat t\in\frac2N\Z/\Z}D(e(\hat t))
=\sum_{\hat t\in\frac1N+\frac2N\Z/\Z}D(e(\hat t))
=0
\end{equation}
in the case of even $N$ (because conjugate roots of unity $e(\hat t)$ and $e(-\hat t)$, when different from $\pm1\in\R$, combine).
Furthermore, the situations $(\text{d}_z)\times(\text{d}_y)$, $(\text{d}_z)\times(\text{e}_y)$, $(\text{e}_z)\times(\text{d}_y)$ and $(\text{e}_z)\times(\text{e}_y)$ are all treated with the help of Lemma~\ref{lem Dsum aux} applied to the summation over $\hat y_2\in\frac1N\Z/\Z$ and the external summation $\sum'_{\hat z_2\in\frac1N\Z/\Z}$ is performed on the basis of~\eqref{eq even and odd sum} if $N$ is even.
Finally, the cases $(\text{b}_z)$, $(\text{c}_z)$ may only pair up with $(\text{b}_y)$, $(\text{c}_y)$ in view of condition~\eqref{eq min=min}, and we obtain the sum
\[
\sideset{}{'}\sum_{\hat z_2,\hat y_2\in\frac1N\Z/\Z}D(e(\pm\hat z_2\pm\hat y_2))
\]
for an appropriate choice of both `$\pm$', again a vanishing sum.
\end{proof}

Consequently, Lemmas \ref{lem1 A1 term} and \ref{lem2 A1 term} imply the following.

\begin{pro} \label{pro A1 term}
We have $A_1=0$.
\end{pro}

Though proving that $A_1$ vanishes is surprisingly involved, we do not exclude intrinsic reasons behind this degeneracy.

\subsection{The $A_2$ term} \label{A2 term}

We now deal with the $A_2$ term \eqref{eq A2 term}.

\begin{lem} \label{lem A2 term}
If the coordinates of $\boldx,\boldy,\boldz \in (\Z/N\Z)^2$ are non-zero, then
\begin{equation} \label{eq lem A2 term}
\int_0^\infty \darg g_\boldx(\tau) \int_\tau^\infty \eta(g_\boldy, g_\boldz) = \Lambda^{--+}(\boldx,\boldy,\boldz) - \Lambda^{--+}(\boldx,\boldz,\boldy).
\end{equation}
\end{lem}

\begin{proof}
We expand $\eta(g_\boldy, g_\boldz)$ using just the definition. The form $\dlog |g_\boldy|$ is admissible, so Lemma~\ref{lem int df infty} implies
\begin{equation*}
\log |g_\boldy(\tau_1)| = \log |g_\boldy|(\infty) - \int_{\tau_1}^\infty \dlog |g_\boldy|.
\end{equation*}
Noting that $\log |g_\boldy|(\infty) = 0$ here, this leads to
\begin{equation*}
\eta(g_\boldy, g_\boldz)(\tau_1) = - \darg g_\boldz(\tau_1) \int_{\tau_1}^\infty \dlog |g_\boldy| + \darg g_\boldy(\tau_1) \int_{\tau_1}^\infty \dlog |g_\boldz|,
\end{equation*}
which implies \eqref{eq lem A2 term}.
\end{proof}

Expanding $A_2$ in \eqref{eq A2 term} using Lemma \ref{lem A2 term}, we get $A_2 = I_1 + \cdots + I_6$ with
\begin{align*}
I_1 & = \Lambda^{--+}(\boldb,\bolda,\boldb) - \Lambda^{--+}(\boldb,\boldb,\bolda), & I_2 & = \Lambda^{--+}(\boldb,\boldb,\boldc) - \Lambda^{--+}(\boldb,\boldc,\boldb), \\
I_3 & = \Lambda^{--+}(\boldb,\boldc,\bolda) - \Lambda^{--+}(\boldb,\bolda,\boldc), & I_4 & = - \Lambda^{--+}(\bolda,\bolda,\boldb) + \Lambda^{--+}(\bolda,\boldb,\bolda), \\
I_5 & = - \Lambda^{--+}(\bolda,\boldb,\boldc) + \Lambda^{--+}(\bolda,\boldc,\boldb), & I_6 & = - \Lambda^{--+}(\bolda,\boldc,\bolda) + \Lambda^{--+}(\bolda,\bolda,\boldc).
\end{align*}
To simplify this expression for $A_2$, we use shuffle relations between iterated integrals $\Lambda^{\varepsilon_1 \varepsilon_2 \varepsilon_3}$ with $\varepsilon_1,\varepsilon_2,\varepsilon_3 \in \{\pm 1\}$. Consider the relation
\begin{equation*}
\Lambda^-(\boldx) \Lambda^{-+}(\boldy,\boldz) = \Lambda^{--+}(\boldx,\boldy,\boldz) + \Lambda^{--+}(\boldy,\boldx,\boldz) + \Lambda^{-+-}(\boldy,\boldz,\boldx).
\end{equation*}
Specialising to $\boldz=\boldx$ and $\boldy=\boldx$ respectively, we get
\begin{align}
\label{eq shuffle2} \Lambda^{--+}(\boldx,\boldy,\boldx) & =  - \Lambda^{-+-}(\boldy,\boldx,\boldx) - \Lambda^{--+}(\boldy,\boldx,\boldx) + \Lambda^-(\boldx) \Lambda^{-+}(\boldy,\boldx), \\
\label{eq shuffle3} \Lambda^{-+-}(\boldx,\boldz,\boldx) & = -2 \Lambda^{--+}(\boldx,\boldx,\boldz) + \Lambda^-(\boldx) \Lambda^{-+}(\boldx,\boldz).
\end{align}
Similarly,
\begin{equation*}
\Lambda^-(\boldx) \Lambda^{+-}(\boldz,\boldy) = \Lambda^{-+-}(\boldx,\boldz,\boldy) + \Lambda^{+--}(\boldz,\boldx,\boldy) + \Lambda^{+--}(\boldz,\boldy,\boldx)
\end{equation*}
which, taking $\boldy=\boldx$, specialises to
\begin{equation}
\label{eq shuffle6} \Lambda^{-+-}(\boldx,\boldz,\boldx) = -2 \Lambda^{+--}(\boldz,\boldx,\boldx) + \Lambda^-(\boldx) \Lambda^{+-}(\boldz,\boldx).
\end{equation}
Equating the right-hand sides of \eqref{eq shuffle6} and \eqref{eq shuffle3} gives
\begin{equation} \label{eq shuffle7}
\Lambda^{--+}(\boldx,\boldx,\boldz) = \Lambda^{+--}(\boldz,\boldx,\boldx) + \frac12 \Lambda^-(\boldx) (\Lambda^{-+}(\boldx,\boldz) - \Lambda^{+-}(\boldz,\boldx)).
\end{equation}

We now simplify $I_1,\ldots,I_6$. We introduce the shortcut
\begin{equation*}
\Lambda_1(\boldx,\boldy,\boldz) := (\Lambda^{+--} + \Lambda^{-+-} + \Lambda^{--+})(\boldx,\boldy,\boldz).
\end{equation*}
Note that in this way,
\begin{equation} \label{eq Re Lambda}
\re \Lambda(\boldx,\boldy,\boldz) = -\Lambda_1(\boldx,\boldy,\boldz) + \Lambda^{+++}(\boldx,\boldy,\boldz)
\end{equation}
for any $\boldx,\boldy,\boldz$. Using \eqref{eq shuffle7} and \eqref{eq shuffle2}, we have
\begin{align*}
I_1 & = \Lambda^{--+}(\boldb,\bolda,\boldb) - \Lambda^{--+}(\boldb,\boldb,\bolda) \\
\nonumber & = \Lambda^{--+}(\boldb,\bolda,\boldb) - \Lambda^{+--}(\bolda,\boldb,\boldb) - \frac12 \Lambda^-(\boldb) ( \Lambda^{-+}(\boldb,\bolda)-\Lambda^{+-}(\bolda,\boldb) ) \\
\nonumber & = - \Lambda^{-+-}(\bolda,\boldb,\boldb) - \Lambda^{--+}(\bolda,\boldb,\boldb) + \Lambda^-(\boldb) \Lambda^{-+}(\bolda,\boldb) - \Lambda^{+--}(\bolda,\boldb,\boldb)\\
\nonumber & \qquad - \frac12 \Lambda^-(\boldb) ( \Lambda^{-+}(\boldb,\bolda)-\Lambda^{+-}(\bolda,\boldb) ) \\
\nonumber & = - \Lambda_1(\bolda,\boldb,\boldb) + \Lambda^-(\boldb) \Big(\Lambda^{-+}(\bolda,\boldb) - \frac12 \Lambda^{-+}(\boldb,\bolda) + \frac12 \Lambda^{+-}(\bolda,\boldb) \Big).
\end{align*}
The integrals $I_2$, $I_4$ and $I_6$ are obtained from $I_1$ by simply rearranging the letters:
\begin{align*}
I_2 & = + \Lambda_1(\boldc,\boldb,\boldb) - \Lambda^-(\boldb) \Big(\Lambda^{-+}(\boldc,\boldb) - \frac12 \Lambda^{-+}(\boldb,\boldc) + \frac12 \Lambda^{+-}(\boldc,\boldb) \Big), \\
I_4 & = - \Lambda_1(\boldb,\bolda,\bolda) + \Lambda^-(\bolda) \Big(\Lambda^{-+}(\boldb,\bolda) - \frac12 \Lambda^{-+}(\bolda,\boldb) + \frac12 \Lambda^{+-}(\boldb,\bolda) \Big), \\
I_6 & = + \Lambda_1(\boldc,\bolda,\bolda) - \Lambda^-(\bolda) \Big(\Lambda^{-+}(\boldc,\bolda) - \frac12 \Lambda^{-+}(\bolda,\boldc) + \frac12 \Lambda^{+-}(\boldc,\bolda) \Big).
\end{align*}
It remains to treat the terms $I_3$ and $I_5$, involving permutations of $(\bolda,\boldb,\boldc)$. By the shuffle relations, we have
\begin{align}
\label{eq shuffle8} \Lambda^{--+}(\boldb,\boldc,\bolda) & = \Lambda^-(\boldb) \Lambda^{-+}(\boldc,\bolda) - \Lambda^{--+}(\boldc,\boldb,\bolda) - \Lambda^{-+-}(\boldc,\bolda,\boldb), \\
\label{eq shuffle9} \Lambda^{--+}(\bolda,\boldc,\boldb) & = \Lambda^-(\bolda) \Lambda^{-+}(\boldc,\boldb) - \Lambda^{--+}(\boldc,\bolda,\boldb) - \Lambda^{-+-}(\boldc,\boldb,\bolda).
\end{align}
We also have
\begin{align*}
\Lambda^-(\boldb) \Lambda^{-+}(\bolda,\boldc) & = \Lambda^{--+}(\boldb,\bolda,\boldc) + \Lambda^{--+}(\bolda,\boldb,\boldc) + \Lambda^{-+-}(\bolda,\boldc,\boldb), \\
\nonumber \Lambda^-(\bolda) \Lambda^{+-}(\boldc,\boldb) & = \Lambda^{-+-}(\bolda,\boldc,\boldb) + \Lambda^{+--}(\boldc,\bolda,\boldb) + \Lambda^{+--}(\boldc,\boldb,\bolda),
\end{align*}
and thus
\begin{equation} \label{eq shuffle10}
\Lambda^{--+}(\boldb,\bolda,\boldc) + \Lambda^{--+}(\bolda,\boldb,\boldc) = \Lambda^-(\boldb) \Lambda^{-+}(\bolda,\boldc) - \Lambda^-(\bolda) \Lambda^{+-}(\boldc,\boldb) + \Lambda^{+--}(\boldc,\bolda,\boldb) + \Lambda^{+--}(\boldc,\boldb,\bolda).
\end{equation}
Therefore,
\begin{align*}
I_3 + I_5 & = \eqref{eq shuffle8} + \eqref{eq shuffle9} - \eqref{eq shuffle10} \\
\nonumber & = \Lambda^-(\boldb) \Lambda^{-+}(\boldc,\bolda) - \Lambda^{--+}(\boldc,\boldb,\bolda) - \Lambda^{-+-}(\boldc,\bolda,\boldb) \\
& \quad + \Lambda^-(\bolda) \Lambda^{-+}(\boldc,\boldb) - \Lambda^{--+}(\boldc,\bolda,\boldb) - \Lambda^{-+-}(\boldc,\boldb,\bolda) \\
& \quad - \Lambda^-(\boldb) \Lambda^{-+}(\bolda,\boldc) + \Lambda^-(\bolda) \Lambda^{+-}(\boldc,\boldb) - \Lambda^{+--}(\boldc,\bolda,\boldb) - \Lambda^{+--}(\boldc,\boldb,\bolda) \\
& = - \Lambda_1(\boldc,\boldb,\bolda) - \Lambda_1(\boldc,\bolda,\boldb) \\
& \quad + \Lambda^-(\boldb) \Lambda^{-+}(\boldc,\bolda) + \Lambda^-(\bolda) \Lambda^{-+}(\boldc,\boldb) - \Lambda^-(\boldb) \Lambda^{-+}(\bolda,\boldc) + \Lambda^-(\bolda) \Lambda^{+-}(\boldc,\boldb).
\end{align*}

Putting everything together, we obtain
\begin{align} \label{eq1 A2}
A_2 & = - \Lambda_1(\bolda,\boldb,\boldb) + \Lambda_1(\boldc,\boldb,\boldb) - \Lambda_1(\boldb,\bolda,\bolda) + \Lambda_1(\boldc,\bolda,\bolda) - \Lambda_1(\boldc,\boldb,\bolda) - \Lambda_1(\boldc,\bolda,\boldb) \\
\nonumber & \quad + \Lambda^-(\boldb) \Big(\Lambda^{-+}(\bolda,\boldb) + \Lambda^{-+}(\boldc,\bolda) - \Lambda^{-+}(\bolda,\boldc) - \Lambda^{-+}(\boldc,\boldb) \\
\nonumber & \qquad \qquad \quad - \frac12 \Lambda^{-+}(\boldb,\bolda) + \frac12 \Lambda^{+-}(\bolda,\boldb)+ \frac12 \Lambda^{-+}(\boldb,\boldc) - \frac12 \Lambda^{+-}(\boldc,\boldb) \Big) \\
\nonumber & \quad + \Lambda^-(\bolda) \Big(\Lambda^{-+}(\boldb,\bolda) + \Lambda^{-+}(\boldc,\boldb) + \Lambda^{+-}(\boldc,\boldb) - \Lambda^{-+}(\boldc,\bolda) \\
\nonumber & \qquad \qquad \quad + \frac12 \Lambda^{+-}(\boldb,\bolda) - \frac12 \Lambda^{-+}(\bolda,\boldb) + \frac12 \Lambda^{-+}(\bolda,\boldc) - \frac12 \Lambda^{+-}(\boldc,\bolda) \Big).
\end{align}
The terms involving double modular values can be rewritten using the shuffle relations. In our generic situation when the coordinates of the vectors are non-zero, we have $\Lambda^{-+}(\boldx,\boldy) + \Lambda^{+-}(\boldy,\boldx) = \Lambda^-(\boldx) \Lambda^+(\boldy) = 0$ by \eqref{eq Lambda x +-}. Therefore,
\begin{align*}
& \Lambda^{-+}(\bolda,\boldb) + \Lambda^{-+}(\boldc,\bolda) - \Lambda^{-+}(\bolda,\boldc) - \Lambda^{-+}(\boldc,\boldb) \\
\nonumber & \;\quad - \frac12 \Lambda^{-+}(\boldb,\bolda) + \frac12 \Lambda^{+-}(\bolda,\boldb)+ \frac12 \Lambda^{-+}(\boldb,\boldc) - \frac12 \Lambda^{+-}(\boldc,\boldb) \\
\nonumber &\; = (\Lambda^{-+}(\bolda,\boldb) + \Lambda^{+-}(\bolda,\boldb)) + (\Lambda^{-+}(\boldb,\boldc) + \Lambda^{+-}(\boldb,\boldc)) + (\Lambda^{-+}(\boldc,\bolda) + \Lambda^{+-}(\boldc,\bolda)) \\
\nonumber &\; = \im \Lambda(\bolda,\boldb) + \im \Lambda(\boldb,\boldc) + \im \Lambda(\boldc,\bolda).
\end{align*}
Similarly,
\begin{align*}
& \Lambda^{-+}(\boldb,\bolda) + \Lambda^{-+}(\boldc,\boldb) + \Lambda^{+-}(\boldc,\boldb) - \Lambda^{-+}(\boldc,\bolda) \\
\nonumber & \;\quad + \frac12 \Lambda^{+-}(\boldb,\bolda) - \frac12 \Lambda^{-+}(\bolda,\boldb) + \frac12 \Lambda^{-+}(\bolda,\boldc) - \frac12 \Lambda^{+-}(\boldc,\bolda) \\
\nonumber & \; = - \im \Lambda(\bolda,\boldb) - \im \Lambda(\boldb,\boldc) - \im \Lambda(\boldc,\bolda).
\end{align*}
The resulting three-term expressions $\pm\im \bigl(\Lambda(\bolda, \boldb) + \Lambda(\boldb,\boldc) + \Lambda(\boldc,\bolda) \bigr)$ vanish thanks to \cite[Th\'eor\`eme 12.3.1]{Boi26}, whose proof is based on the 3-term relation \eqref{eq gagb 3term}. This simplifies our finding in \eqref{eq1 A2} to the following form.

\begin{pro} \label{pro A2 term}
Let $\bolda,\boldb,\boldc \in (\Z/N\Z)^2$ such that $\bolda+\boldb+\boldc=\boldsymbol0$, with all the coordinates of $\bolda,\boldb,\boldc$ non-zero. Then
\[
A_2 = - \Lambda_1(\bolda,\boldb,\boldb) + \Lambda_1(\boldc,\boldb,\boldb) - \Lambda_1(\boldb,\bolda,\bolda) + \Lambda_1(\boldc,\bolda,\bolda) - \Lambda_1(\boldc,\boldb,\bolda) - \Lambda_1(\boldc,\bolda,\boldb).
\]
\end{pro}

\subsection{The $A_3$ term}

Finally, we treat the $A_3$ term \eqref{eq A3 term}. We leave the required admissibility properties of the differential forms to the reader.

\begin{lem} \label{lem A3 term}
If the coordinates of $\boldx,\boldy,\boldz \in (\Z/N\Z)^2$ are non-zero, then
\begin{equation} \label{eq lem A3 term}
\int_0^\infty \log |g_\boldx| \, \alpha(g_\boldy, g_\boldz) = - \Lambda^{+++}(\boldz,\boldy,\boldx) - \Lambda^{+++}(\boldz,\boldx,\boldy) + \Lambda^{+++}(\boldy,\boldz,\boldx) + \Lambda^{+++}(\boldy,\boldx,\boldz).
\end{equation}
\end{lem}

\begin{proof}
By definition
\begin{equation} \label{eq lem A3 term 2}
\int_0^\infty \log |g_\boldx| \, \alpha(g_\boldy, g_\boldz) = \int_0^\infty - \dlog |g_\boldz| \cdot \log |g_\boldx| \log |g_\boldy| + \dlog |g_\boldy| \cdot \log |g_\boldx| \log |g_\boldz|.
\end{equation}
Recall that the regularised value of the various $\log |g_\boldx|$ at $\infty$ is zero. Therefore
\begin{align*}
\log |g_\boldx(\tau)| \log |g_\boldy(\tau)| & = (\log |g_\boldx| \log |g_\boldy|)(\infty) - \int_\tau^\infty d(\log |g_\boldx| \log |g_\boldy|) \\
& = - \int_\tau^{\infty} \dlog |g_\boldy| \cdot \log |g_\boldx| + \dlog |g_\boldx| \cdot \log |g_\boldy| \\
& = \int_\tau^{\infty} \dlog |g_\boldy| \dlog |g_\boldx| + \dlog |g_\boldx| \dlog |g_\boldy|,
\end{align*}
so that \eqref{eq lem A3 term 2} continues as
\begin{align*}
\int_0^\infty \log |g_\boldx| \, \alpha(g_\boldy, g_\boldz) & = - \int_0^\infty \dlog |g_\boldz(\tau)| \int_\tau^{\infty} \dlog |g_\boldy| \dlog |g_\boldx| + \dlog |g_\boldx| \dlog |g_\boldy| \\
& \quad + \int_0^\infty \dlog |g_\boldy(\tau)| \int_\tau^{\infty} \dlog |g_\boldz| \dlog |g_\boldx| + \dlog |g_\boldx| \dlog |g_\boldz|. \qedhere
\end{align*}
\end{proof}

Using Lemma \ref{lem A3 term}, the term $A_3$ can be written as a sum of six expressions of type \eqref{eq lem A3 term}:
\begin{align} \label{eq A3 1}
3 A_3 & = - \Lambda^{+++}(\boldb,\bolda,\boldb) - \Lambda^{+++}(\boldb,\boldb,\bolda) + \Lambda^{+++}(\bolda,\boldb,\boldb) + \Lambda^{+++}(\bolda,\boldb,\boldb) \\
\nonumber & \quad - \Lambda^{+++}(\boldc,\boldb,\boldb) - \Lambda^{+++}(\boldc,\boldb,\boldb) + \Lambda^{+++}(\boldb,\boldc,\boldb) + \Lambda^{+++}(\boldb,\boldb,\boldc) \\
\nonumber & \quad - \Lambda^{+++}(\bolda,\boldc,\boldb) - \Lambda^{+++}(\bolda,\boldb,\boldc) + \Lambda^{+++}(\boldc,\bolda,\boldb) + \Lambda^{+++}(\boldc,\boldb,\bolda) \displaybreak[2]\\
\nonumber & \quad + \Lambda^{+++}(\boldb,\bolda,\bolda) + \Lambda^{+++}(\boldb,\bolda,\bolda) - \Lambda^{+++}(\bolda,\boldb,\bolda) - \Lambda^{+++}(\bolda,\bolda,\boldb) \\
\nonumber & \quad + \Lambda^{+++}(\boldc,\boldb,\bolda) + \Lambda^{+++}(\boldc,\bolda,\boldb) - \Lambda^{+++}(\boldb,\boldc,\bolda) - \Lambda^{+++}(\boldb,\bolda,\boldc) \\
\nonumber & \quad + \Lambda^{+++}(\bolda,\boldc,\bolda) + \Lambda^{+++}(\bolda,\bolda,\boldc) - \Lambda^{+++}(\boldc,\bolda,\bolda) - \Lambda^{+++}(\boldc,\bolda,\bolda).
\end{align}
Using the shuffle relations
\begin{equation*}
    0 = \Lambda^+(\boldx) \Lambda^{++}(\boldy,\boldz) = \Lambda^{+++}(\boldx,\boldy,\boldz) + \Lambda^{+++}(\boldy,\boldx,\boldz) + \Lambda^{+++}(\boldy,\boldz,\boldx),
\end{equation*}
the six lines in \eqref{eq A3 1} can be simplified, respectively, to
\begin{align*}
3 \Lambda^{+++}(\bolda,\boldb,\boldb), & & -3 \Lambda^{+++}(\boldc,\boldb,\boldb), & & 2 \Lambda^{+++}(\boldc,\bolda,\boldb) + \Lambda^{+++}(\boldc,\boldb,\bolda), \\
3 \Lambda^{+++}(\boldb,\bolda,\bolda), & & 2 \Lambda^{+++}(\boldc,\boldb,\bolda) + \Lambda^{+++}(\boldc,\bolda,\boldb), & & -3 \Lambda^{+++}(\boldc,\bolda,\bolda).
\end{align*}
In this way we obtain the following expression for $A_3$.

\begin{pro} \label{pro A3 term}
Let $\bolda,\boldb,\boldc \in (\Z/N\Z)^2$ such that $\bolda+\boldb+\boldc=\boldsymbol0$, with all the coordinates of $\bolda,\boldb,\boldc$ non-zero. Then
\begin{equation*}
A_3 = \Lambda^{+++}(\bolda,\boldb,\boldb) - \Lambda^{+++}(\boldc,\boldb,\boldb) + \Lambda^{+++}(\boldc,\bolda,\boldb) + \Lambda^{+++}(\boldc,\boldb,\bolda) + \Lambda^{+++}(\boldb,\bolda,\bolda) - \Lambda^{+++}(\boldc,\bolda,\bolda).
\end{equation*}
\end{pro}

Putting together Propositions \ref{pro A1 term}, \ref{pro A2 term} and \ref{pro A3 term}, we obtain an expression for $\mathcal{G}(\bolda,\boldb)$. The terms of type $\Lambda_1$ and $\Lambda^{+++}$ collect thanks to \eqref{eq Re Lambda}. This results in the following final formula.

\begin{thm} \label{thm r32 formula}
Let $\bolda,\boldb,\boldc \in (\Z/N\Z)^2$ such that $\bolda+\boldb+\boldc=\boldsymbol0$. Assume that all the coordinates of $\bolda$, $\boldb$ and $\boldc$ are non-zero. Then
\[
\mathcal{G}(\bolda,\boldb)
= \re \bigl( \Lambda(\bolda,\boldb,\boldb) - \Lambda(\boldc,\boldb,\boldb) + \Lambda(\boldb,\bolda,\bolda) - \Lambda(\boldc,\bolda,\bolda) + \Lambda(\boldc,\boldb,\bolda) + \Lambda(\boldc,\bolda,\boldb)\big).
\]
\end{thm}

\begin{cor}
The Goncharov regulator $\mathcal{G}(\bolda,\boldb)$ interpolates as a differentiable function of $\bolda, \boldb$ in the domain $\{(\bolda,\boldb) \in (\R/\Z)^4 : a_k, b_k, a_k+b_k \neq 0 \;\text{for}\; k=1,2\}$.
\end{cor}

\section{The Borisov--Gunnells relations} \label{BG-rels}

Borisov and Gunnells \cite{BG03} have shown that certain pairwise products of Eisenstein series on the group $\Gamma_1(N)$ satisfy linear dependence relations, which strikingly resemble the Manin 3-term relations for modular symbols. We show in Theorem \ref{thm BG Ek} below an explicit version of the result of Borisov and Gunnells \cite[Theorem 6.2]{BG03} in weight 3 and for Eisenstein series on $\Gamma(N)$. We then deduce in Theorems~\ref{thm BG Gk 1} and \ref{thm BG Gk 2} similar relations for Eisenstein series with rational Fourier coefficients.

\begin{thm} \label{thm BG Ek}
Let $\boldx,\boldy,\boldz \in (\R/\Z)^2 \setminus \{\boldsymbol0\}$ such that $\boldx+\boldy+\boldz=\boldsymbol0$. Then
\begin{equation} \label{eq BG Gk}
E^{(1)}_\boldz E^{(2)}_\boldy - E^{(1)}_\boldy E^{(2)}_\boldx - E^{(1)}_\boldz E^{(2)}_\boldx + E^{(1)}_\boldy E^{(2)}_\boldz = E^{(3)}_\boldx - \frac12 E^{(3)}_\boldy - \frac12 E^{(3)}_\boldz.
\end{equation}
\end{thm}

\begin{proof}
Our starting point is an addition formula due to Weil \cite[IV, \S2, eq.~(10)]{Wei76}:
\begin{multline} \label{eq addition formula Weil}
    (E_2^*(x) - E_2^*(x'))(E_1^*(x+x')-E_1^*(x)-E_1^*(x'))+E_3^*(x)-E_3^*(x')=0 \\
    (x,x' \in \C/(\Z+\tau\Z), \; x,x',x+x' \neq 0),
\end{multline}
where, in Weil's notations, $E_k^*(x) = K_k(x,0,k;\tau)$. (Weil states the identity for series denoted by $E_k(x)$, but they can be expressed in terms of $E_k^*(x)$ \cite[VI, \S 2]{Wei76}.) In terms of $\hat{E}_\boldx$ (see Definition \ref{def E Ehat}), the identity \eqref{eq addition formula Weil} can be rewritten
\begin{equation} \label{eq Weil Ehat}
(\hat{E}^{(1)}_\bolda + \hat{E}^{(1)}_\boldb + \hat{E}^{(1)}_{-\bolda-\boldb})(\hat{E}^{(2)}_\bolda - \hat{E}^{(2)}_{-\bolda-\boldb}) = \frac12 (\hat{E}^{(3)}_\bolda - \hat{E}^{(3)}_{-\bolda-\boldb}) \qquad (\bolda,\boldb \in (\R/\Z)^2, \; \bolda,\boldb,\bolda+\boldb \neq \boldsymbol0).
\end{equation}
Our original source of \eqref{eq addition formula Weil} was a nice geometric interpretation given by Khuri-Makdisi \cite[eq.~(4.39)]{Khu12}: this identity expresses the slope of the line passing through 3 points $P,Q,R$ on $\mathcal{E}_\tau = \C/(\Z+\tau\Z)$, where $P+Q+R=0$. Another proof is given in \cite[p.~177--178]{KR17}.

Now, after restricting to $N$-torsion points, the Eisenstein series $E^{(k)}_{x_1,x_2}$ is essentially the discrete Fourier transform of $\hat{E}^{(k)}_{a_1,a_2}$. In the sequel, we implicitly identify $(\Z/N\Z)^2$ with a subset of $(\R/\Z)^2$ by mapping a pair $(\overline{x}_1,\overline{x}_2)$ to the class of $(x_1/N, x_2/N)$. Moreover, let us introduce the Weil pairing on $\mathcal{E}_\tau[N] \cong (\Z/N\Z)^2$:
\begin{equation*}
e_N \colon (\Z/N\Z)^2 \times (\Z/N\Z)^2 \to \C^\times, \qquad (\bolda,\boldx) \mapsto e\left(\frac{a_2 x_1 - a_1 x_2}{N}\right).
\end{equation*}
For $k \geq 1$, the relation between $E^{(k)}$ and $\hat{E}^{(k)}$ is
\begin{equation*}
\sum_{\bolda \in (\Z/N\Z)^2} e_N(\bolda,\boldx) \hat{E}^{(k)}_\bolda = (-1)^{k+1} N^k E^{(k)}_\boldx \qquad \left(\boldx \in (\Z/N\Z)^2\right).
\end{equation*}
This can be proved directly from the definitions of $E^{(k)}_\bolda$ and $\hat{E}^{(k)}_\boldx$ (Definition \ref{def E Ehat}).

This leads us to taking the Fourier transform of \eqref{eq Weil Ehat} with respect to both $\bolda$ and $\boldb$. However, it is important to note that \eqref{eq Weil Ehat} only holds when $\bolda$, $\boldb$ and $\bolda+\boldb$ are non-zero. For example, when $\bolda=\boldsymbol0$, the left-hand side of \eqref{eq Weil Ehat} is zero while the right-hand side may not be. We thus take the Fourier transform of both sides of \eqref{eq Weil Ehat} separately, and then use the inclusion-exclusion principle:
\begin{equation} \label{eq in ex}
\sum_{\substack{\bolda,\boldb \in (\Z/N\Z)^2 \\ \bolda, \boldb, \bolda+\boldb \neq \boldsymbol0}} = \sum_{\bolda,\boldb \in (\Z/N\Z)^2} - \sum_{\substack{\bolda = \boldsymbol0 \\ \boldb \in (\Z/N\Z)^2}} - \sum_{\substack{\bolda \in (\Z/N\Z)^2 \\ \boldb = \boldsymbol0}} - \sum_{\substack{\bolda \in (\Z/N\Z)^2 \\ \boldb = -\bolda}} + 2 \sum_{\bolda=\boldb=\boldsymbol0}.
\end{equation}
We will denote by $\mathcal{L}_{\bolda,\boldb}$ the left-hand side of \eqref{eq Weil Ehat}, and by $\mathcal{R}_{\bolda,\boldb}$ its right-hand side. Let $\boldx,\boldy,\boldz \in (\Z/N\Z)^2$ be as in the statement of Theorem \ref{thm BG Ek}. Noting that $\mathcal{L}_{\bolda,\boldb}$ is zero when $\bolda$, $\boldb$ or $\bolda+\boldb$ is zero, we have
\begin{align} \label{eq FT Lab}
& \sum_{\substack{\bolda,\boldb \in (\Z/N\Z)^2 \\ \bolda, \boldb, \bolda+\boldb \neq \boldsymbol0}} e_N(\bolda,\boldx) e_N(-\boldb,\boldy) \times \mathcal{L}_{\bolda,\boldb} \\
\nonumber & = \sum_{\substack{\bolda,\boldb \in (\Z/N\Z)^2}} e_N(\bolda,\boldx) e_N(-\boldb,\boldy) (\hat{E}^{(1)}_\bolda + \hat{E}^{(1)}_\boldb + \hat{E}^{(1)}_{-\bolda-\boldb})(\hat{E}^{(2)}_\bolda - \hat{E}^{(2)}_{-\bolda-\boldb}) \\
\nonumber & = \sum_{\substack{\bolda,\boldb \in (\Z/N\Z)^2}} e_N(\bolda,\boldx) e_N(-\boldb,\boldy) (- \hat{E}^{(1)}_\bolda \hat{E}^{(2)}_{-\bolda-\boldb} + \hat{E}^{(1)}_\boldb \hat{E}^{(2)}_\bolda - \hat{E}^{(1)}_\boldb \hat{E}^{(2)}_{-\bolda-\boldb} + \hat{E}^{(1)}_{-\bolda-\boldb} \hat{E}^{(2)}_\bolda) \\
\nonumber & = N^3 \left(E^{(1)}_{\boldx+\boldy} E^{(2)}_\boldy + E^{(1)}_\boldy E^{(2)}_\boldx - E^{(1)}_{\boldx+\boldy} E^{(2)}_\boldx - E^{(1)}_\boldy E^{(2)}_{\boldx+\boldy}\right).
\end{align}
We compute the Fourier transform of $\mathcal{R}_{\bolda,\boldb}$ similarly, keeping in mind the correction terms \eqref{eq in ex}:
\begin{align}
\label{eq FT Rab} 
& \sum_{\substack{\bolda,\boldb \in (\Z/N\Z)^2 \\ \bolda, \boldb, \bolda+\boldb \neq \boldsymbol0}} e_N(\bolda,\boldx) e_N(-\boldb,\boldy) \times \mathcal{R}_{\bolda,\boldb} \\
\nonumber & = - \frac12  \Biggl(\sum_{\substack{\bolda = \boldsymbol0 \\ \boldb \in (\Z/N\Z)^2}} + \sum_{\substack{\bolda \in (\Z/N\Z)^2 \\ \boldb = \boldsymbol0}} + \sum_{\substack{\bolda \in (\Z/N\Z)^2 \\ \boldb = -\bolda}}\Biggr) e_N(\bolda,\boldx) e_N(-\boldb,\boldy) (\hat{E}^{(3)}_\bolda - \hat{E}^{(3)}_{-\bolda-\boldb}) \\
\nonumber & = N^3 \left( - E^{(3)}_\boldx + \frac12 E^{(3)}_\boldy - \frac12 E^{(3)}_{\boldx+\boldy} \right).
\end{align}
The identity \eqref{eq BG Gk} now follows from \eqref{eq FT Lab}, \eqref{eq FT Rab} and the relation $E^{(k)}_{\boldx+\boldy} = (-1)^k E^{(k)}_\boldz$.

So far we have established the result for $N$-torsion points. Since both sides of the identity are continuous in $\boldx,\boldy,\boldz$ by Lemma \ref{diff property Ek}, the result is true in general.
\end{proof}

Let us now consider Eisenstein series with rational Fourier coefficients, and investigate the Borisov--Gunnell type relations for them. As the following lemma shows, $G^{(k);N}_\boldx$ is essentially the partial Fourier transform of $E^{(k)}_\boldx$ with respect to the second parameter.

\begin{lem} \label{lem ekab}
For $x_1,u \in \Z/N\Z$, we have
\begin{equation*}
\sum_{x_2 \in \Z/N\Z} e\left(-\frac{ux_2}{N}\right) E^{(k)}_{x_1,x_2} = - N^{2-k} G^{(k);N}_{x_1,u}.
\end{equation*}
\end{lem}

\begin{proof}
This is a direct computation using the $q$-expansions \eqref{def GkN} and \eqref{eq Fourier Ekx}.
\end{proof}

The Borisov--Gunnells relation for $G^{(k)}_\boldx$ is as follows. We first state the case when the first coordinates are non-zero.

\begin{thm} \label{thm BG Gk 1}
Let $x_1,y_1,u_2,v_2 \in (\R/\Z) \setminus \{0\}$ such that $x_1+y_1, u_2-v_2 \neq 0$. Then
\begin{equation*}
G^{(1)}_{x_1+y_1,u_2} G^{(2)}_{y_1,v_2-u_2} + G^{(1)}_{y_1,v_2} G^{(2)}_{x_1,u_2} - G^{(1)}_{x_1+y_1,v_2} G^{(2)}_{x_1,u_2-v_2} - G^{(1)}_{y_1,v_2-u_2} G^{(2)}_{x_1+y_1,u_2} = 0.
\end{equation*}
\end{thm}

\begin{proof}
As for Theorem \ref{thm BG Ek}, it suffices to treat the case of $N$-torsion points. In this case, the identity takes the form
\begin{equation*}
    G^{(1);N}_{x_1+y_1,u_2} G^{(2);N}_{y_1,v_2-u_2} + G^{(1);N}_{y_1,v_2} G^{(2);N}_{x_1,u_2} - G^{(1);N}_{x_1+y_1,v_2} G^{(2);N}_{x_1,u_2-v_2} - G^{(1);N}_{y_1,v_2-u_2} G^{(2);N}_{x_1+y_1,u_2} = 0
\end{equation*}
with $x_1,y_1,u_2,v_2 \in (\Z/N\Z) \setminus \{0\}$ such that $x_1+y_1, u_2-v_2 \neq 0$. Now the idea is to apply the partial Fourier transform to the identity \eqref{eq BG Gk}. Using Lemma \ref{lem ekab}, the transform of the left-hand side $\mathcal{L}_{x_2,y_2}$ of \eqref{eq BG Gk} can be computed as
\begin{align} \label{eq BG proof}
& \sum_{x_2,y_2 \in \Z/N\Z} e\left(-\frac{u_2 x_2 + v_2 y_2}{N}\right) \times \mathcal{L}_{x_2,y_2} \\
\nonumber & = -N \left(G^{(1);N}_{x_1+y_1,u_2} G^{(2);N}_{y_1,v_2-u_2} + G^{(1);N}_{y_1,v_2} G^{(2);N}_{x_1,u_2} - G^{(1);N}_{x_1+y_1,v_2} G^{(2);N}_{x_1,u_2-v_2} - G^{(1);N}_{y_1,v_2-u_2} G^{(2);N}_{x_1+y_1,u_2}\right).
\end{align}
Moreover, the transform of the right-hand side vanishes, as for example
\begin{equation*}
\sum_{x_2,y_2 \in \Z/N\Z} e\left(-\frac{u_2 x_2 + v_2 y_2}{N}\right) E^{(3)}_{z_1,-x_2-y_2} \stackrel{t=-x_2-y_2}{=} \sum_{x_2,t \in \Z/N\Z} e\left(-\frac{(u_2-v_2) x_2 - v_2 t}{N}\right) E^{(3)}_{z_1,t} = 0
\end{equation*}
thanks to our assumption $u_2-v_2 \neq 0$.
\end{proof}

The case when the first coordinate of $\boldx$, $\boldy$ or $\boldz$ is zero requires special care. We will not state it in general, but content ourselves with the following result.

\begin{thm} \label{thm BG Gk 2}
Let $u_1,u_2 \in \R/\Z$ with $u_1 \neq 0$. Then
\begin{equation*}
G^{(1)}_{u_1,u_2} G^{(2)}_{u_1,-u_2} - G^{(1)}_{u_1,-u_2} G^{(2)}_{u_1,u_2} = G^{(3)}_{0,u_2}.
\end{equation*}
\end{thm}

\begin{proof}
Again, it suffices to show that
\begin{equation*}
G^{(1);N}_{u_1,u_2} G^{(2);N}_{u_1,-u_2} - G^{(1);N}_{u_1,-u_2} G^{(2);N}_{u_1,u_2} = \frac1N G^{(3);N}_{0,u_2} \qquad (u_1,u_2 \in \Z/N\Z, \; u_1 \neq 0).
\end{equation*}
We use \eqref{eq BG Gk} with $x_1=0$, $y_1 = u_1$ and $x_2 \neq 0$. The left-hand side of \eqref{eq BG Gk} is
\begin{equation*}
\mathcal{L}_{x_2, y_2} = -E^{(1)}_{u_1,x_2+y_2} E^{(2)}_{u_1,y_2} + (E^{(1)}_{u_1,x_2+y_2} - E^{(1)}_{u_1,y_2}) E^{(2)}_{0,x_2} + E^{(1)}_{u_1,y_2} E^{(2)}_{u_1,x_2+y_2}.
\end{equation*}
Note that $\mathcal{L}_{x_2, y_2}$ is zero when $x_2=0$. Using \eqref{eq BG proof} with $v_2=0$, we have
\begin{align*}
\sum_{\substack{x_2 \neq 0 \\ y_2 \in \Z/N\Z}} e\left(-\frac{u_2 x_2}{N}\right) \times \mathcal{L}_{x_2,y_2} & = \sum_{x_2, y_2 \in \Z/N\Z} e\left(-\frac{u_2 x_2}{N}\right) \times \mathcal{L}_{x_2,y_2} \\
& = -N \left(G^{(1);N}_{u_1,u_2} G^{(2);N}_{u_1,-u_2} + G^{(1);N}_{u_1,0} G^{(2);N}_{0,u_2} - G^{(1);N}_{u_1,0} G^{(2);N}_{0,u_2} - G^{(1);N}_{u_1,-u_2} G^{(2);N}_{u_1,u_2}\right).
\end{align*}
A similar computation gives
\begin{equation*}
\sum_{\substack{x_2 \neq 0 \\ y_2 \in \Z/N\Z}} e\left(-\frac{u_2 x_2}{N}\right) \times \mathcal{R}_{x_2,y_2} = - G^{(3);N}_{0,u_2}. \qedhere
\end{equation*}
\end{proof}

\section{Differentiating the Goncharov regulator} \label{G-reg}

All elliptic parameters $\boldx=(x_1,x_2)$ etc., $\bolda=(a_1,a_2)$ etc.\ considered below are generic, not hitting the integers.
In order to differentiate our expression for $\mathcal{G}(\bolda, \boldb)$ obtained in Theorem~\ref{thm r32 formula}, we use the differentiation rule \eqref{Lambda diff} for MEVs in the case of length 3. Explicitly, it reads:
\begin{align*}
\frac{\partial}{\partial x_2}\Lambda(\boldx,\boldy,\boldz)
&=(2\pi i)^3\int_{0}^{\infty}\omega_{\boldx;\boldy}^{(1;2)}(\tau_1)\omega_\boldz^{(2)}(\tau_2),
\\
\frac{\partial}{\partial y_2}\Lambda(\boldx,\boldy,\boldz)
&=(2\pi i)^3\int_{0}^{\infty}\omega_\boldx^{(2)}(\tau_1)\omega_{\boldy;\boldz}^{(1;2)}(\tau_2)
-(2\pi i)^3\int_{0}^{\infty}\omega_{\boldx;\boldy}^{(2;1)}(\tau_1)\omega_\boldz^{(2)}(\tau_2),
\\
\frac{\partial}{\partial z_2}\Lambda(\boldx,\boldy,\boldz)
&=2\pi i\,\Lambda(\boldx,\boldy)E^{(1)}_\boldz(\infty)
-(2\pi i)^3\int_{0}^{\infty}\omega_\boldx^{(2)}(\tau_1)\omega_{\boldy;\boldz}^{(2;1)}(\tau_2),
\end{align*}
where from now on we set $\omega_\boldx^{(k)}(\tau)=E^{(k)}_\boldx(\tau)\,d\tau$,
$\omega_{\boldx;\boldy}^{(k;m)}(\tau)=E^{(k)}_\boldx(\tau)E^{(m)}_\boldy(\tau)\,d\tau$.
With this in hand, the $\delta_{a_2}$-derivative of
\[
\Lambda(\bolda,\boldb,\boldb) - \Lambda(\boldc,\boldb,\boldb) + \Lambda(\boldb,\bolda,\bolda) - \Lambda(\boldc,\bolda,\bolda) + \Lambda(\boldc,\boldb,\bolda) + \Lambda(\boldc,\bolda,\boldb)
\]
is as follows:
\begin{align*}
&
(2\pi i)^2\int_{0}^{\infty}\omega_{\bolda;\boldb}^{(1;2)}\omega_\boldb^{(2)}
+(2\pi i)^2\int_{0}^{\infty}\omega_{\boldc;\boldb}^{(1;2)}\omega_\boldb^{(2)}
\\ &\quad
+(2\pi i)^2\int_{0}^{\infty}\omega_\boldb^{(2)}\omega_{\bolda;\bolda}^{(1;2)}
-(2\pi i)^2\int_{0}^{\infty}\omega_{\boldb;\bolda}^{(2;1)}\omega_\bolda^{(2)}
\\ &\quad
+\Lambda(\boldb,\bolda)E^{(1)}_\bolda(\infty)
-(2\pi i)^2\int_{0}^{\infty}\omega_\boldb^{(2)}\omega_{\bolda;\bolda}^{(2;1)}
\\ &\quad
+(2\pi i)^2\int_{0}^{\infty}\omega_{\boldc;\bolda}^{(1;2)}\omega_\bolda^{(2)}
-(2\pi i)^2\int_{0}^{\infty}\omega_\boldc^{(2)}\omega_{\bolda;\bolda}^{(1;2)}
\\ &\quad
+(2\pi i)^2\int_{0}^{\infty}\omega_{\boldc;\bolda}^{(2;1)}\omega_\bolda^{(2)}
-\Lambda(\boldc,\bolda)E^{(1)}_\bolda(\infty)
+(2\pi i)^2\int_{0}^{\infty}\omega_\boldc^{(2)}\omega_{\bolda;\bolda}^{(2;1)}
\\ &\quad
-(2\pi i)^2\int_{0}^{\infty}\omega_{\boldc;\boldb}^{(1;2)}\omega_\bolda^{(2)}
+\Lambda(\boldc,\boldb)E^{(1)}_\bolda(\infty)
-(2\pi i)^2\int_{0}^{\infty}\omega_\boldc^{(2)}\omega_{\boldb;\bolda}^{(2;1)}
\\ &\quad
-(2\pi i)^2\int_{0}^{\infty}\omega_{\boldc;\bolda}^{(1;2)}\omega_\boldb^{(2)}
+(2\pi i)^2\int_{0}^{\infty}\omega_\boldc^{(2)}\omega_{\bolda;\boldb}^{(1;2)}
-(2\pi i)^2\int_{0}^{\infty}\omega_{\boldc;\bolda}^{(2;1)}\omega_\boldb^{(2)}
\displaybreak[2]\\ &\;
=E^{(1)}_\bolda(\infty)\,\big(\Lambda(\boldb,\bolda)-\Lambda(\boldc,\bolda)+\Lambda(\boldc,\boldb)\big)
\\ &\quad
+(2\pi i)^2\int_{0}^{\infty}(\omega_{\bolda;\boldc}^{(1;2)}+\omega_{\boldc;\bolda}^{(1;2)}-\omega_{\bolda;\boldb}^{(1;2)}-\omega_{\boldc;\boldb}^{(1;2)})(\omega_\bolda^{(2)}-\omega_\boldb^{(2)}).
\end{align*}
Recall that
\[
\omega_{\bolda;\boldc}^{(1;2)}+\omega_{\boldc;\bolda}^{(1;2)}-\omega_{\bolda;\boldb}^{(1;2)}-\omega_{\boldc;\boldb}^{(1;2)}
=\big(E^{(1)}_\bolda E^{(2)}_\boldc+E^{(1)}_\boldc E^{(2)}_\bolda-(E^{(1)}_\bolda+E^{(1)}_\boldc)E^{(2)}_\boldb\big)\,d\tau.
\]
The latter expression is subject to the Borisov--Gunnells type relation in weight 3 (Theorem~\ref{thm BG Ek}),
\[
E^{(1)}_\bolda E^{(2)}_\boldc+E^{(1)}_\boldc E^{(2)}_\bolda-(E^{(1)}_\bolda+E^{(1)}_\boldc)E^{(2)}_\boldb
=E^{(3)}_\boldb-\tfrac12E^{(3)}_\bolda-\tfrac12E^{(3)}_\boldc,
\]
implying
\[
\omega_{\bolda;\boldc}^{(1;2)}+\omega_{\boldc;\bolda}^{(1;2)}-\omega_{\bolda;\boldb}^{(1;2)}-\omega_{\boldc;\boldb}^{(1;2)}
=\omega_\boldb^{(3)}-\tfrac12\omega_\bolda^{(3)}-\tfrac12\omega_\boldc^{(3)}.
\]
In addition, the shuffle relations imply
\begin{equation*}
\Lambda(\boldb, \bolda) - \Lambda(\boldc, \bolda) + \Lambda(\boldc, \boldb) = \Lambda(\bolda)\Lambda(\boldb)-\Lambda(\bolda, \boldb) - \Lambda(\boldc, \bolda) + \Lambda(\boldb)\Lambda(\boldc) - \Lambda(\boldb, \boldc).
\end{equation*}
The pairwise products $\Lambda(\boldx) \Lambda(\boldy)$ are real by Proposition \ref{pro single Eis}. On the other hand, $\Lambda(\bolda, \boldb) + \Lambda(\boldb, \boldc) + \Lambda(\boldc, \bolda)$ is also real by \cite[Th\'eor\`eme 12.3.1]{Boi26}. Finally, the constant term of $E^{(1)}_\bolda$ is real. Combining the above derivations, and taking into account that the Goncharov regulator $\mathcal{G}(\bolda, \boldb)$ is real, we finally arrive at
\begin{equation}
\label{eq dGab final} \frac1{2\pi}\frac{\partial}{\partial a_2} \mathcal{G}(\bolda,\boldb)
=-4\pi^2\im\Big(\int_{0}^{\infty}(\omega_\bolda^{(2)}-\omega_\boldb^{(2)})(\omega_\boldb^{(3)}-\tfrac12\omega_\bolda^{(3)}-\tfrac12\omega_\boldc^{(3)})\Big).
\end{equation}

\section{Using the Rogers--Zudilin method} \label{RZ}

To handle the integrals $\int_0^\infty \omega^{(2)}_\boldu \omega^{(3)}_\boldv$ in \eqref{eq dGab final}, we use the Rogers--Zudilin method.

\subsection{The setup}

For weights $\ell \geq k \geq 2$, we want to work out the integral
\begin{equation*}
I^{(k,\ell)}_{\boldu, \boldv} = \int_0^\infty E^{(k)}_\boldu(iy) \widetilde{E}^{(\ell)}_\boldv(iy) \, dy \qquad (\boldu, \boldv \in (\R/\Z)^2)
\end{equation*}
in terms of $L$-values. Here $\widetilde{E}^{(\ell)}_\boldv$ denotes the Eichler integral of $E^{(\ell)}_\boldv$, that is, the unique primitive of $2\pi i E^{(\ell)}_\boldv(\tau) \, d\tau$ whose regularised value at $\infty$ is zero. The function $E^{(k)}_\boldu(\tau) \widetilde{E}^{(\ell)}_\boldv(\tau)$ is admissible, so that $I^{(k,\ell)}_{\boldu, \boldv}$ is well-defined.

Recall the modularity with respect to $\sigma = (\begin{smallmatrix} 0 & -1 \\ 1 & 0 \end{smallmatrix})$ (Lemma \ref{lem modularity Ekx}):
\begin{equation} \label{eq Ek sigma}
E^{(k)}_\boldu(iy) = (E^{(k)}_{\boldu \sigma^{-1}} |_k \sigma)(iy) = (iy)^{-k} E^{(k)}_{u\sigma^{-1}}\bigl(\frac{i}{y}\bigr) = (-i)^k y^{-k} E^{(k)}_{-\boldu \sigma}\bigl(\frac{i}{y}\bigr) = i^k y^{-k} E^{(k)}_{\boldu \sigma}\bigl(\frac{i}{y}\bigr).
\end{equation}

\subsection{The computation}

We have
\begin{equation*}I^{(k,\ell)}_{\boldu,\boldv} = i^k \int_0^\infty E^{(k)}_{\boldu \sigma}\bigl(\frac{i}{y}\bigr) \widetilde{E}^{(\ell)}_\boldv(iy) \, \frac{dy}{y^k}.
\end{equation*}
Write $E^{(k)}_{\boldu\sigma}(i/y) = C_1 + S_1(y)$ and $\widetilde{E}^{(\ell)}_\boldv(iy) = C_2 y + S_2(y)$, where $S_1(y)$, respectively $S_2(y)$, decays exponentially as $y \to 0^+$, respectively $y \to +\infty$. Explicitly,
\begin{align*}
C_1 & = a_0(E^{(k)}_{\boldu\sigma}), \\
C_2 & = -2\pi a_0(E^{(\ell)}_\boldv), \\
S_1(y) & = \sum_{\substack{m_1 \geq 1 \\ n_1 \in \R_{>0}}} (a(m_1) b(n_1) + (-1)^k a(-m_1) b(-n_1)) n_1^{k-1} e^{-2\pi m_1 n_1/y}, \\
S_2(y) & = \sum_{\substack{m_2 \geq 1 \\ n_2 \in \R_{>0}}} (c(m_2) d(n_2) + (-1)^\ell c(-m_2) d(-n_2)) \frac{n_2^{\ell-2}}{m_2} e^{-2\pi m_2 n_2 y},
\end{align*}
where the functions $a,b,c,d \colon \R \to \C$ are defined by
\begin{align*}
a(m) & = -e(-mu_1), & c(m) & = -e(mv_2), \\
b(n) & = \un_{n \equiv u_2 \bmod{1}}, & d(n) & = \un_{n \equiv v_1 \bmod{1}}.
\end{align*}
We can write $I^{(k,\ell)}_{\boldu,\boldv} = T_1 + T_2 + T_3$ with
\begin{align*}
T_1 & = i^k \int_0^\infty S_1(y) S_2(y) \frac{dy}{y^k}, \\
T_2 & = i^k C_1 \int_0^\infty \widetilde{E}^{(\ell)}_\boldv(iy) \frac{dy}{y^k}, \\
T_3 & = i^k C_2 \int_0^\infty E^{(k)}_{\boldu\sigma}\bigl(\frac{i}{y}\bigr) \frac{dy}{y^{k-1}},
\end{align*}
where each term $T_i$ is understood as the regularised value of the corresponding Mellin transform (actually the integral $T_1$ converges exponentially at $0$ and $\infty$). The terms $T_2$ and $T_3$ essentially boil down to $L$-values of Eisenstein series, and will be dealt with later.

We compute $T_1$ using the Rogers--Zudilin method. We first consider the terms $a(m_1) b(n_1)$ and $c(m_2) d(n_2)$ inside the series $S_1$ and $S_2$ respectively:
\begin{align*}
& \int_0^\infty \left(\sum_{\substack{m_1 \geq 1 \\ n_1 \in \R_{>0}}} a(m_1) b(n_1) n_1^{k-1} e^{-2\pi m_1 n_1/y}\right) \left(\sum_{\substack{m_2 \geq 1 \\ n_2 \in \R_{>0}}} c(m_2) d(n_2) \frac{n_2^{\ell-2}}{m_2} e^{-2\pi m_2 n_2 y}\right) \frac{dy}{y^k} \\
&\quad = \sum_{\substack{m_1 \geq 1 \\ n_1 \in \R_{>0}}} \sum_{\substack{m_2 \geq 1 \\ n_2 \in \R_{>0}}} a(m_1) b(n_1) c(m_2) d(n_2) n_1^{k-1} \cdot \frac{n_2^{\ell-2}}{m_2} \int_0^\infty e^{-2\pi (m_2 n_2 y + \frac{m_1 n_1}{y})} \frac{dy}{y^k} \\
&\quad\; \stackrel{y \to \frac{n_1}{m_2} \cdot y}{=} \sum_{\substack{m_1 \geq 1 \\ n_1 \in \R_{>0}}} \sum_{\substack{m_2 \geq 1 \\ n_2 \in \R_{>0}}} a(m_1) b(n_1) c(m_2) d(n_2) n_2^{\ell-2} m_2^{k-2} \int_0^\infty e^{-2\pi (n_1 n_2 y + \frac{m_1 m_2}{y})} \frac{dy}{y^k} \\
&\quad = \int_0^\infty \left(\sum_{m_1, m_2 \geq 1} a(m_1) c(m_2) m_2^{k-2} e^{-2\pi m_1 m_2/y}\right) \left(\sum_{n_1, n_2 \in \R_{>0}} b(n_1) d(n_2) n_2^{\ell-2} e^{-2\pi n_1 n_2 y}\right) \frac{dy}{y^k}.
\end{align*}
This computation will be summarised with the formal transformation $ab \otimes cd \to ac \otimes bd$.

Now the term $T_1$ is a linear combination of four terms, involving substitutions $(m_i, n_i) \to (-m_i, -n_i)$ for $i=1,2$. As a shortcut, write $f^-(x)= f(-x)$ for a function $f \colon \R \to \C$. Then the computation of $T_1$ can be written formally
\begin{align*}
& (ab + (-1)^k a^- b^-) \otimes (cd + (-1)^\ell c^- d^-) \\
& \quad \to ac \otimes bd + (-1)^\ell ac^- \otimes bd^- + (-1)^k a^- c \otimes b^- d + (-1)^{k+\ell} a^- c^- \otimes b^- d^-.    
\end{align*}
This linear combination does not produce Eisenstein series: for example $\sum b(n_1) d(n_2) q^{n_1 n_2}$ has no modularity property, because of the lack of parity conditions in $b$ and $d$. To get Eisenstein series, we have to take the imaginary part of $T_1$; this corresponds to considering the Beilinson regulator map with values in \emph{real} Deligne--Beilinson cohomology. Noting that $\bar{a} = a^-$, $\bar{c} = c^-$, $\bar{b}=b$ and $\bar{d} = d$, we see that $T_1 - \overline{T_1}$ can be computed as
\begin{align*}
& (ab + (-1)^k a^- b^-) \otimes (cd + (-1)^\ell c^- d^-) + (-1)^{k-1} (a^- b + (-1)^k a b^-) \otimes (c^- d + (-1)^\ell c d^-) \\ &\;
\to ac \otimes bd + (-1)^\ell ac^- \otimes bd^- + (-1)^k a^- c \otimes b^- d + (-1)^{k+\ell} a^- c^- \otimes b^- d^- \\
& \;\quad + (-1)^{k-1} a^- c^- \otimes bd + (-1)^{k+\ell-1} a^- c \otimes bd^- - a c^- \otimes b^- d + (-1)^{\ell-1} a c \otimes b^- d^- \\
&\; = (ac + (-1)^{k-1} a^- c^-) \otimes (bd + (-1)^{\ell-1} b^- d^-) + (-1)^k (a^- c + (-1)^{k-1} a c^-) \otimes (b^- d + (-1)^{\ell-1} bd^-).
\end{align*}
Up to the constant terms, we recognise the sum of two pairwise products of Eisenstein series of weights $k-1$ and $\ell-1$, respectively. Denoting by $f^0 = f - a_0(f)$ the rapidly decreasing part of $f$, we have
\begin{equation} \label{eq T1 1}
    \im(T_1) = \frac{i^{-k-1}}{2} \int_0^\infty \left(H^{(k-1),0}_{u_1,v_2}\bigl(\frac{i}{y}\bigr) G^{(\ell-1),0}_{v_1,-u_2}(iy) - H^{(k-1),0}_{u_1,-v_2}\bigl(\frac{i}{y}\bigr) G^{(\ell-1),0}_{v_1,u_2}(iy) \right) \frac{dy}{y^k},
\end{equation}
where for $\boldx = (x_1,x_2) \in (\R/\Z)^2$, the Eisenstein series $H^{(k)}_\boldx$ is given by 
\begin{equation*}
H^{(k)}_\boldx(\tau) = a_0(H^{(k)}_\boldx) + \sum_{m,n \geq 1} \bigl(e(m x_1 + n x_2) + (-1)^k e(-m x_1 - n x_2) \bigr) n^{k-1} q^{mn},
\end{equation*}
with
\begin{align*}
a_0(H^{(1)}_\boldx) & = \begin{cases} 0 & \textrm{if } \boldx = \boldsymbol0, \\
\frac12 \frac{1+e(x_2)}{1-e(x_2)} & \textrm{if } x_1=0 \textrm{ and } x_2 \neq 0,\\
\frac12 \frac{1+e(x_1)}{1-e(x_1)} & \textrm{if } x_1 \neq 0 \textrm{ and } x_2=0,\\
\frac12 \left(\frac{1+e(x_1)}{1-e(x_1)} + \frac{1+e(x_2)}{1-e(x_2)}\right) & \textrm{if } x_1 \neq 0 \textrm{ and } x_2 \neq 0,
\end{cases} \\
(k \geq 2) \qquad a_0(H^{(k)}_\boldx) & = (-1)^k \hat{\zeta}(-x_2,1-k).
\end{align*}

The Eisenstein series $G^{(k)}_\boldx$ and $H^{(k)}_\boldx$ are related as follows.

\begin{lem} \label{lem Gk Hk}
Let $k \geq 1$ and $\boldx = (x_1,x_2) \in (\R/\Z)^2$, with $x_1 \neq 0$ in the case $k=2$. Then $H^{(k)}_\boldx |_k \sigma = G^{(k)}_\boldx$. In particular, we have $H^{(k)}_\boldx(i/y) = (iy)^k G^{(k)}_\boldx(iy)$ for any $y>0$.
\end{lem}

\begin{proof}
In the case $\boldx \in (\frac{1}{N}\Z/\Z)^2$, this is \cite[Lemme 3.10]{Bru17}.
The general case follows since both sides are continuous in $\boldx$.
\end{proof}

We compute \eqref{eq T1 1} by `completing' the Eisenstein series $H^{(k-1)}$ and $G^{(\ell-1)}$, and separating the contribution from the constant terms, using also Lemma \ref{lem Gk Hk}:
\begin{align}
\label{eq T1 2}   & \int_0^\infty H^{(k-1),0}_{\bolda}\bigl(\frac{i}{y}\bigr) G^{(\ell-1),0}_{\boldb}(iy) y^s \frac{dy}{y} \\
\nonumber   &\; = \mathcal{M}\Bigl(H^{(k-1)}_\bolda\bigl(\frac{i}{y}\bigr) G^{(\ell-1)}_\boldb(iy), s \Bigr) - a_0(H^{(k-1)}_{\bolda}) \mathcal{M}(G^{(\ell-1)}_{\boldb}, s ) - a_0(G^{(\ell-1)}_{\boldb}) \mathcal{M} \Bigl(H^{(k-1)}_{\bolda}\bigl(\frac{i}{y}\bigr), s \Bigr) \\ 
\nonumber   &\; = i^{k-1} \mathcal{M}(G^{(k-1)}_\bolda G^{(\ell-1)}_\boldb, s+k-1) \\
\nonumber   & \;\quad - a_0(H^{(k-1)}_{\bolda}) \mathcal{M}(G^{(\ell-1)}_{\boldb}, s) - i^{k-1} a_0(G^{(\ell-1)}_{\boldb}) \mathcal{M}(G^{(k-1)}_{\bolda}, s+k-1).
\end{align}

Putting \eqref{eq T1 1} and \eqref{eq T1 2} together, we get the following formula for the imaginary part of $T_1$:
\begin{align*}
    \im(T_1) & = T'_1 + T'_2 + T'_3, \\
    T'_1 & = - \frac12 \mathcal{M}^*\bigl(G^{(k-1)}_{u_1,v_2} G^{(\ell-1)}_{v_1,-u_2} - G^{(k-1)}_{u_1,-v_2} G^{(\ell-1)}_{v_1,u_2}, 0\bigr), \\
    T'_2 & = \frac{i^{1-k}}{2} a_0(H^{(k-1)}_{u_1,v_2}) \mathcal{M}(G^{(\ell-1)}_{v_1,-u_2}, 1-k) - i^{1-k} a_0(H^{(k-1)}_{u_1,-v_2}) \mathcal{M}(G^{(\ell-1)}_{v_1,u_2}, 1-k), \\
    T'_3 & = \frac12 a_0(G^{(\ell-1)}_{v_1,-u_2}) \mathcal{M}^*(G^{(k-1)}_{u_1,v_2}, 0) - \frac12 a_0(G^{(\ell-1)}_{v_1,u_2}) \mathcal{M}^*(G^{(k-1)}_{u_1,-v_2}, 0).
\end{align*}
If $\boldu$ and $\boldv$ are $N$-torsion (and $u_1,v_1 \neq 0$), the main term $T'_1$ is the (completed) $L$-value of a modular form of weight $k+\ell-2$ and level $\Gamma(N)$ with rational Fourier coefficients.

\subsection{The constant terms}

We henceforth assume that $k=2$, which is enough for our purpose. Also, we put ourselves in the generic situation where the coordinates of $\boldu$ and $\boldv$ are non-zero. In this case, the Eisenstein series appearing in $T'_1$ have no constant term (see Definition \ref{def Gkx}), so that the Mellin transform in $T'_1$ is holomorphic at $s=0$. Moreover $a_0(G^{(\ell-1)}_{v_1,-u_2}) = 0$ and thus $T'_3 = 0$.

Let us compute $\im(T_2)$. The Mellin transform of the Eichler integral $\widetilde{E}^{(\ell)}_\boldv$ is given by
\begin{equation*}
    \mathcal{M}(\widetilde{E}^{(\ell)}_\boldv, s) = \frac{2\pi}{s} \mathcal{M}(E^{(\ell)}_\boldv, s+1) = \frac{2\pi}{s} \Bigl(-\frac{a_0(E^{(\ell)}_\boldv)}{s+1} + \mathcal{M}^*(E^{(\ell)}_\boldv, 0) + O_{s \to -1}(s+1)\Bigr),
\end{equation*}
from which we deduce
\begin{equation*}
T_2 = - a_0(E^{(2)}_{u \sigma}) \mathcal{M}^*(\widetilde{E}^{(\ell)}_\boldv, -1) = \pi B_2(\{u_2\}) \bigl(\mathcal{M}^*(E^{(\ell)}_\boldv, 0) - a_0(E^{(\ell)}_\boldv)\bigr).
\end{equation*}
Using Lemma \ref{Mellin Ek Gk} and equation \eqref{eq hat zeta 0}, this leads to
\begin{align*}
    \im(T_2) & = -\frac{\pi i}{2} B_2(\{u_2\}) \lim_{s \to 0} \Gamma(s) \bigl( -\zeta(v_1, s-\ell+1) + (-1)^\ell \zeta(-v_1, s-\ell+1) \bigr) \bigl( \hat{\zeta}(v_2,s) - \hat{\zeta}(-v_2,s) \bigr) \\
    & = \frac{\pi i}{2} B_2(\{u_2\}) \frac{1+e(v_2)}{1-e(v_2)} \lim_{s \to 0} \Gamma(s) \bigl( \zeta(v_1, s-\ell+1) + (-1)^{\ell+1} \zeta(-v_1, s-\ell+1) \bigr).
\end{align*}
For the term $T_3$, we rewrite it using \eqref{eq Ek sigma}:
\begin{equation*}
    T_3 = -2\pi a_0(E^{(\ell)}_\boldv) \int_0^\infty E^{(2)}_\boldu(iy) y \, dy = -2\pi \frac{B_\ell(\{v_1\})}{\ell} \mathcal{M}^*(E^{(2)}_\boldu, 2).
\end{equation*}
Moreover,
\begin{align*}
    \lim_{\substack{s \to 2 \\ s \in \R}} \im(\mathcal{M}(E^{(2)}_\boldu, s)) & = (2\pi)^{-2} \lim_{\substack{s \to 2 \\ s \in \R}} \bigl(-\zeta(u_1,s-1) \im(\hat{\zeta}(u_2,s)) - \zeta(-u_1,s-1) \im(\hat{\zeta}(-u_2,s)) \bigr) \\
    & = (2\pi)^{-2} \im(\hat{\zeta}(u_2,2)) \times \pi i \frac{1+e(u_1)}{1-e(u_1)}.
\end{align*}
Therefore,
\begin{equation*}
    \im(T_3) = -\frac{i}2 \frac{B_\ell(\{v_1\})}{\ell} \cdot \im(\hat{\zeta}(u_2,2)) \cdot \frac{1+e(u_1)}{1-e(u_1)}.
\end{equation*}

It remains to compute $T'_2$. We have $T'_2 = A + B$ with
\begin{align*}
    A & = \frac{i}4 \frac{1+e(u_1)}{1-e(u_1)} (\mathcal{M}(G^{(\ell-1)}_{v_1,u_2},-1) - \mathcal{M}(G^{(\ell-1)}_{v_1,-u_2},-1)), \\
    B & = -\frac{i}4 \frac{1+e(v_2)}{1-e(v_2)} (\mathcal{M}(G^{(\ell-1)}_{v_1,u_2},-1) + \mathcal{M}(G^{(\ell-1)}_{v_1,-u_2},-1)).
\end{align*}
Let us compute $A$. Using Lemma \ref{Mellin Ek Gk}, we obtain
\begin{align*}
  & \mathcal{M}(G^{(\ell-1)}_{v_1,u_2},-1) - \mathcal{M}(G^{(\ell-1)}_{v_1,-u_2},-1) \\
  &\; = 2\pi \lim_{s \to -1} \Gamma(s) \bigl(\zeta(v_1,s-\ell+2) + (-1)^\ell \zeta(-v_1,s-\ell+2)\bigr) \bigl(\zeta(u_2,s)-\zeta(u_2,-s)\bigr) \\
  &\; = -\frac{4\pi}{\ell} B_\ell(\{v_1\}) \lim_{s \to -1} \Gamma(s) (\zeta(u_2,s)-\zeta(-u_2,s)).
\end{align*}
Now the Hurwitz formula \cite[Theorem 9.4]{AIK14} enables us to write
\begin{equation*}
    \zeta(u_2,s) - \zeta(-u_2,s) = (2\pi)^{s-1} \Gamma(1-s) (e^{-\pi i(1-s)/2}-e^{\pi i(1-s)/2}) (\hat{\zeta}(u_2,1-s)-\hat{\zeta}(-u_2,1-s)),
\end{equation*}
which gives
\begin{equation*}
    \lim_{s \to -1} \Gamma(s) (\zeta(u_2,s) - \zeta(-u_2,s)) = - \frac{1}{2\pi} \im(\hat{\zeta}(u_2,2)).
\end{equation*}
Therefore,
\begin{equation*}
    A = \frac{i}{2} \frac{B_\ell(\{v_1\})}{\ell} \frac{1+e(u_1)}{1-e(u_1)} \im(\hat{\zeta}(u_2,2)).
\end{equation*}
Similarly, the term $B$ is equal to
\begin{equation*}
    B = - \frac{\pi i}{2} \frac{1+e(v_2)}{1-e(v_2)} B_2(\{u_2\}) \lim_{s \to 0} \Gamma(s) \bigl(\zeta(v_1,s-\ell+1) + (-1)^{\ell+1} \zeta(-v_1,s-\ell+1)\bigr).
\end{equation*}
Collecting everything, we see that $\im(T_2)+B=0$ and $\im(T_3)+A=0$. Thus, $\im(I^{(2,\ell)}_{\boldu,\boldv}) = T'_1$, as summarised in the following theorem.

\begin{thm} \label{thm RZ Iuv}
Let $\ell \geq 2$ be an integer, and $u=(u_1,u_2)$, $v=(v_1,v_2)$ in $(\R/\Z)^2$, where all $u_i$ and $v_i$ are non-zero. Then
\begin{equation*}
\im(I^{(2,\ell)}_{\boldu,\boldv}) = - \frac12 \mathcal{M}\bigl(G^{(1)}_{u_1,v_2} G^{(\ell-1)}_{v_1,-u_2} - G^{(1)}_{u_1,-v_2} G^{(\ell-1)}_{v_1,u_2}, 0\bigr).
\end{equation*}
\end{thm}

\section{Getting to the $L$-value} \label{L-value}

In Section \ref{G-reg}, we established that the $a_2$-derivative of the (interpolated) Goncharov regulator of $\tilde{\xi}(\bolda,\boldb)$ is
\begin{equation*}
\frac{\partial}{\partial a_2} \mathcal{G}(\bolda,\boldb) = -4 \pi^2 \im \Big(\int_0^\infty\big(E^{(2)}_\boldb(\tau)-E^{(2)}_\bolda(\tau)\big) \big(\widetilde{E}^{(3)}_\boldb(\tau)-\tfrac12\widetilde{E}^{(3)}_\bolda(\tau)-\tfrac12\widetilde{E}^{(3)}_\boldc(\tau)\big)\,d\tau\Big),
\end{equation*}
see formula \eqref{eq dGab final}.
This holds in the domain where all the coordinates of $\bolda,\boldb,\boldc \in (\R/\Z)^2$ are non-zero, with $\bolda+\boldb+\boldc=\boldsymbol0$ as usual. Using Theorem \ref{thm RZ Iuv}, we have
\begin{align} \label{eq G 1}
\frac{\partial}{\partial a_2} \mathcal{G}(\bolda,\boldb) = 2\pi^2 & \mathcal{M}\Big((G^{(1)}_{b_1,b_2} G^{(2)}_{b_1,-b_2} - G^{(1)}_{b_1,-b_2} G^{(2)}_{b_1,b_2}) - \frac12 (G^{(1)}_{b_1,a_2} G^{(2)}_{a_1,-b_2} - G^{(1)}_{b_1,-a_2} G^{(2)}_{a_1,b_2}) \\
\nonumber & \qquad - \frac12 (G^{(1)}_{b_1,c_2} G^{(2)}_{c_1,-b_2} - G^{(1)}_{b_1,-c_2} G^{(2)}_{c_1,b_2}) - (G^{(1)}_{a_1,b_2} G^{(2)}_{b_1,-a_2} - G^{(1)}_{a_1,-b_2} G^{(2)}_{b_1,a_2}) \\
\nonumber & \qquad + \frac12 ( G^{(1)}_{a_1,a_2} G^{(2)}_{a_1,-a_2} - G^{(1)}_{a_1,-a_2} G^{(2)}_{a_1,a_2} ) + \frac12 ( G^{(1)}_{a_1,c_2} G^{(2)}_{c_1,-a_2} - G^{(1)}_{a_1,-c_2} G^{(2)}_{c_1,a_2} ), 0 \Big).
\end{align}
Let us write $f = f_1 + \dots + f_6$ for the modular form inside \eqref{eq G 1}. We rewrite $f$ using Theorems \ref{thm BG Gk 1} and \ref{thm BG Gk 2}. Theorem \ref{thm BG Gk 2} gives
\begin{equation*}
\mathcal{M}(f_1, 0) = \mathcal{M}(G^{(3)}_{0,b_2}, 0), \qquad
\mathcal{M}(f_5, 0) = \frac12 \mathcal{M}(G^{(3)}_{0,a_2}, 0).
\end{equation*}

Using Theorem \ref{thm BG Gk 1} with $x_1=c_1$, $y_1=a_1$, $u_2=a_2$ and $v_2=-c_2$, we have
\begin{equation} \label{eq G 3}
G^{(1)}_{-b_1,a_2} G^{(2)}_{a_1,b_2} + G^{(1)}_{a_1,-c_2} G^{(2)}_{c_1,a_2} - G^{(1)}_{-b_1,-c_2} G^{(2)}_{c_1,-b_2} = G^{(1)}_{a_1,b_2} G^{(2)}_{b_1,-a_2};
\end{equation}
and with $x_1=c_1$, $y_1=b_1$, $u_2=b_2$ and $v_2=-c_2$, we obtain
\begin{equation} \label{eq G 4}
G^{(1)}_{b_1,-c_2} G^{(2)}_{c_1,b_2} - G^{(1)}_{-a_1,-c_2} G^{(2)}_{c_1,-a_2} - G^{(1)}_{b_1,a_2} G^{(2)}_{-a_1,b_2} = G^{(1)}_{a_1,-b_2} G^{(2)}_{b_1,a_2}.
\end{equation}
Combining \eqref{eq G 3} and \eqref{eq G 4}, we have
\begin{equation*}
f_2 + f_3 + f_6 = - \frac12 \times \eqref{eq G 3} + \frac12 \times \eqref{eq G 4} = - \frac12 G^{(1)}_{a_1,b_2} G^{(2)}_{b_1,-a_2} + \frac12 G^{(1)}_{a_1,-b_2} G^{(2)}_{b_1,a_2} = \frac12 f_4.
\end{equation*}
Therefore,
\begin{align} \label{eq G 5}
\frac{\partial}{\partial a_2} \mathcal{G}(\bolda,\boldb) & = 2\pi^2 \mathcal{M} \Big(G^{(3)}_{0,b_2} + \frac12 G^{(3)}_{0,a_2} + \frac32 f_4, 0 \Big) \\
\nonumber & = -3\pi^2 \mathcal{M} (G^{(1)}_{a_1, b_2} G^{(2)}_{b_1, -a_2} - G^{(1)}_{a_1,-b_2} G^{(2)}_{b_1, a_2}, 0) + \pi^2 \mathcal{M}(G^{(3)}_{0,a_2} + 2 G^{(3)}_{0,b_2}, 0).
\end{align}

To find the $a_2$-antiderivative of the right-hand side of \eqref{eq G 5}, we use Lemma \ref{lem diff Gk}. We have formally
\begin{align} \label{eq G 6}
\mathcal{M}(G^{(1)}_{a_1,-b_2} G^{(2)}_{b_1, a_2}, 0) & = \int_0^\infty G^{(1)}_{a_1,-b_2}(iy) G^{(2)}_{b_1, a_2}(iy) \frac{dy}{y} \\
\nonumber & = -\frac{1}{2\pi} \int_0^\infty G^{(1)}_{a_1,-b_2}(iy) \frac{\partial}{\partial a_2} G^{(1)}_{b_1, a_2}(iy) \frac{dy}{y^2} \\
\nonumber & = -\frac{1}{2\pi} \frac{\partial}{\partial a_2} \mathcal{M}(G^{(1)}_{a_1,-b_2} G^{(1)}_{b_1, a_2}, -1).
\end{align}

\begin{lem} \label{lem Lambda G3}
For $x \in \R/\Z$, $x \neq 0$, we have $\mathcal{M}(G^{(3)}_{0,x},0) = -2\zeta'(-2) B_1(\{x\})$.
\end{lem}

\begin{proof}
Using \eqref{eq Mellin Gk}, we obtain
\begin{equation*}
\mathcal{M}(G^{(3)}_{0,x},0) = \lim_{s \to 0} \Gamma(s) \zeta(s-2) (\zeta(x,s)-\zeta(-x,s)) = \zeta'(-2)(\zeta(x,0)-\zeta(-x,0)).
\end{equation*}
We conclude using the evaluation $\zeta(x,0) = -B_1(\{x\})$ \cite[Proposition 9.3]{AIK14}.
\end{proof}

From \eqref{eq G 5}, \eqref{eq G 6} and Lemma \ref{lem Lambda G3}, we get
\begin{equation} \label{eq G 7}
\frac{\partial}{\partial a_2} \mathcal{G}(\bolda,\boldb) = -\frac{3\pi}{2} \frac{\partial}{\partial a_2} \mathcal{M}(G^{(1)}_{a_1, b_2} G^{(1)}_{b_1, -a_2} + G^{(1)}_{a_1,-b_2} G^{(1)}_{b_1, a_2}, -1) + \frac{\zeta(3)}{2} \big( B_1(a_2) + 2 B_1(b_2)\big).
\end{equation}
This identity holds in the domain
\begin{equation*}
D_0 = \{(\bolda, \boldb) : 0 < a_1, a_2, b_1, b_2 < 1, \; a_1+b_1 \neq 1, \; a_2 + b_2 \neq 1 \},
\end{equation*}
which has four connected components:
\begin{align*}
D_{++} &= \{a_1+b_1 > 1, \; a_2+b_2 > 1\}, & D_{+-} &= \{a_1+b_1 > 1, \; a_2+b_2 < 1\}, \\
D_{-+} &= \{a_1+b_1 < 1, \; a_2+b_2 > 1\}, & D_{--} &= \{a_1+b_1 < 1, \; a_2+b_2 < 1\}. 
\end{align*}
We can integrate \eqref{eq G 7} on each of these domains, with possibly different integration constants. So for $\square \in \{++, +-, -+, --\}$ and $(\bolda,\boldb) \in D_\square$, we have
\begin{align}
\nonumber \mathcal{G}(\bolda,\boldb) & = -\frac{3\pi}2 \mathcal{M}(G^{(1)}_{a_1, b_2} G^{(1)}_{b_1, -a_2} + G^{(1)}_{a_1, -b_2} G^{(1)}_{b_1, a_2}, -1) \\
\label{eq Gab 1} & \quad + \frac{\zeta(3)}{4} \big(B_2(a_2) + B_2(b_2) + 4 B_1(a_2) B_1(b_2) \big) + C_\square(\bolda,\boldb),
\end{align}
where $C_\square(\bolda,\boldb)$ does not depend on $a_2$. For convenience, write
\begin{equation*}
L(\bolda,\boldb) = -\frac{3\pi}2 \mathcal{M}(G^{(1)}_{a_1, b_2} G^{(1)}_{b_1, -a_2} + G^{(1)}_{a_1, -b_2} G^{(1)}_{b_1, a_2}, -1).
\end{equation*}
To get further, note that the symmetry $(\bolda,\boldb) \to (\boldb,\bolda)$ leaves stable the connected components $D_\square$. And we have
\begin{equation*}
\mathcal{G}(\bolda,\boldb) = \mathcal{G}(\boldb,\bolda) \qquad ( (\bolda,\boldb) \in D_\square),
\end{equation*}
which follows from the identity of cocycles $\tilde\xi(\bolda,\boldb) = \tilde\xi(\boldb,\bolda)$, or from the expression of $\mathcal{G}(\bolda,\boldb)$ in terms of triple modular values. Taking into account $L(\bolda,\boldb)=L(\boldb,\bolda)$, we see from \eqref{eq Gab 1} that $C_\square(\bolda,\boldb)$ is symmetric in $\bolda,\boldb$. Therefore $C_\square(\bolda,\boldb)$ does not depend on $b_2$ either, and we can write
\begin{equation*}
C_\square(\bolda,\boldb) = C'_\square(a_1, b_1).
\end{equation*}
(The function $C'_\square(\alpha,\beta)$ is defined either on the domain $\alpha+\beta>1$ or on the domain $\alpha+\beta<1$, depending on the first sign in $\square$.)

Now, let us use the matrix $\sigma = (\begin{smallmatrix} 0 & -1 \\ 1 & 0 \end{smallmatrix})$ acting as
\begin{equation*}
(\bolda,\boldb) = (a_1,a_2,b_1,b_2) \to (\bolda\sigma, \boldb\sigma) = (a_2,-a_1,b_2,-b_1).
\end{equation*}
It permutes the connected components of the domain by $D_{++} \to D_{+-} \to D_{--} \to D_{-+} \to D_{++}$. With the regulator, we have
\begin{align*}
\mathcal{G}(\bolda,\boldb) & = \int_0^\infty r_3(2)(\tilde\xi(\bolda,\boldb)) = \int_\infty^0 r_3(2)(\tilde\xi(\bolda,\boldb)) | \sigma \\
& = - \int_0^\infty r_3(2)(\tilde\xi(\bolda,\boldb) | \sigma) = - \int_0^\infty r_3(2)(\tilde\xi(\bolda\sigma,\boldb\sigma)) = - \mathcal{G}(\bolda\sigma,\boldb\sigma).
\end{align*}
One also checks that $L(\bolda,\boldb) = -L(\bolda\sigma, \boldb\sigma)$, using the identity of Eisenstein series $G^{(1)}_{x_1,x_2} = G^{(1)}_{x_2,x_1} = - G^{(1)}_{-x_1,-x_2}$. Therefore,
\begin{align*}
0 & = \mathcal{G}(\bolda,\boldb) + \mathcal{G}(\bolda\sigma,\boldb\sigma) \\
& = \frac{\zeta(3)}{4} \bigl(B_2(a_2) + B_2(b_2) + 4 B_1(a_2) B_1(b_2) \bigr) + C'_\square(a_1, b_1) \\
& \quad + \frac{\zeta(3)}{4} \bigl(B_2(a_1) + B_2(b_1) + 4 B_1(a_1) B_1(b_1) \bigr) + C'_{\sigma(\square)}(a_2, b_2).
\end{align*}
This identity can be rewritten as
\begin{align*}
& \frac{\zeta(3)}{4} \bigl(B_2(a_1) + B_2(b_1) + 4 B_1(a_1) B_1(b_1) \bigr) + C'_\square(a_1, b_1) \\
& \qquad = - \frac{\zeta(3)}{4} \bigl(B_2(a_2) + B_2(b_2) + 4 B_1(a_2) B_1(b_2) \bigr) - C'_{\sigma(\square)}(a_2, b_2).
\end{align*}
The left-hand side depends only on $a_1,b_1$, while the right-hand side depends only on $a_2,b_2$. Therefore, they do not depend on $(\bolda,\boldb)$ in $D_\square$ and we can write
\begin{align*}
C'_\square(a_1, b_1) & = - \frac{\zeta(3)}{4} \bigl(B_2(a_1) + B_2(b_1) + 4 B_1(a_1) B_1(b_1) \bigr) + C''_\square, \\
C'_{\sigma(\square)}(a_2, b_2) & = - \frac{\zeta(3)}{4} \bigl(B_2(a_2) + B_2(b_2) + 4 B_1(a_2) B_1(b_2) \bigr) - C''_\square.
\end{align*}
Reporting into \eqref{eq Gab 1} we have, for $(\bolda,\boldb) \in D_\square$,
\begin{align} \label{eq Gab C''}
\mathcal{G}(\bolda,\boldb) & = -\frac{3\pi}{2} \mathcal{M}(G^{(1)}_{a_1, b_2} G^{(1)}_{b_1, -a_2} + G^{(1)}_{a_1, -b_2} G^{(1)}_{b_1, a_2}, -1) \\
\nonumber & \quad - \frac{\zeta(3)}{4} \bigl(B_2(a_1) + B_2(b_1) + 4 B_1(a_1) B_1(b_1) \\
\nonumber & \qquad \qquad \quad - B_2(a_2) - B_2(b_2) - 4 B_1(a_2) B_1(b_2) \bigr) + C''_\square.
\end{align}
Substituting $(\bolda,\boldb) \to (\bolda\sigma,\boldb\sigma)$ in this relation, we see that $C''_{\sigma(\square)} = - C''_\square$ for any component $\square \in \{++, +-, -+, --\}$.

Finally, let us take $\bolda=\boldb$ in \eqref{eq Gab C''}. Since the cocycle $\tilde\xi(\bolda,\bolda)$ is zero, we have $\mathcal{G}(\bolda,\bolda)=0$. Specialising even further to $\bolda=\boldb=(\alpha,\alpha)$ with $\alpha \in (0,1)$, $\alpha \neq \frac12$, the $L$-value part in \eqref{eq Gab C''} vanishes since $G^{(1)}_{\alpha,-\alpha} = 0$. Moreover, the $\zeta(3)$ part also vanishes. It follows that $C''_{++} = C''_{--} = 0$, hence $C''_\square = 0$ for every $\square$. We have thus shown the following.

\begin{thm} \label{thm Gab final}
For any $\bolda,\boldb \in (\R/\Z)^2$ such that the coordinates of $\bolda$, $\boldb$ and $\bolda+\boldb$ are non-zero, we have
\begin{align} \label{eq Gab final}
\mathcal{G}(\bolda,\boldb) & = -\frac{3\pi}{2} \mathcal{M}(G^{(1)}_{a_1, b_2} G^{(1)}_{b_1, -a_2} + G^{(1)}_{a_1, -b_2} G^{(1)}_{b_1, a_2}, -1) \\
\nonumber & \quad - \frac{\zeta(3)}{4} \big(B_2(a_1) + B_2(b_1) + 4 B_1(a_1) B_1(b_1) \\
\nonumber & \qquad \qquad \quad - B_2(a_2) - B_2(b_2) - 4 B_1(a_2) B_1(b_2) \big).
\end{align}
\end{thm}

Theorem \ref{main thm 1} follows by specialising Theorem \ref{thm Gab final} to the case of $N$-torsion points. More precisely, using \eqref{eq Gk GkN}, we have the relation, for $\boldx,\boldy \in (\Z/N\Z)^2$, 
\begin{equation} \label{eq Mellin GkN}
    \mathcal{M}(G^{(1)}_{\boldx/N} G^{(1)}_{\boldy/N}, -1) = \frac{1}{N} \mathcal{M}(G^{(1);N}_\boldx G^{(1);N}_\boldy, -1) = -\frac{2\pi}{N} L'(G^{(1);N}_\boldx G^{(1);N}_\boldy, -1).
\end{equation}


\section{Relation to the Beilinson regulator} \label{B-reg}

In \cite[Conjecture 1.3]{Bru20}, the first author conjectured that the elements $\xi(\bolda, \boldb)$, $\bolda, \boldb \in (\Z/N\Z)^2$, are proportional to the Beilinson elements $\Eis^{0,0,1}(\bolda,\boldb)$ in $K_4^{(3)}(Y(N))$ \cite[Definition 2.3.6]{Wan20}. There is an explicit representative $\Eis^{0,0,1}_{\mathcal{D}}(\bolda,\boldb)$ of the Beilinson regulator of $\Eis^{0,0,1}(\bolda,\boldb)$ \cite[Proposition 2.4.2]{Wan20}. This is a differential $1$-form on $Y(N)(\C)$, and Weijia Wang proved the following explicit formula \cite[Example 6.1.5]{Wan20}, for $\bolda,\boldb \neq \boldsymbol 0$:
\begin{equation*}
\mathcal{B}(\bolda,\boldb) := \int_0^\infty \Eis^{0,0,1}_{\mathcal{D}}(\bolda,\boldb)
= \frac{9\pi}{N^3} \mathcal{M}(G^{(1);N}_{a_2,-b_1} G^{(1);N}_{b_2,a_1} + G^{(1);N}_{a_2,b_1} G^{(1);N}_{b_2,-a_1}, -1).
\end{equation*}
Using the relations $G^{(1);N}_{x_1,x_2} = G^{(1);N}_{x_2,x_1} = - G^{(1);N}_{-x_1,-x_2}$, as well as \eqref{eq Mellin GkN}, we obtain
\begin{equation*}
\mathcal{B}(\bolda,\boldb) = - \frac{9\pi}{N^2} \mathcal{M}(G^{(1)}_{a_1, b_2} G^{(1)}_{b_1, -a_2} + G^{(1)}_{a_1, -b_2} G^{(1)}_{b_1, a_2}, -1) \qquad (\bolda, \boldb \neq \boldsymbol 0),
\end{equation*}
where we identify $\Z/N\Z$ and $\frac{1}{N}\Z/\Z$. The modular form on the right-hand side matches with the one in \eqref{eq Gab final}, and from comparing the two expressions we deduce Theorem \ref{main thm 2}.

As explained in the introduction, Theorem \ref{main thm 2} gives evidence for the conjectural coincidence of the motivic cohomology classes $\xi(\bolda,\boldb)$ and $\pm \frac{N^2}{3} \Eis^{0,0,1}(\bolda,\boldb)$. Note a different from $\pm N^2/3$ factor $N^2/6$ in Theorem \ref{main thm 2}: this is due to the fact that the Beilinson regulator, which is used to define $\Eis^{0,0,1}_{\mathcal{D}}(\bolda,\boldb)$, is expected to be $\pm 2$ times the Goncharov regulator $r_3(2)$, via De Jeu's map. De Jeu has proved this compatibility for general curves under some assumptions \cite[Theorem 5.4]{Jeu00}; see the discussion in \cite[Section 6.1]{Bru20}.

\section{Conclusion}

One important application of Theorem~\ref{main thm 1} is to proving the longstanding conjecture of Boyd and Rodriguez Villegas on the Mahler measure \cite{BZ20} of the three-variable polynomial $P = (1+x)\*(1+y)+z$, namely $\mathrm m(P)=-2L'(E,-1)$, where $E$ is the elliptic curve over $\Q$ defined by the affine equation $(1+x)(1+y)(1+\frac{1}{x})(1+\frac{1}{y}) = 1$. To do this, the starting point is the work of Lal\'\i n \cite{Lal15} expressing this Mahler measure as a Goncharov regulator on the elliptic curve $E$:
\begin{equation} \label{eq mP regE}
\mathrm m((1+x)(1+y)+z)=\frac1{4\pi^2}\int_{\gamma_E} r_3(2)(\xi_E),
\end{equation}
where $\xi_E$ is a degree $2$ cohomology class in the weight $3$ Goncharov complex of $E$, and $\gamma_E$ is a generator of $H_1(E(\C),\Z)^+$, the subgroup of invariants under complex conjugation in the homology of $E$. What allows one to compute the regulator integral \eqref{eq mP regE} is that $E$ is actually isomorphic to the modular curve $X_1(15)$, and using this identification, the class $\xi_E$ has the simple expression $\xi_E = -20 \xi((0,1),(0,4))$. The details of this are given by the first author in \cite{Bru23}.

\medskip
Though our Theorems \ref{main thm 1} and \ref{main thm 2} do not cover the boundary cases, where some coordinates of $\bolda,\boldb,\bolda+\boldb\in(\Z/N\Z)^2$ are zero, they indicate some interesting behaviour when the parameters approach the boundary. The rational multiple of $\zeta(3)$ in Theorems \ref{main thm 1} and \ref{main thm 2} has discontinuities at the boundary due to the Bernoulli polynomial $B_1$, which may have to be replaced by the sawtooth wave or by regularised values as in Propositions \ref{pro single Eis} and \ref{pro K2 regulator}. It is also not clear, to begin with, whether the Goncharov regulator $\mathcal{G}(\bolda,\boldb)$ can be interpolated as a continuous function along the boundary. It would be interesting to gain a more conceptual understanding of these continuity issues; some numerical experiments may shed light on that.

\medskip
In essence, the explicit relation between the regulator integrals $\mathcal{G}(\bolda,\boldb)$ and $\mathcal{B}(\bolda,\boldb)$ should be enough to prove \cite[Conjecture 1.3]{Bru20} at the level of Deligne--Beilinson cohomology (as well as its more general version for the modular curve $Y(N)$). It would be interesting to develop a similar machinery in the $p$-adic setting, in terms of $p$-adic integration and $p$-adic regulators. At the motivic level, however, the conjecture looks more difficult and seems to require new ideas; a Hodge theoretic interpretation of the computations in this article would be already very interesting.


\begin{thebibliography}{99}

\bibitem{AIK14}
\textsc{T.~Arakawa}, \textsc{T.~Ibukiyama} and \textsc{M.~Kaneko},
\emph{Bernoulli Numbers and Zeta Functions}.
Springer Monogr. Math.,
Springer, Tokyo, 2014.

\bibitem{BK10}
\textsc{K.~Bannai} and \textsc{S.~Kobayashi},
Algebraic theta functions and the $p$-adic interpolation of Eisenstein--Kronecker numbers,
\emph{Duke Math. J.} \textbf{153} (2010), no. 2, 229--295.

\bibitem{Bei86}
\textsc{A.\,A.~Beilinson},
Higher regulators of modular curves,
in \emph{Applications of algebraic $K$-theory to algebraic geometry and number theory}, Part I, II (Boulder, Colo., 1983),
Contemp. Math. \textbf{55} (Amer. Math. Soc., Providence, RI, 1986), 1--34.

\bibitem{Boi26}
\textsc{M.~Boisan},
\emph{Relations linéaires entre valeurs modulaires multiples associées à des séries d'Eisenstein},
PhD thesis (\'Ecole normale sup\'erieure de Lyon, 2026); \url{https://theses.hal.science/tel-05677029v1}

\bibitem{BG03}
\textsc{L.\,A.~Borisov} and \textsc{P.\,E.~Gunnells},
Toric modular forms of higher weight,
\emph{J. Reine Angew. Math.} \textbf{560} (2003), 43--64.

\bibitem{Bro17}
\textsc{F.~Brown}, Multiple Modular Values and the relative completion of the fundamental group of $M_{1,1}$, \href{https://arxiv.org/abs/1407.5167v4}{arXiv:1407.5167v4}

\bibitem{Bro19}
\textsc{F.~Brown}, From the Deligne--Ihara conjecture to Multiple Modular Values, \href{https://arxiv.org/abs/1904.00179}{arXiv:1904.00179}

\bibitem{Bru16}
\textsc{F.~Brunault},
Regulators of Siegel units and applications,
\emph{J. Number Theory} \textbf{163} (2016), 542--569.

\bibitem{Bru17}
\textsc{F.~Brunault},
R\'egulateurs modulaires explicites via la m\'ethode de Rogers--Zudilin,
\emph{Compos. Math.} \textbf{153} (2017), 1119--1152.

\bibitem{Bru20}
\textsc{F.~Brunault},
On the $K_4$ group of modular curves, \href{https://arxiv.org/abs/2009.07614v3}{arXiv:2009.07614v3}

\bibitem{Bru23}
\textsc{F.~Brunault},
On the Mahler measure of $(1+x)(1+y)+z$, \href{https://arxiv.org/abs/2305.02992}{arXiv:2305.02992}

\bibitem{BZ20}
\textsc{F.~Brunault} and \textsc{W.~Zudilin},
\emph{Many Variations of Mahler Measures: A Lasting Symphony},
Australian Math. Soc. Lecture Ser. \textbf{28} (Cambridge University Press, Cambridge, 2020);
\href{https://doi.org/10.1017/9781108885553}{doi:10.1017/9781108885553}

\bibitem{Car92}
\textsc{P.~Cartier},
An introduction to zeta functions,
in: \emph{From number theory to physics} (Les Houches, 1989) (Springer, Berlin, 1992), pp.~1--63.


\bibitem{Gon02}
\textsc{A.~B.~Goncharov},
Explicit regulator maps on polylogarithmic motivic complexes,
in: \emph{Motives, polylogarithms and Hodge theory, Part I} (Irvine, CA, 1998),
Int. Press Lect. Ser. \textbf{3}, I (Int. Press, Somerville, MA, 2002), pp.~245--276.

\bibitem{Jeu95}
\textsc{R.~de Jeu},
Zagier's conjecture and wedge complexes in algebraic $K$-theory,
\emph{Compositio Math.} \textbf{96} (1995), no.~2, 197--247. 

\bibitem{Jeu96}
\textsc{R.~de Jeu},
On $K_4^{(3)}$ of curves over number fields,
\emph{Invent. Math.} \textbf{125} (1996), no.~3, 523--556. 

\bibitem{Jeu00}
\textsc{R.~de Jeu},
Towards regulator formulae for the $K$-theory of curves over number fields,
\emph{Compositio Math.} \textbf{124} (2000), no.~2, 137--194. 

\bibitem{Kat04}
\textsc{K.~Kato},
$p$-adic Hodge theory and values of zeta functions of modular forms,
\emph{Ast\'erisque} \textbf{295} (2004), 117--290.

\bibitem{Khu12}
\textsc{K.~Khuri-Makdisi},
Moduli interpretation of Eisenstein series, \emph{Int. J. Number Theory} \textbf{8} (2012), no. 3, 715--748.

\bibitem{KR17}
\textsc{K.~Khuri-Makdisi} and \textsc{W.~Raji},
Periods of modular forms and identities between Eisenstein series, \emph{Math. Ann.} \textbf{367} (2017), no. 1-2, 165--183. 

\bibitem{KZ01}
\textsc{M.~Kontsevich} and \textsc{D.~Zagier},
Periods,
in: \emph{Mathematics unlimited\,---\,2001 and beyond} (Springer, Berlin, 2001), pp.~771--808.

\bibitem{Lal15}
\textsc{M.\,N.~Lal\'\i n},
Mahler measure and elliptic curve $L$-functions at $s=3$,
\emph{J. Reine Angew. Math.} \textbf{709} (2015), 201--218.

\bibitem{Man05}
\textsc{Y.\,I.~Manin},
Iterated Shimura integrals,
\emph{Moscow Math. J.} \textbf{5} (2005), no.~4, 869--881.

\bibitem{Man06}
\textsc{Y.\,I.~Manin},
Iterated integrals of modular forms and noncommutative modular symbols,
in: \emph{Algebraic geometry and number theory},
Progr. Math. \textbf{253} (Birkh\"auser Boston, Boston, MA, 2006), pp.~565--597.


\bibitem{Sch74}
\textsc{B.~Schoeneberg}, \emph{Elliptic modular functions: an introduction}, Die Grundlehren der mathematischen Wissenschaften, vol.~203 (Springer-Verlag, New York--Heidelberg, 1974).


\bibitem{Wan20}
\textsc{W.~Wang},
\emph{Regularized integrals and $L$-functions of modular forms via the Rogers--Zudilin method},
PhD thesis (\'Ecole normale sup\'erieure de Lyon, 2020); \url{https://theses.hal.science/tel-02965542v1}

\bibitem{Wei76}
\textsc{A. Weil},
\emph{Elliptic functions according to Eisenstein and Kronecker},
Ergebnisse der Mathematik und ihrer Grenzgebiete, vol.~88 (Springer-Verlag, Berlin--New York, 1976).

\end{thebibliography}
\end{document}